\documentclass[english]{article}
\usepackage{lmodern}

\usepackage[T1]{fontenc}
\usepackage[latin9]{luainputenc}
\usepackage{geometry}
\usepackage{color}
\usepackage{babel}
\usepackage{refstyle}
\usepackage{amsmath}
\usepackage{amsthm}
\usepackage{amssymb}
\usepackage{stackrel}
\PassOptionsToPackage{normalem}{ulem}
\usepackage{ulem}
\usepackage[pdfusetitle,
 bookmarks=true,bookmarksnumbered=false,bookmarksopen=false,
 breaklinks=false,pdfborder={0 0 1},backref=false,colorlinks=true]
 {hyperref}
\hypersetup{
 pdfborderstyle=,linkcolor=black,citecolor=black,urlcolor=blue,filecolor=blue,pdfpagelayout=OneColumn,pdfnewwindow=true,pdfstartview=XYZ,plainpages=false}

\makeatletter

\AtBeginDocument{\providecommand\thmref[1]{\ref{thm:#1}}}
\AtBeginDocument{\providecommand\secref[1]{\ref{sec:#1}}}
\AtBeginDocument{\providecommand\Secref[1]{\ref{Sec:#1}}}
\AtBeginDocument{\providecommand\lemref[1]{\ref{lem:#1}}}
\AtBeginDocument{\providecommand\corref[1]{\ref{cor:#1}}}
\AtBeginDocument{\providecommand\Lemref[1]{\ref{Lem:#1}}}
\AtBeginDocument{\providecommand\Thmref[1]{\ref{Thm:#1}}}
\RS@ifundefined{subsecref}
  {\newref{subsec}{name = \RSsectxt}}
  {}
\RS@ifundefined{thmref}
  {\def\RSthmtxt{theorem~}\newref{thm}{name = \RSthmtxt}}
  {}
\RS@ifundefined{lemref}
  {\def\RSlemtxt{lemma~}\newref{lem}{name = \RSlemtxt}}
  {}

\theoremstyle{plain}
\newtheorem*{thm*}{\protect\theoremname}
\theoremstyle{plain}
\newtheorem{thm}{\protect\theoremname}[section]
\theoremstyle{plain}
\newtheorem{conjecture}[thm]{\protect\conjecturename}
\theoremstyle{definition}
\newtheorem{defn}[thm]{\protect\definitionname}
\theoremstyle{remark}
\newtheorem{rem}[thm]{\protect\remarkname}
\theoremstyle{plain}
\newtheorem{lem}[thm]{\protect\lemmaname}
\theoremstyle{definition}
\newtheorem{example}[thm]{\protect\examplename}
\theoremstyle{plain}
\newtheorem{cor}[thm]{\protect\corollaryname}

\usepackage{babel}
\AtBeginDocument{\providecommand\conjref[1]{\ref{conj:#1}}}\AtBeginDocument{\providecommand\Secref[1]{\ref{Sec:#1}}}\AtBeginDocument{\providecommand\defref[1]{\ref{def:#1}}}\AtBeginDocument{\providecommand\remref[1]{\ref{rem:#1}}}\AtBeginDocument{\providecommand\Defref[1]{\ref{Def:#1}}}\AtBeginDocument{\providecommand\Lemref[1]{\ref{Lem:#1}}}\AtBeginDocument{\providecommand\Thmref[1]{\ref{Thm:#1}}}\newref{def}{name=definition~,Name=Definition~,names=definitions~,Names=Definitions~}

\newref{thm}{name=theorem~,Name=Theorem~,names=theorems~,Names=Theorems~}

\newref{cor}{name=corollary~,Name=Corollary~,names=Corollaries~,Names=Corollaries~}

\newref{rem}{name=remark~,Name=Remark~,names=remarks~,Names=Remarks~}

\newref{conj}{name=conjecture~,Name=Conjecture~,names=conjectures~,Names=Conjectures~}

\newref{lem}{name=lemma~,Name=Lemma~,names=lemmas~,Names=Lemmas~}

\providecommand{\conjecturename}{Conjecture}
\providecommand{\corollaryname}{Corollary}
\providecommand{\definitionname}{Definition}
\providecommand{\examplename}{Example}
\providecommand{\lemmaname}{Lemma}
\providecommand{\remarkname}{Remark}
\providecommand{\theoremname}{Theorem}

\makeatother

\providecommand{\conjecturename}{Conjecture}
\providecommand{\corollaryname}{Corollary}
\providecommand{\definitionname}{Definition}
\providecommand{\examplename}{Example}
\providecommand{\lemmaname}{Lemma}
\providecommand{\remarkname}{Remark}
\providecommand{\theoremname}{Theorem}

\begin{document}
\title{Sentences over Random Groups I: Existential Sentences}
\author{Sobhi Massalha}
\maketitle
\begin{abstract}
Random groups of density $d<\frac{1}{2}$ are infinite hyperbolic,
and of density $d>\frac{1}{2}$ are finite. We prove that for any
given system of equations $\Sigma$, all the solutions of $\Sigma$
over a random group of density $d<\frac{1}{2}$ are projected from
solutions of $\Sigma$ over the free group $F_{k}$, with overwhelming
probability, where $k$ is the rank of the group. We conclude that
any given sentence in the Boolean algebra of universal sentences,
is a truth sentence over $F_{k}$ if and only if it is a truth sentence
over random groups of density $d<\frac{1}{2}$, with overwhelming
probability. 
\end{abstract}
\tableofcontents{}

\section{Introduction}

Around 1945, a well-known question was presented by Alfred Tarski
on the first order theory of free groups. He asked if every two non-abelian
finitely generated free groups are elementary equivalent. His question
was answered affirmatively by Z. Sela through his seminal work, which
was published in a series consisting of seven papers titled ``Diophantine
Geometry''. 
\begin{thm*}
(Z. Sela, \cite{DGI,DGII,DGIII,DGIV,DGV,DGV2,DGVI}) Every two non-abelian
free groups share the same first order theory. 
\end{thm*}
For a different approach to Tarski's problem, see also \cite{Kharlampovich=000020Myasnikov=000020-=000020Elementary=000020theory=000020of=000020free=000020non-abelian=000020groups}.

Actually, Sela obtained a classification of all the f.g. (finitely
generated) groups that are elementary equivalent to a non-abelian
f.g. free group.

Hence, all the non-abelian f.g. free groups share the same collection
of truth sentences. This collection is called the theory of free groups.
In this paper, and in the coming papers of this series, we will be
interested in the groups that can or cannot be distinguished from
the free groups by a single sentence. More precisely, we aim to prove
that ``almost all the groups'' cannot be distinguished from the
free groups by a single sentence.

In order to describe the mathematical meaning of the term ``almost
all the groups'', we need to fix a model of groups through which
one can pick a group at random. A model for picking a finitely presented
group at random was suggested by M. Gromov in \cite{Gromov=000020-=000020Asymptotic=000020invariants}.
This model has two fixed parameters $0\leq d\leq1$ and an integer
$k\geq2$. The value of $d$ is called the density of the model, and
$k$ is the rank of a fixed free group $F_{k}$ with a fixed basis
$a=(a_{1},...,a_{k})$. All the groups in this model are given by
presentations in terms of the generating set $a$. The number of relations
in a group of level $l$ in the density model, grows exponentially
with $l$. This is in contrast to the few relators model, also introduced
by Gromov in \cite{Gromov=000020Hyperbolic=000020groups}, where the
number of relations is fixed in advance.

Picking a group at random in the density model, enables us to consider
the probability of the satisfaction of a given property in that random
group. For example, in \cite{Gromov=000020-=000020Asymptotic=000020invariants}
Gromov proved that in all the densities $d<\frac{1}{2}$, a random
group is infinite hyperbolic in overwhelming probability (tends to
$1$). This fact gives a possible answer for the question ``What
does a generic group look like?''. The answer is ``It is hyperbolic
almost always''. It was proved in \cite{Gromov=000020-=000020Asymptotic=000020invariants}
that at densities $d>\frac{1}{2}$, the random groups are either trivial
or isomorphic to $\mathbb{Z}/2\mathbb{Z}$, with overwhelming probability.

In 2013, J. Knight raised the following natural question: for any
given first order sentence $\psi$, the sentence $\psi$ is a truth
sentence over free groups if and only if $\psi$ is a truth sentence
over random groups (in overwhelming probability) in the few relators
model. In \cite{Kharlampovich=000020Sklinos=000020-=000020FIRST-ORDER=000020SENTENCES},
O. Kharlampovich and R. Sklinos stated a similar conjecture, but for
random groups in the density model with densities $d<\frac{1}{16}$.
We extend these conjectures to the following more general conjecture. 
\begin{conjecture}
\label{conj:1} Let $0\leq d<\frac{1}{2}$ and let $\psi$ be a first
order sentence in the language of groups. The random group of density
$d$ satisfies $\psi$ (in overwhelming probability) if and only if
the free group $F_{k}$ satisfies $\psi$. 
\end{conjecture}

In this paper, we prove \conjref{1} for sentences that belong to
the Boolean algebra of universal sentences. In the next papers we
consider general sentences. 
\begin{thm}
\label{thm:2} Let $0\leq d<\frac{1}{2}$. Let $\psi$ be a first
order sentence in the language of groups, and assume that $\psi$
belongs to the Boolean algebra of universal sentences. Then the random
group of density $d$ satisfies $\psi$ (in overwhelming probability)
if and only if the free group $F_{k}$ satisfies $\psi$. 
\end{thm}

The statement of \thmref{2} was proved for densities $d<\frac{1}{16}$
in \cite{Kharlampovich=000020Sklinos=000020-=000020FIRST-ORDER=000020SENTENCES},
using a different argument than the one we present in this paper.

In \cite{DGI,DGII,DGIII,DGIV,DGV,DGV2,DGVI}, Sela introduced new
mathematical tools, that can be applied on a given sentence $\psi$
over a non-abelian f.g. free group. Roughly speaking, when a sentence
$\psi$ is a truth sentence over some non-abelian f.g. free group,
these tools give a syntactic ``proof'' for the correctness of that
sentence. These tools, together with the ``proof'' that is ``written''
by them, are shared uniformly by the various non-abelian f.g. free
groups. In particular, the ``proof'' of $\psi$ obtained over a
single non-abelian f.g. free group can be used as a ``proof'' for
the correctness of $\psi$ over all the other f.g. free groups.

Our strategy is proving that all of these tools, developed over free
groups, can be applied over a random group without changing them or
their functionality.

The most basic among these tools is the Makanin-Razborov diagram,
which was introduced in \cite{DGI}. The Makanin-Razborov diagram
of a given system of equations over a free group is a finite diagram
that describes the structure of the set of solutions of that system.
The development of the more advanced tools introduced by Sela, rely
heavily on the Makanin-Razborov diagram.

Hence, it is natural that our initial step towards the proof of the
conjecture is showing that the Makanin-Razborov diagram of a given
system of equations can be used uniformly for describing the sets
of solutions of that system over random groups.

In order to do so, we will prove the following statement. 
\begin{thm}
\label{thm:3} (Main Theorem) Let $k\geq2$ be an integer, let $d<\frac{1}{2}$
be a real number, and consider the density model with $k$ generators
and density $d$. Let $\Sigma_{0}$ be a system of equations. Then
the random presentation $\pi_{\Gamma}:F_{k}\rightarrow\Gamma$ of
density $d$ satisfies the following property with overwhelming probability.

The set of solutions of the system $\Sigma_{0}$ in $\Gamma$ equals
the projection to $\Gamma$ of the set of solutions of $\Sigma_{0}$
in $F_{k}$. 
\end{thm}

The majority of the work in this paper is devoted to the proof of
\thmref{3}. For doing that, we rely partially on some of the techniques
introduced originally in \cite{Gromov=000020-=000020Asymptotic=000020invariants}
and stated formally in \cite{Ollivier=000020-=000020Sharp=000020phase}.
These techniques lead us to a better understanding for the hyperbolic
nature of random groups, an understanding that leads to an application
of a shortening procedure. The shortening procedure was introduced
originally in \cite{Structure=000020and=000020rigidity=000020in=000020hyperbolic=000020groups=000020I},
and used in \cite{DGI} mainly for constructing the Makanin-Razborov
diagram of a given system of equations over free groups.

The paper is organized as follows.

In \secref{The-Gromov-Density-Model}, we recall the definition of
the Gromov density model.

In \secref{Van-Kampen-Diagrams}, we recall the definition of a Van-Kampen
diagram and a decorated Van-Kampen diagram. Decorated Van-Kampen diagrams
appeared originally in \cite{Ollivier=000020-=000020Sharp=000020phase},
and were used in \cite{Kharlampovich=000020Sklinos=000020-=000020FIRST-ORDER=000020SENTENCES}.
We need to handle problems regarding lengths in the Cayley graphs
of Random groups, rather than triviality problems. Therefore, we found
it necessary to extend this object. Then, in \secref{Van-Kampen-Diagrams},
we continue by introducing the procedure of ``generally reducing''
a decorated Van-Kampen diagram. This procedure will be necessary for
the purposes of bounding probabilities of the satisfaction of equations
over random groups.

In \secref{Van-Kampen-Diagrams}, we also introduce ``mining modification'',
which will be applied on decorated diagrams. The purpose of this modification
is to enable us to handle variables which are supposed to get values
of small lengths, and that occur in equations over random groups.
Also in \secref{Van-Kampen-Diagrams} we further estimate bounds for
the probabilities of a random group to fulfill a decorated diagram.

In \secref{Isoperimetric-Inequality-in-Random-Groups}, we recall
some facts from \cite{Ollivier=000020-=000020Sharp=000020phase},
regarding the hyperbolicity of random groups, and we bound the number
of decorated Van-Kampen diagrams in terms of appropriate parameters.

In \secref{Axes-of-Elements-in-Free-Groups}, we prove some facts
that hold in free groups, and that are related to commutativity in
free groups. We will count on these facts in order to be able to handle
diagrams with no cells over the random group (i.e., trees).

In \secref{Lifting-Solutions-of-Bounded-Lengths}, we prove that there
are no new solutions of bounded lengths of a given system over random
groups, i.e., all such solutions are projected from solutions in free
groups.

In \secref{Axes-of-Elements-in-Random-Groups}, we start to handle
lengths problems in random groups. In particular, we bound the ``time''
for which the axes of two non-commutative elements could travel together
(in close distance). This bound is interpreted in terms of the level
of the random group under consideration (and the density parameter
$d$). This bound will play an essential role during the application
of the shortening procedure in the coming sections. \Secref{Axes-of-Elements-in-Random-Groups}
is the most technically involved section in this paper.

In \secref{Hyperbolic-Geometry-over-Random-Groups}, we use the facts
proved in the previous sections in order to understand the hyperbolic
nature of a random group.

In \secref{Limit-Groups-over-Ascending-Sequences-of-Random-Groups},
we apply a shortening procedure on sequences of solutions of systems
of equations over random groups. As stated above, the shortening procedure
was used in \cite{DGI} in order to construct the Makanin-Razborov
diagram of a given system over free groups. Here, in \secref{Limit-Groups-over-Ascending-Sequences-of-Random-Groups},
our goal in applying the shortening procedure is proving the non-existence
of ``non-lift sequences'' over random groups, i.e., the non-existence
of sequences of solutions that are not projected from free groups.

Finally, in \secref{Consequences}, we deduce \thmref{2}, and that
the Makanin-Razborov diagram of a given system of equations over free
groups, encodes all the solutions of that system over a random group.
\\
 \\
 \textbf{Acknowledgment.} I am indebted to my advisor Zlil Sela who
introduced me to this problem. He shared his knowledge and ideas with
me, and without his input, I couldn't have completed this work.

\section{The Gromov Density Model}\label{sec:The-Gromov-Density-Model}

We fix an integer $k\geq2$, a free group $F_{k}$, and a basis $a=(a_{1},...,a_{k})$
of $F_{k}$. Given a property $P$, and a tuple $T$ of words in $F_{k}$,
we may consider the following claim: the group presentation $\langle a:T\rangle$
satisfies the property $P$. For example, we may choose the property
$P$ to be ``the group is Hopfian''. Then, given a tuple $T$ of
words in $F_{k}$, our claim becomes: the group $\langle a:T\rangle$
is Hopfian. The property $P$ may also be a property of a presentation,
such as ``the group presentation is small cancellation $C'(1/6)$''.
Then, given a tuple $T$ of words in $F_{k}$, the claim will be:
``the presentation $\langle a:T\rangle$ is $C'(1/6)$''. Since
every presentation with generators $a$ has a natural interpretation
of the elements of the free group $F_{k}$, the claim $P$ can contain
constants too. For example, $P$ can be the claim ``the word $a_{1}a_{2}a_{1}^{-1}a_{3}^{2}$
is trivial''.

Suppose further that the tuples that we consider are ordered in levels,
i.e., for each integer $l$, we are given a collection $\mathcal{M}_{l}$
of tuples consisting of words in the free group $F_{k}$. Then, the
description of the claim $P$ could also include the parameter $l$,
such as the claim ``the group is $2l$-hyperbolic''.

The group presentations defined by the tuples in the collection $\mathcal{M}_{l}$
are called the groups of level $l$. Once we define the groups of
level $l$, for all integers $l$, we obtain a probability model of
groups.

The way that a probability model of groups is used in general, is
as follows. We fix a claim $P$ and a collection of tuples ordered
in levels $\mathcal{M}_{l}$. And then, for every integer $l$, we
consider the following number: 
\[
p_{l}:=\frac{|\{T\in\mathcal{M}_{l}:\text{The group presentation }\langle a:T\rangle\text{ satisfies the property }P\}|}{|\mathcal{M}_{l}|}\,.
\]
This number is called the probability that the random group of level
$l$ (in the relevant model) satisfies the property $P$. Then, our
interest will be in the asymptotic behavior of $p_{l}$. I.e., we
check if the limit $\underset{l\rightarrow\infty}{\lim}p_{l}$ exists.
If the limit 
\[
p_{0}=\underset{l\rightarrow\infty}{\lim}p_{l}
\]
do exist, we say that the probability that the random group (in the
relevant model) satisfies the property $P$ is the real number $p_{0}$.
If in addition we have that $p_{0}=1$, we say then that the random
group (in the relevant model) satisfies the property $P$ in overwhelming
probability.

In this paper, we are interested in a specific model, which was introduced
in \cite{Gromov=000020-=000020Asymptotic=000020invariants} by M.
Gromov. 
\begin{defn}
\label{def:4}(Gromov density model). Fix a natural number $k\geq2$,
fix a free group $F_{k}$ with a fixed basis $a=(a_{1},...,a_{k})$,
and fix a real number $0\leq d\leq1$. The number $d$ is called the
\emph{density} of the model. Given an integer $l$, we denote by $B_{l}$
the set of cyclically reduced words of length $l$ in $F_{k}$

\[
B_{l}:=\{r\in F_{k}:|r|=l,\ r\text{ is cyclically reduced}\}\,.
\]
We also define the set $\mathcal{M}_{l}$, called \emph{the set of
group presentations of level $l$}, to be the collection of all the
tuples $T=\left(r_{1},...,r_{s(l)}\right)$ consisting of $s(l)=|B_{l}|^{d}$
elements $r_{i}\in B_{l}$, $i=1,...,\left[s(l)\right]$.

Given a property $P$, we call the number 
\[
p_{l}:=\frac{|\{T\in\mathcal{M}_{l}:\text{The group presentation }\langle a:T\rangle\text{ satisfies the property }P\}|}{|\mathcal{M}_{l}|}\,,
\]
the probability that \emph{the random group of level $l$ and density
$d$} satisfies the property $P$. If the limit $p_{0}=\underset{l\rightarrow\infty}{\lim}p_{l}$
exists, then we call it \emph{the probability that the random group
of density $d$ satisfies the property $P$}.

If $p_{0}=1$, we say that the random group of density $d$ satisfies
the property $P$ with \emph{overwhelming probability}. In this case
we say also that \emph{almost all the groups of density $d$ satisfy
the property $P$}. If $p_{0}=0$ in contrast, we say then that the
property $P$ is \emph{negligible in density $d$}.

Suppose that $\mathcal{N}=\underset{l}{\cup}\mathcal{N}_{l}$ is a
given collection of tuples. If the tuple defining the group (presentation)
$\Gamma$ does not belong to the collection $\mathcal{N}$, for almost
all the groups $\Gamma$ of density $d$, we say then that \emph{the
collection $\mathcal{N}$ is negligible in density $d$}.

Note that if the probability that the random group of density $d$
satisfies the property $P$ is $p_{0}$, then the probability that
the random group of density $d$ satisfies the property $\neg P$
is $1-p_{0}$. In particular, the random group of density $d$ satisfies
the property $P$ if and only if the property $\neg P$ is negligible
in density $d$. 
\end{defn}

As the reader may note, \defref{4} defines a family of models rather
than a single one. For each value $0\leq d\leq1$ of $d$, there is
a model defined. Actually, random groups in those various models may
behave differently, depending on the fixed density, even for properties
that are naturally brought up in the context of groups. For example,
we have the following facts. 
\begin{thm}
(\cite{Ollivier=000020-=000020Some=000020small=000020cancellation}) 
\begin{enumerate}
\item For $d<\frac{1}{5}$, almost all the presentations of density $d$
are Dehn presentations. 
\item For $d>\frac{1}{5}$, almost all the presentations of density $d$
are not Dehn presentations. 
\end{enumerate}
\end{thm}

In fact, for densities $d>\frac{1}{2}$, the density models are trivial. 
\begin{thm}
(\cite{Gromov=000020-=000020Asymptotic=000020invariants}) For $d>\frac{1}{2}$,
almost all the groups of density $d$ are either trivial or isomorphic
to $\mathbb{Z}\slash2\mathbb{Z}$. 
\end{thm}

\section{Van-Kampen Diagrams }\label{sec:Van-Kampen-Diagrams}

\subsection{Van-Kampen Diagrams over Group Presentations}

Let $l\geq1$ be an integer. We want to define Van-Kampen diagrams
over groups (presentations) with relations each of length $l$.

First, we start by defining topological Van-Kampen diagrams. 
\begin{defn}
A \emph{topological Van-Kampen diagram} (of level $l$), usually denoted
by $TD$, is a planar graph whose bounded faces are tiled by cells
(as disks). The \emph{boundary of each cell} in the diagram $TD$
consists of exactly $l$ edges. The boundary of each cell admits a
\emph{starting vertex} and a \emph{direction} (between two: clockwise
or counterclockwise), and the boundary of the whole diagram $TD$
admits a \emph{starting vertex} and a \emph{direction}.

An equivalence relation is defined between the cells, whose classes
are called \emph{numberings}, and are linearly ordered. In particular,
different cells may admit the same numbering. The \emph{direction
of an edge $e$ with respect to the cell $c$}, where $e$ belongs
to the boundary of $c$, is defined to be the direction of the cell
$c$, and \emph{the order of $e$ in $c$} is the natural one induced
by the starting point and the direction of $c$. An edge $e$ in the
diagram $TD$ is called a \emph{filament}, if $e$ does not belong
to (the boundary of) a cell.

A \emph{topological Van-Kampen diagram without numbering} is to forget
the numberings of the cells in the topological Van-Kampen diagram
(i.e., to forget the equivalence relation defined on the cells).

For every topological Van-Kampen diagram $TD$ we associate an \emph{auxiliary
graph} $K$, defined as follows. Assume that the cells of $TD$ admit
exactly $n$ different numberings, which we denote $1,...,n$. For
each $i=1,...,n$, we introduce $l$ different vertices, that correspond
to the (directed) edges of a (single copy of a) cell numbered $i$
in $TD$. The matchings between the edges of the various cells inside
the diagram $TD$, naturally suggests the following way for connecting
the vertices in the graph $K$. Let $1\leq i,j\leq n$ be two numberings,
and let $v_{i},v_{j}$ be two vertices associated to the numberings
$i,j$ in $K$ respectively. If an edge $e_{i}$ represented by the
vertex $v_{i}$ in some cell $c_{i}$ numbered $i$ inside the diagram
$TD$ equals the edge $e_{j}$ represented by the vertex $v_{j}$
in some cell $c_{j}$ numbered $j$ in $TD$, then, we connect $v_{i}$
to $v_{j}$. If further the direction of $e_{i}$ w.r.t. $c_{i}$
equals the direction of $e_{j}$ w.r.t. $c_{j}$, we label the edge
between $v_{i}$ and $v_{j}$ by $-1$, and by $1$ otherwise.

The topological Van-Kampen diagram $TD$ is called \emph{reduced},
if there exist no loops (edges between a vertex and itself) in the
associated graph $K$. Hence, by iteratively eliminating pairs of
cells $c,c'$ with the same numbering that share a common edge $e$
of the same order in $c$ and in $c'$ (and identifying pairs of edges
in the boundary of the constructed hole accordingly), every non-reduced
topological Van-Kampen diagram $TD$ can be converted to a reduced
one while reducing the number of cells, and without changing the boundary
(length) of the diagram. We call the procedure of eliminating all
such pairs of cells from the diagram \emph{standard reduction} of
the diagram.

Once we assign for each (directed) edge in a topological Van-Kampen
diagram $TD$ a letter in the alphabet $a^{\pm}$, the boundary of
each cell in $TD$ will be given some word in $F_{k}$, called the
\emph{contour of that cell}. Similarly, the boundary of the whole
diagram will be given a word in $F_{k}$, called \emph{the contour
of the diagram}. 
\end{defn}

\begin{defn}
Let $\Gamma=\langle a:\mathcal{R}\rangle$ be a group (presentation)
whose relations are all (reduced words) of length $l$. A \emph{Van-Kampen
diagram over $\Gamma$}, denoted usually by $VKD$, is a topological
Van-Kampen diagram $TD$ without numbering, together with assigning
to each edge in $TD$ a letter in the alphabet $a^{\pm}$, so that
the contour of every cell is a relator of $\Gamma$. If there exists
no relator $r$ of $\Gamma$ so that $r$ is the contour of two distinct
cells $c,c'$ in $TD$ that share an edge $e$ so that $e$ has the
same orders in $c,c'$ (but the direction of $e$ in $c$ differs
from the direction of $e$ in $c'$), then, we say that the Van-Kampen
diagram $VKD$ is \emph{reduced}.

If $TD$ is a topological Van-Kampen diagram without numbering, then,
we say that \emph{the group (presentation) $\Gamma$ fulfills the
topological diagram without numbering $TD$} if one can assign to
each edge in $TD$ a letter in the alphabet $a^{\pm}$, so that the
obtained diagram is a Van-Kampen diagram over $\Gamma$.

Let $TD$ be a topological Van-Kampen diagram (with numbering), that
admits $n$ distinct numberings. Let $t=(t_{i})$ be an $n$-tuple
of reduced words in $F_{k}$. We say that \emph{the $n$-tuple $t$
fulfills the topological diagram $TD$} if one can assign to each
edge in $TD$ a letter in the alphabet $a^{\pm}$, so that for every
numbering $i=1,...,n$, the contours of the cells numbered $i$ in
$TD$ are all equal to the word $t_{i}$. We say that \emph{the group
(presentation) $\Gamma$ fulfills the topological diagram $TD$} if
there exists an $n$-tuple $t$ consisting of \uline{pairwise distinct}
relators of $\Gamma$, so that $t$ satisfies $TD$. 
\end{defn}

\begin{rem}
Let $TD$ be a topological Van-Kampen diagram, and consider its associated
graph $K$. 
\begin{enumerate}
\item If $n$ is the amount of distinct numberings in $TD$, then, the group
$\Gamma$ fulfills $TD$, if and only if $\Gamma$ admits an $n$-tuple
of \uline{pairwise distinct} relators that satisfy the restrictions
defined by $K$. 
\item If the group $\Gamma$ fulfills $TD$, then the obtained Van-Kampen
diagram $VKD$ is reduced if and only if the topological diagram $TD$
is reduced. 
\end{enumerate}
\end{rem}

A more general notion than topological Van-Kampen diagram, is a decorated
topological Van-Kampen diagram. 
\begin{defn}
A \emph{decorated topological Van-Kampen diagram} (of level $l$),
usually denoted by $DTD$, is a topological Van-Kampen diagram (of
level $l$), who is equipped with the following extra structure. The
directed path (cycle) that represents the boundary of the diagram
$DTD$ is partitioned into consecutive parts (i.e., they can pairwise
intersect in at most two points and they cover all the boundary cycle),
that are called \emph{the parts of the boundary of $DTD$}. An equivalence
relation is defined on these parts, and each part is given a \emph{direction}
(between two directions). Each class in this equivalence relation
is called a \emph{variable}, and the equivalence relation together
with the parts of the boundary and their directions are called \emph{the
decoration of $DTD$}. If two parts represent the same variable, then
they are required to be of the same length (number of edges). If $e$
is an edge that belongs to some part $p$ of the boundary, then, \emph{the
direction of $e$ w.r.t. $p$} is defined to be the direction of $p$,
and \emph{the order of $e$ in $p$} is the natural one induced by
the starting point and the direction of $p$.

Informally, a variable $x$ can be written along a part of the boundary,
and $x^{-1}$ (the same variable but the other direction) could be
written on another part of the boundary of the same length as the
previous one. A decoration of the boundary could be explained briefly
by a word in the variables, e.g. $xyx^{-1}y^{-1}$.

We declare a subset (maybe empty) of the variables to be \emph{rigid},
and another subset (maybe empty) to be \emph{constant}. Every (directed)
edge that lie in a part of the boundary that represents a constant
variable, is given a labeling in advance, i.e., a letter from the
alphabet $a^{\pm}$. Moreover, as a part of the definition, we require
that each rigid variable contains only one element (only one part
of the boundary represents that rigid variable). If $e$ is an edge
in the diagram $DTD$, then $e$ is said to be a \emph{rigid edge}
if $e$ belongs to a rigid part of the boundary, and $e$ is called
a \emph{constant edge} if it belongs to a constant part of the boundary.
A cell $c$ in the diagram $DTD$ is called an \emph{isolated cell},
if it admits a rigid edge. A numbering $i$ of a cell in $DTD$ is
called an \emph{isolated numbering}, if all the cells with numbering
$i$ in the diagram are isolated cells.

A \emph{decorated topological Van-Kampen diagram without numbering}
is to forget the numberings of the cells in the decorated topological
Van-Kampen diagram.

A decorated topological Van-Kampen diagram is called \emph{reduced},
if the underlying topological Van-Kampen diagram is reduced.

For every decorated topological Van-Kampen diagram $DTD$ we associate
an \emph{auxiliary graph} $K$. This graph is initialized to be the
graph associated to the underlying topological Van-Kampen diagram.
However, the structure of the decorated diagram $DTD$ (the existence
of the decoration and possible cancellations between the parts along
the boundary - filaments) implies a natural generalization of the
standard notion of ``contiguity'' (matching) between edges in a
topological Van-Kampen diagram.

This contiguity is defined by declaring two edges $e_{1},e_{2}$ in
the diagram to be contiguous if $e_{1},e_{2}$ belong to some parts
$P_{1}x,P_{2}x$ (respectively) of the boundary of $DTD$, so that
the orders of $e_{1},e_{2}$ in $P_{1}x,P_{2}x$ (respectively) are
equal, and $P_{1}x,P_{2}x$ represent the same variable $x$ in the
decoration. The obtained minimal equivalence relation on the set of
edges in $DTD$ is called \emph{the contiguity between edges in the
decorated diagram $DTD$}. Edges in the same contiguity class are
called\emph{ contiguous edges}.

Motivated by this notion of contiguity, we extend the graph $K$ as
follows. Let $1\leq i,j\leq n$ be two numberings, and let $v_{i},v_{j}$
be two vertices associated to the numberings $i,j$ in $K$ respectively.
If an edge $e_{i}$ represented by the vertex $v_{i}$ in some cell
$c_{i}$ numbered $i$ inside the diagram $DTD$ is contiguous to
an edge $e_{j}$ represented by the vertex $v_{j}$ in some cell $c_{j}$
numbered $j$ in $DTD$, then, we connect $v_{i}$ to $v_{j}$. If
further the direction of $e_{i}$ w.r.t. $c_{i}$ equals the direction
of $e_{j}$ w.r.t. $c_{j}$, we label the edge between $v_{i}$ and
$v_{j}$ by $-1$, and by $1$ otherwise.

A connected component of the graph $K$ that contains a vertex that
represents some constant edge, is called a \emph{constant component}.
A vertex $v$ of $K$ is called a \emph{constant vertex}, if it belongs
to a constant component of $K$, and $v$ is called an \emph{isolated
vertex}, if $\{v\}$ is a non-constant connected component of $K$.
An edge $e$ in a decorated topological Van-Kampen diagram $DTD$
is called a \emph{cusp edge}, if $e$ is a non-rigid filament edge,
so that $\{e\}$ is a contiguity class. 
\end{defn}

\begin{rem}
$ $ 
\begin{enumerate}
\item Note that the only edge in the diagram $DTD$ that is contiguous to
a non-filament rigid edge $e$, is $e$ itself. 
\item A cusp edge could occur only when a variable is represented by exactly
two parts on the boundary, and these parts meet in ``dual'' edges. 
\end{enumerate}
\end{rem}

\begin{defn}
Let $\Gamma=\langle a:\mathcal{R}\rangle$ be a group (presentation)
whose relations are all (reduced words) of length $l$.

Let $DTD$ be a decorated topological Van-Kampen diagram without numbering.
We say that \emph{the group (presentation) $\Gamma$ fulfills the
decorated topological diagram without numbering $DTD$}, if one can
assign to each edge in $DTD$ a letter in the alphabet $a^{\pm}$,
so that the obtained diagram is a Van-Kampen diagram over $\Gamma$,
and so that the decoration is satisfied, i.e., for every part $p$
of the boundary, the word that is given to the part $p$ (according
to the assignment), equals (graphical equality - not only in $\Gamma$)
to the word given to any other equivalent part of the boundary.

Let $DTD$ be a decorated topological Van-Kampen diagram (with numbering),
that admits $n$ distinct numberings. Let $t=(t_{i})$ be an $n$-tuple
of reduced words in $F_{k}$. We say that \emph{the $n$-tuple $t$
fulfills the decorated topological diagram $DTD$}, if one can assign
to each edge in $DTD$ a letter in the alphabet $a^{\pm}$, so that
for every numbering $i=1,...,n$, the contours of the cells numbered
$i$ are all equal to the word $t_{i}$, and so that the decoration
is satisfied. We say that \emph{the group (presentation) $\Gamma$
fulfills the decorated topological diagram $DTD$}, if there exists
an $n$-tuple $t$ consisting of \uline{pairwise distinct} relators
of $\Gamma$, so that $t$ fulfills $DTD$. 
\end{defn}

\begin{rem}
Let $DTD$ a decorated topological Van-Kampen diagram, and consider
its associated graph $K$. Denote by $n$ the amount of distinct numberings
in $DTD$. Then, the group $\Gamma$ fulfills $DTD$, if and only
if, $\Gamma$ admits an $n$-tuple of \uline{pairwise distinct} relators
that satisfy the restrictions defined by $K$. 
\end{rem}

\subsection{Topological Van-Kampen Diagrams over Random Groups}
\begin{lem}
\label{lem:5} Let $DTD$ be a decorated topological Van-Kampen diagram,
that admits $n$ distinct numberings. Let $K$ be the graph associated
to the decorated diagram $DTD$, and let $u$ be the amount of distinct
non-constant connected components in $K$. Denote by $B_{l}$ the
set of cyclically reduced words of length $l$ in $F_{k}$. Then,
the probability that a random $n$-tuple of words in $B_{l}$ fulfills
the diagram $DTD$ is at most

\[
\frac{(2k)^{n}\cdot(2k-1)^{u}}{B_{l}^{n}}\,.
\]
\end{lem}

\begin{proof}
To construct an $n$-tuple of cyclically reduced words that fulfills
$DTD$, we start by specifying a label from $a^{\pm}$ for the first
edge in every cell, and then we choose one label from $a^{\pm}$ for
each non-constant connected component in the auxiliary graph $K$
associated to $DTD$.

Hence, we deduce that there exist at most $(2k)^{n}\cdot(2k-1)^{u}$
distinct $n$-tuples of cyclically reduced words of length $l$ that
fulfill $DTD$.

Then, the probability that a random $n$-tuple of cyclically reduced
words fulfills $DTD$ is at most

\[
\frac{(2k)^{n}\cdot(2k-1)^{u}}{B_{l}^{n}}\,.
\]
\end{proof}
\begin{lem}
\label{lem:6} Let $DTD$ be a decorated topological Van-Kampen diagram,
that admits $n$ distinct numberings. Denote by $B_{l}$ the set of
cyclically reduced words of length $l$ in $F_{k}$. Then, the probability
that the random group $\Gamma$ of level $l$ and density $d$, fulfills
the decorated diagram $DTD$, is at most

\[
B_{l}^{nd}\cdot p\,,
\]
where $p$ is the probability that a random $n$-tuple of words in
$B_{l}$, fulfills the decorated diagram $DTD$.

Hence, the probability that the random group $\Gamma$ of level $l$
and density $d$, fulfills the decorated diagram $DTD$ is at most

\[
\frac{(2k)^{n}\cdot(2k-1)^{u}}{B_{l}^{n\cdot(1-d)}}\,,
\]
where $u$ is given in \lemref{5}. 
\end{lem}

\begin{proof}
Recall that the random group $\Gamma$ of level $l$ and density $d$,
is defined by choosing in random a $B_{l}^{d}$-tuple of cyclically
reduced words of length $l$, and note that, by definition, the group
$\Gamma$ fulfills $DTD$ if and only if $\Gamma$ admits an $n$-tuple
of \uline{pairwise distinct} relators that fulfills the diagram $DTD$.

Hence, the probability that the random group of level $l$ and density
$d$ fulfills $DTD$, is at most the probability that when choosing
in random a $B_{l}^{d}$ words $w=(w_{i})$ from $B_{l}$, picked
one after the other, one of the $n$-tuples $(w_{i_{1}},...,w_{i_{n}})$,
where $1\leq i_{1},...,i_{n}\leq B_{l}^{d}$, fulfills $DTD$.

Hence, the probability that the random group of level $l$ and density
$d$ fulfills $DTD$, is bounded by 
\[
T_{n}\cdot p\,,
\]
where $T_{n}$ is the amount of distinct $n$-tuples consisting of
elements in the sequence $1,2,...,B_{l}^{d}$. That is, the probability
that the random group of level $l$ and density $d$ fulfills $DTD$,
is bounded by 
\[
B_{l}^{nd}\cdot p\,.
\]
\end{proof}
\begin{lem}
\label{lem:7} Let $DTDWN$ be a decorated topological Van-Kampen
diagram \uline{without numbering}, and let $\mathcal{A}$ be the collection
of all the \uline{reduced} decorated topological Van-Kampen diagrams
$DTD$ (with numberings) for which forgetting the numbering from $DTD$
yields the diagram $DTDWN$.

Then, the probability that the random group $\Gamma$ of level $l$
fulfills the diagram $DTDWN$ \uline{so that the obtained Van-Kampen
diagram over %
\mbox{%
\mbox{$\Gamma$}%
} is reduced}, is at most the probability that the random group $\Gamma$
fulfills one of the diagrams in $\mathcal{A}$. 
\end{lem}

\begin{proof}
Obvious. 
\end{proof}
\begin{example}
Consider the diagram composed of two cells, each of length $l$ (i.e.,
$l$ edges), sharing more than $\frac{l}{6}$ consecutive edges. Assume
that the cells share the same starting point, and admit opposite directions.
This information gives a topological Van-Kampen diagram without numberings
$D$. The probability that the random group of level $l$ fulfills
$D$ is of course $1$, since any (non-free) group (presentation)
fulfills this diagram.

On the other side, in this example there is only two ways to write
numberings in the cells, which are $1,1$, and $1,2$ (the two cells
are equivalent, or not otherwise). Let $D'$ be the one with the first
numbering, and $D''$ be the second.

The probability that the random group $\Gamma$ of level $l$ fulfills
$D$ so that the obtained Van-Kampen diagram over $\Gamma$ is reduced,
is bounded by the probability that $\Gamma$ satisfies $D''$, since
$D'$ is not reduced as a topological Van-Kampen diagram (and hence
it will not provide any reduced Van-Kampen diagram over $\Gamma$).
Let $K$ be the graph associated to $D''$. Then $K$ admits at most
$\frac{11l}{6}$ connected components. Hence, the probability that
the random group of level $l$ and density $d$ satisfies $D''$ is
bounded by

\[
\frac{(2k)^{2}\cdot(2k-1)^{\frac{11l}{6}}}{B_{l}^{2\cdot(1-d)}}\leq\frac{(2k)^{2}}{(2k-1)^{(\frac{1}{6}-2d)\cdot l}}\,.
\]
In particular, this argument shows that if the density $d$ satisfies
$d<\frac{1}{12}$, then the random group of density $d$ cannot admit
two distinct relators that begin with the same $\frac{l}{6}$-piece.
Of course, this argument can easily extended to prove that the random
group of density $d<\frac{1}{12}$, is $C'(\frac{1}{6})$-small cancellation
(it was shown in \cite{Gromov=000020-=000020Asymptotic=000020invariants}
that the random group of density $d<\frac{r}{2}$ is $C'(r)$-small
cancellation). 
\end{example}

\subsection{General Reduction of a Decorated Topological Van-Kampen Diagram}

Let $DTD$ be a decorated topological Van-Kampen diagram. We describe
the following modifications on $DTD$ that does not change the decoration
of the diagram $DTD$. 
\begin{defn}
\label{def:8} Let $DTD$ be a decorated topological Van-Kampen diagram,
and let $K$ be the auxiliary graph associated to $DTD$.

Let $v$ be an \uline{isolated vertex of %
\mbox{%
\mbox{$K$}%
}} (i.e., $\{v\}$ is a non-constant connected component of $K$),
that represents some edge $e$ in (the boundary of) some cell numbered
$i$ in the diagram $DTD$. Let $C$ be the contiguity class of the
edge $e$ in $DTD$. Recall that, since every rigid variable in the
decoration corresponds only to one part of the boundary, if $C$ contains
a non-filament rigid edge, then $C=\{e\}$ is a singleton. 
\begin{itemize}
\item If $C$ does not contain a rigid edge, we eliminate all the cells
that contain an edge in $C$ (all of these cells share the numbering
$i$, since $v$ is isolated), and stretch the filament edges that
belong to $C$ accordingly (while keeping the parts of the boundary
that does not contain an edge in $C$ unchanged - so that the decoration
is preserved). 
\item If $C$ contains a rigid edge, then, if not all of the cells containing
an edge in $C$ are isolated, we eliminate all the cells that contain
an edge in $C$, we stretch the filament non-rigid edges that belong
to $C$ accordingly, and on each of the (filament) rigid edges in
$C$, we paste a copy of the cell numbered $i$ accordingly (since
we don't want to change the rigid parts, we do not stretch the filament
edges in $C$ that belong to rigid parts. Instead, we paste a copy
of the corresponding cell on that edges correspondingly, without changing
the rigid parts). 
\end{itemize}
For clarifying the geometric picture in the last case, if $C$ contains
a non-filament rigid edge, then we do nothing, and otherwise all the
rigid edges in $C$ must be filaments, and in particular we paste
on each of them a copy of the cell numbered $i$ accordingly, in condition
that not all of the cells containing edges in $C$ are isolated (if
all of them are isolated, in the last case, we do nothing).

We call this modification, a \emph{$(v,C)$-isolation} applied on
the decorated diagram $DTD$. 
\end{defn}

\begin{lem}
\label{lem:9} Let the notation be as in \defref{8}, and let $DTD'$
be the decorated diagram obtained by applying a $(v,C)$-isolation
on $DTD$. Let $K'$ be the graph associated to $DTD'$. Assume that
the $(v,C)$-isolation did not reduce the amount of numberings. Then,
the vertex (corresponding to) $v$ is (again) an isolated vertex in
$K'$. 
\end{lem}

\begin{proof}
We keep the notation of \defref{8}. Let $f$ be an edge represented
by $v$ in the diagram $DTD$, and let $F$ be the contiguity class
of $f$ in $DTD$. Assume that $f$ is not contiguous to $e$ in $DTD$.
Since $v$ is an isolated numbering of $K$, all of the cells in $DTD$
that admit an edge that belong to $F$ share the numbering $i$, and,
each of these edges have the same order $r$ in that cells, which
equals also to the order of $e$ in the corresponding cell.

Hence, by the definition of $(v,C)$-isolation, neither the filament
edges in $F$, nor the non-filament edges in $F$, are affected by
this modification. Hence, the collection of edges in $DTD'$ corresponding
to the edges in $F$, forms again a contiguity class in $DTD'$.

Recalling that a $(v,C)$-isolation can add cells to the diagram only
by pasting cells numbered $i$ on the rigid parts along the edges
of order $r$, we conclude that every contiguity class of the edge
of order $r$ in a cell numbered $i$ in the obtained diagram $DTD'$,
contains only filaments and edges that are of order $r$ in cells
numbered $i$. That is, if the $(v,C)$-isolation did not eliminate
the numbering $i$ totally from the diagram, then $v$ will keep to
be an isolated vertex of $K'$. 
\end{proof}
\begin{defn}
\label{def:10} Let the notation be as in \defref{8}. In accordance
with \ref{lem:9}, the procedure of applying iterative $(v,C_{j})$-isolations
on the decorated diagram $DTD$, is called \emph{isolating the numbering
$i$}, or briefly, a \emph{numbering isolation}. 
\end{defn}

\begin{lem}
\label{lem:11} The numbering isolation procedure terminates after
finitely many steps. 
\end{lem}

\begin{proof}
Let $DTD$ be a decorated diagram, and let $v$ be an isolated vertex
in the graph associated to $DTD$, corresponding to an edge in a cell
numbered $i$ in $DTD$. By definition, every $(v,C_{j})$-isolation
either reduces the amount of non-isolated cells with numbering $i$,
or otherwise, it reduces the amount of cells with numbering $i$ in
the diagram. Moreover, a $(v,C_{j})$-isolation does not increase
the amount of non-isolated cells with numbering $i$.

Hence, the procedure of isolating the numbering $i$ must stop to
reduce the amount of non-isolated cells numbered $i$ eventually.
Hence, it must reduce the amount of cells numbered $i$ eventually.
Hence, it must terminate eventually. 
\end{proof}
\begin{lem}
\label{lem:12} Let $DTD$ be a decorated topological Van-Kampen diagram.
An (effective) numbering isolation applied on $DTD$, keeps isolated
numberings isolated, and does not increase the amounts of their (isolated)
cells in the diagram. And, it either 
\begin{enumerate}
\item Isolates a new numbering (that wasn't isolated before), or 
\item Reduces the amount of (isolated) cells of an isolated numbering, or 
\item Reduces the amount of numberings in $DTD$. 
\end{enumerate}
\end{lem}

\begin{proof}
Let $K$ be the graph associated to $DTD$, and let $v$ be an isolated
vertex of $K$ corresponding to an edge in a cell numbered $i$ in
$DTD$. We consider the numbering isolation procedure corresponding
to $v$.

First, note that if $j\neq i$ is another numbering, then the iteration
will not change the amount of cells numbered $j$ in the diagram.
In particular, if $j$ was an isolated numbering then the procedure
above keeps it isolated, and preserves the amount of (isolated) cells
numbered $j$.

Now assume that $i$ was isolated. Then, the procedure will not paste
any new cell to the diagram (it only removes cells in this case).
Hence, the numbering $i$ will stay isolated after the iteration,
and, assuming the procedure was effective, the procedure reduces the
amount of cells numbered $i$ in this case.

Now assume that $i$ was not isolated. Then, the procedure will eliminate
every non-isolated cell numbered $i$ in the diagram, and (maybe)
will paste cells numbered $i$ on rigid edges. Hence, $i$ will be
an isolated numbering after the procedure. 
\end{proof}
\begin{defn}
\label{def:13} Let $DTD$ be a decorated topological Van-Kampen diagram.
Let $x$ be a non-constant variable in the decoration, and let $P_{1}x,...,P_{q}x$
be all the parts of the boundary of $DTD$ representing the variable
$x$. Assume that $P_{1}x$ contains a self-intersection $S$.

For all $i=1,...,q$, 
\begin{itemize}
\item if the part $P_{i}x$ contains a self-intersection compatible with
$S$, we eliminate it from the part $P_{i}x$, and 
\item otherwise, we paste $S$ on the corresponding edges in $P_{i}x$ (filling
the cut sub-diagram $S$ into $P_{i}x$ from the direction of the
unbounded face of the graph $DTD$). 
\end{itemize}
We call this modification an \emph{elimination of self-intersection}.
Note that self-intersection elimination modification, reduces the
boundary length of the diagram. 
\end{defn}

\begin{lem}
\label{lem:14} Let $DTD$ be a decorated topological Van-Kampen diagram.
An (effective) self-intersection elimination modification, reduces
the boundary length of the diagram. Moreover, it keeps isolated numberings
isolated, and does not increase the amounts of their (isolated) cells
in the diagram. 
\end{lem}

\begin{proof}
Let $x$, $P_{1}x,...,P_{q}x$, and $S$ be as in \ref{def:13}. If
$x$ is a rigid variable, then, $q=1$, and $x$ is represented only
by the part $P_{1}x$, so the claim is obvious. Assume now that $x$
is a non-rigid variable, and let $i$ be an isolated numbering in
the diagram. Since $i$ is an isolated numbering, and $x$ is not
a rigid variable, the cut sub-diagram $S$ does not contain any cell
numbered $i$. 
\end{proof}
Finally, we note that we can extend our ``reduction process'' to
include also standard reduction: 
\begin{lem}
\label{lem:15} Let $DTD$ be a decorated topological Van-Kampen diagram.
Applying standard reduction on the diagram $DTD$ (with forgetting
the decoration), reduces the amount of cells in the diagram. Moreover,
it keeps isolated numberings isolated, and does not increase the amounts
of their (isolated) cells in the diagram. 
\end{lem}

\begin{proof}
Obvious. 
\end{proof}
\begin{lem}
\label{lem:16} Let $DTD$ be a decorated topological Van-Kampen diagram.
The procedure of consecutively applying numbering isolation, self-intersection
elimination, and standard reduction, applied on the diagram $DTD$,
terminates after finitely many steps. 
\end{lem}

\begin{proof}
According to \lemref{12}, \lemref{14}, and \lemref{15}, every iteration
of the procedure, keeps isolated numberings isolated, and does not
increase the amounts of their (isolated) cells.

However, since the sequence of amounts of numberings in the diagrams
obtained along the procedure is descending, it must be stabilized
eventually. Hence, the procedure must stop to isolate new numberings
eventually. Hence, according to \lemref{12}, every (effective) numbering
isolation, will reduce the amount of (isolated) cells of the isolated
numberings eventually. Hence, the procedure must stop to apply (effective)
numbering isolation eventually.

But self-intersection elimination modification reduces the boundary
length of the diagram, and standard reduction does not change the
boundary length. Hence, the procedure must stop to apply self-intersection
elimination modifications eventually.

However, standard reduction reduces the amount of cells in the diagram.
Hence, the procedure then must stop to apply standard reduction eventually. 
\end{proof}
\begin{defn}
\label{def:17} Let $DTD$ be a decorated topological Van-Kampen diagram.
In light of \lemref{16}, we summarize our terminology so that the
procedure of consecutively applying numbering isolation, self-intersection
elimination, and standard reduction, on the diagram $DTD$, is called
\emph{generally reducing the decorated diagram $DTD$}.

Hence, a \emph{generally reduced decorated diagram $DTD$} will mean
a reduced decorated diagram with no self-intersections, and for which
for every isolated vertex $v$ in the graph associated to $DTD$,
the contiguity class $C$ of any edge $e$ in $DTD$ represented by
$v$, contains a rigid edge, but all the cells that contain an edge
in $C$ are isolated. 
\end{defn}

\begin{rem}
Note that running general reduction on a decorated diagram, could
affect the rigid parts of the boundary only by eliminating self-intersections
from them. It is not hard to see that running general reduction on
a decorated diagram, does not affect the constant parts of the boundary. 
\end{rem}

\begin{thm}
\label{thm:18} Let $DTD$ be a \uline{generally reduced} decorated
topological Van-Kampen diagram, with associated graph $K$. Let $A$
be the amount of isolated vertices in $K$, and let $R$ be the total
amount of rigid edges in $DTD$. Then

\[
A\leq R\,.
\]
Moreover, if $i$ is a numbering of a cell in the diagram, so that
there exists an isolated vertex $v$ in $K$ corresponding to the
numbering $i$, then $i$ is an isolated numbering. In particular,
if there exists no isolated numberings in the generally reduced diagram
$DTD$, then $A=0$. 
\end{thm}

\begin{proof}
Let $v$ be an isolated vertex of $K$, and let $C$ be a contiguity
class of an edge in $DTD$ that is represented by the vertex $v$.
If $C$ does not contain a rigid edge, then, by the definition of
general reduction, the decorated diagram is not generally reduced.
If $v'\neq v$ is another isolated vertex of $K$, and $C'$ is a
contiguity class for an edge in the diagram represented by the vertex
$v'$, then $C$ and $C'$ are disjoint sets (otherwise, $v,v'$ will
be connected by an edge).

For the last statement, assume that $v$ is an isolated vertex of
$K$, and let $e$ be an edge represented by $v$ in $DTD$. Denote
by $c$ a cell containing the edge $e$ in the diagram. If $c$ is
not isolated, then, by the definition of general reduction, the diagram
is not generally reduced. 
\end{proof}

\subsection{Mining in Decorated Van-Kampen Diagrams}

We now describe a modification applied on a decorated diagram, a modification
that could change the rigid parts. 
\begin{defn}
\label{def:19} Let $DTD$ be a decorated topological Van-Kampen diagram.
A \emph{mining modification} in $DTD$ is eliminating all the (isolated)
cells of an isolated numbering, from the diagram, without changing
the non-rigid variables. 
\end{defn}

\begin{lem}
\label{lem:20} Let $DTD$ be a decorated topological Van-Kampen diagram.
Let $DTD'$ be the diagram obtained after an (effective) mining modification
on $DTD$. Then, 
\begin{enumerate}
\item The diagram $DTD'$ contains less numberings than $DTD$. 
\item The probability that the random group $\Gamma$ of level $l$ and
density $d$, fulfills the original diagram $DTD$, is at most the
probability that $\Gamma$ fulfills the obtained diagram $DTD'$. 
\end{enumerate}
\end{lem}

\begin{proof}
Point 1 is obvious. For Point 2, we note that every group that fulfills
$DTD$, must fulfill the obtained diagram $DTD'$. 
\end{proof}
After applying a mining modification on a decorated diagram, we generally
reduce the obtained diagram. Thus, we extend the meaning of a mining
modification, to include further generally reducing the obtained diagram.

\subsection{Decorated Topological Van-Kampen Diagrams over Random Groups}
\begin{lem}
\label{lem:21} Let $DTD$ be a decorated topological Van-Kampen diagram.
Let $DTD'$ be the diagram obtained by generally reducing $DTD$.
Then, the probability that the random group $\Gamma$ of level $l$
and density $d$, fulfills the original diagram $DTD$, is at most
the probability that $\Gamma$ fulfills the obtained generally reduced
diagram $DTD'$. 
\end{lem}

\begin{proof}
We note that every group that fulfills $DTD$, must fulfill the diagram
obtained by applying numbering isolation, or self-intersection elimination,
or standard reduction, on $DTD$. 
\end{proof}
\begin{thm}
\label{thm:22} Let $DTD$ be a \uline{generally reduced} decorated
topological Van-Kampen diagram, that admits $n$ distinct numberings.
Let $R$ be the total amount of rigid edges in $DTD$. Denote by $B_{l}$
the set of cyclically reduced words of length $l$ in $F_{k}$.

Then, the probability that a random $n$-tuple of words in $B_{l}$
fulfills the diagram $DTD$ is at most

\[
\frac{(2k)^{n}\cdot(2k-1)^{\frac{nl+R}{2}}}{B_{l}^{n}}\,.
\]

Moreover, if there exist no isolated numberings in $DTD$, then the
probability that a random $n$-tuple of words in $B_{l}$ fulfills
the diagram is at most

\[
\frac{(2k)^{n}\cdot(2k-1)^{\frac{nl}{2}}}{B_{l}^{n}}\,.
\]
\end{thm}

\begin{proof}
Let $K$ be the auxiliary graph associated to the decorated diagram
$DTD$. Since $DTD$ is generally reduced, we have, according to \thmref{18},
that the amount $A$ of isolated vertices in $K$ is at most $A\leq R$,
and in case $DTD$ contains no isolated numberings, we have that $A=0$.

Denote by $c$ the amount of constant vertices in $K$, and denote
by $V$ the amount of non-isolated non-constant vertices of $K$.
Of course, the graph $K$ consists of $nl$ vertices. Hence,

\[
V=nl-A-c\,,
\]
and by definition, every non-isolated non-constant vertex belongs
to a connected component of $K$ of size at least $2$.

Hence, the amount of distinct non-constant connected components in
$K$ is at most

\[
u\leq\frac{V}{2}+A\leq\frac{nl+A}{2}\,.
\]

Now \lemref{5} completes the proof. 
\end{proof}
\begin{thm}
\label{thm:23} Let $DTD$ be a \uline{generally reduced} decorated
topological Van-Kampen diagram, that admits $n$ distinct numberings.
Let $R$ be the total amount of rigid edges in $DTD$. Denote by $B_{l}$
the set of cyclically reduced words of length $l$ in $F_{k}$.

Then, the probability that the random group $\Gamma$ of level $l$
and density $d$, fulfills the decorated diagram $DTD$, is at most

\[
\frac{(2k)^{n}\cdot(2k-1)^{\frac{nl+R}{2}}}{B_{l}^{n\cdot(1-d)}}\leq\left(\frac{2k}{(2k-1)^{(\frac{1}{2}-d-\frac{R}{2nl})l}}\right)^{n}\,.
\]

Moreover, if there exist no isolated numberings in $DTD$, then the
probability that $\Gamma$ fulfills the diagram is at most

\[
\left(\frac{2k}{(2k-1)^{(\frac{1}{2}-d)l}}\right)^{n}\,.
\]
\end{thm}

\begin{proof}
We only used \thmref{22} and evaluated in \lemref{6}. 
\end{proof}

\section{Isoperimetric Inequality in Random Groups }\label{sec:Isoperimetric-Inequality-in-Random-Groups}

\subsection{Isoperimetric Inequality and the Hyperbolicity Constant}
\begin{thm}
\label{thm:24} (\cite{Ollivier=000020-=000020Sharp=000020phase}).
At density $d<\frac{1}{2}$, for every $\epsilon>0$, with overwhelming
probability, every reduced Van-Kampen diagram $D$ over a random group
at density $d$ satisfies

\[
|\partial D|\ge(1-2d-\epsilon)\cdot l\cdot|D|\,,
\]
where $|D|$ is the number of faces in $D$, and $|\partial D|$ is
the length of its boundary.

In particular, for $\epsilon=\frac{1}{2}-d$, we have with overwhelming
probability that

\[
|D|\leq\frac{|\partial D|}{l\cdot(\frac{1}{2}-d)}\,.
\]
\end{thm}

\begin{rem}
\label{rem:25} Let $d<\frac{1}{2}$. Assume that we have a collection
$\mathcal{V}$ of reduced decorated topological Van-Kampen diagrams
(whose cells are of boundary length $l$), and that we are interested
in bounding the probability that the random group $\Gamma$ of level
$l$ and density $d$ fulfills some diagram in the collection $\mathcal{V}$.

Then, according to \thmref{24}, we can drop all the diagrams $D$
in our collection $\mathcal{V}$, for which $|D|>\frac{|\partial D|}{l\cdot(\frac{1}{2}-d)}$,
since they add nothing (a negligible fraction) to the probability
that we wanted to bound. In accordance with this, we will say that
a diagram $D$ is among \emph{those diagrams that the random group
$\Gamma$} (of level $l$ and density $d$) \emph{could satisfy},
if $|D|\leq\frac{|\partial D|}{l\cdot(\frac{1}{2}-d)}$ (and each
of the cells in $D$ are of boundary length $l$). 
\end{rem}

\begin{thm}
\label{thm:26} Let $d<\frac{1}{2}$, and denote

\[
C_{0}=\frac{\frac{1}{2}-d}{4}\,.
\]
Let $\Gamma$ be the random group of level $l$ and density $d$.

Then, with overwhelming probability, the ball of radius $C_{0}l$
in the Cayley graph of $\Gamma$, is a tree. 
\end{thm}

\begin{proof}
Let $D$ be a reduced Van-Kampen diagram over $\Gamma$ for some relation
of length at most $2C_{0}l$ in $\Gamma$. Then, according to \thmref{24},
the number of cells in $D$ is bounded by 
\[
|D|\leq\frac{1}{2}\,.
\]
Hence, $|D|=0$. 
\end{proof}
\begin{thm}
\label{thm:27} (\cite{Ollivier=000020-=000020Some=000020small=000020cancellation,Ollivier=000020-=000020Sharp=000020phase}).
At density $d<\frac{1}{2}$, for every $\epsilon>0$, with overwhelming
probability, the random group $\Gamma$ of level $l$ is torsion-free
$\frac{12l}{(1-2d-\epsilon)^{2}}$-hyperbolic.

In particular, $\Gamma$ is $\alpha_{0}l$-hyperbolic, where 
\[
\alpha_{0}=\frac{48}{(1-2d)^{2}}\,.
\]
\end{thm}

\subsection{Counting Topological Van-Kampen Diagrams}
\begin{thm}
\label{thm:28} Let $W$ be a word in some free group. Denote by $B(m,S)$
the amount of decorated topological Van-Kampen diagrams $DTD$ with
decoration $W$, with at most $m$ cells each of boundary length $l$,
and with boundary length at most $S$, so that every part of the boundary
of $DTD$ (representing a variable in the decoration $W$), contains
no self-intersection.

Then, the number $B(m,S)$ is bounded by

\[
B(m,S)\leq m\cdot\left(\left(4(ml)^{3}\right)^{m}\cdot S^{3}\right)^{m}\cdot2S^{3|W|+1}\cdot|W|\cdot\left(2(|W|+2)\cdot\left(4(|W|+2)^{2}S\right)^{2(|W|+2)}\right)^{|W|}\,.
\]

In particular, there exists a polynomial $Q$, so that, if $m,|W|\leq\ln l$,
Then,

\[
B(m,S)\leq\left(Q(l)\right)^{\ln^{2}l}\,.
\]
\end{thm}

\begin{proof}
First, let us bound the amount $C(m)$ of circular (no filaments)
topological Van-Kampen diagrams with at most $m$ cells, each of length
$l$. To construct a diagram in $C(m)$, we have to paste a new cell
on a diagram in $C(m-1$). For that we have to choose two points on
the boundary of that diagram in $C(m-1)$, and a direction (where
to paste). Then, we choose a starting vertex, a direction, and a numbering
for the added cell. Hence,

\[
C(m)\leq2(ml)^{2}\cdot l\cdot2\cdot m\cdot C(m-1)=4(ml)^{3}\cdot C(m-1)\,.
\]
Since $C(0)=1$, we get by induction that 
\[
C(m)\leq\left(4(ml)^{3}\right)^{m}\,.
\]

Every diagram in $B(m,S)$, consists of at most $m$ circular components
that belong to $C(m)$, and that are connected by (maybe complicated)
trees. For constructing a diagram in $B(m,S)$, we start by connecting
the circular components by lines, and after that we attach the remaining
filament parts.

For doing that we start by choosing the amount $n\leq m$ of circular
components. Then, we choose iteratively $n$ diagrams from $C(m)$,
and in each iteration, we choose a point on the boundary of the diagram
obtained in the previous iteration, we choose a point on the boundary
of the chosen diagram in the present iteration, and then we connect
these two points by a line that we have to choose its length. Since
we want to get at the end a diagram in $B(m,S)$, whose boundary length
is at most $S$, we can assume that the boundary length of each of
the diagrams obtained or chosen in any iteration is at most $S$.
Hence, for choosing $n\leq m$ circular components from $C(m)$ and
connecting them so that the obtained diagram is a sub-diagram of a
diagram in $B(m,S)$, we have at most 
\[
m\cdot\left(C(m)\cdot S^{3}\right)^{m}
\]
different choices. Denote by $C'(m)$ the collection of diagrams that
are constructed in the last paragraph.

In order to construct a diagram in $B(m,S)$, we choose a diagram
in $C'(m)$ and attach the possible remaining filament components.
These remaining filament components can be described by cancellation
trees for solutions (in a free group) of length at most $S$ of equations
of the form $UwV$, where $w$ is a sub-word of some cyclic permutation
of $W$, and $U,V$ denote new variables. According to \lemref{29},
the amount of such trees is bounded by

\[
T(|W|+2,S)\leq2(|W|+2)\cdot\left(4(|W|+2)^{2}S\right)^{2(|W|+2)}\,.
\]

According to the assumption, every part of the boundary of a decorated
diagram in $B(m,S)$ representing some variable in the decoration
$W$, contains no self-intersection. Thus, the amount of trees of
cancellation that we should attach to a diagram in $C'(m)$ for obtaining
a diagram in $B(m,S)$, is at most $t\leq|W|$. Hence, in order to
construct a diagram in $B(m,S)$, we choose a diagram $D$ in $C'(m)$,
we choose the amount $t\leq|W|$ of trees from $T(|W|+2,S)$ to be
attached to the boundary of $D$, we choose $t$ points $p_{i}$ on
the boundary of $D$, we choose $t$ trees $T_{i}$ in $T(|W|+2,S)$,
we choose a point $q_{i}$ on $T_{i}$ for each $i$, and we identify
the points $p_{i}$ and $q_{i}$ for all $i$. For doing this, we
have at most

\[
m\cdot\left(C(m)\cdot S^{3}\right)^{m}\cdot|W|\cdot S^{|W|}\cdot\left(T(|W|+2,S)\right)^{|W|}\cdot S^{|W|}\,.
\]

Let $C''(m)$ be the diagrams obtained in the last paragraph.

Finally, in order to construct a diagram in $B(m,S)$, we choose a
diagram $D$ in $C''(m)$, we choose a starting vertex and a direction
for the boundary of that diagram $D$. And then, for ``writing''
the decoration $W$ along the boundary of $D$, it remains only to
specify the length of each variable in $W$. We conclude that 
\[
B(m,S)\leq m\cdot\left(C(m)\cdot S^{3}\right)^{m}\cdot|W|\cdot S^{|W|}\cdot\left(T(|W|+2,S)\right)^{|W|}\cdot S^{|W|}\cdot2S\cdot S^{|W|}\,.
\]
\end{proof}
\begin{lem}
\label{lem:29} Let $L$ be an integer. Then, the amount $T(L,S)$
of combinatorial (counting the number of edges but not the labels)
cancellation trees for all the solutions of circumference $S$ in
a free group of the equation $W=1$, for some word $W$ of length
at most $L$, is bounded by

\[
T(L,S)\leq2L\cdot\left(4L^{2}S\right)^{2L}\,.
\]
\end{lem}

\begin{proof}
We note that any tree of cancellation in $T(L,S)$, can be topologically
described by a tree with at most $2L$ vertices, and thus at most
$2L$ edges. Let $b(L)$ be the total amount of such trees. Now for
each tree in $b(L)$ we have to choose lengths for its edges. For
that each edge has at most $S$ options, and hence in total we have
at most $S^{2L}$ options. Finally, for creating a tree in $b(L)$,
we have to choose the number $v\leq2L$ of vertices, and then to choose
$v-1$ pairs from a set of size $v$. 
\end{proof}

\section{Axes of Elements in Free Groups }\label{sec:Axes-of-Elements-in-Free-Groups}
\begin{defn}
\label{def:30} Let $G\in F_{k}$ (reduced word), and suppose that
we have the graphical equality (with no cancellations)

\[
G=YgY^{-1}\,,
\]
for (reduced) words $Y,g\in F_{k}$, so that $g$ is cyclically reduced.

Then, the path in the Cayley graph of $F_{k}$ induced by the elements
$Zg^{n}$, $n\in\mathbb{Z}$, is called the axis of the element $G$
in $F_{k}$, and is usually denoted by $A_{G,F_{k}}$. The translation
length of $G$, denoted by $[G]_{F_{k}}$, is defined to be 
\[
[G]_{F_{k}}=|g|\,.
\]
\end{defn}

\begin{lem}
\label{lem:31} Let $G,H\in F_{k}$ be two elements. If $G,H$ does
not commute, then

\[
\text{diam}(A_{G,F_{k}}\cap A_{H,F_{k}})<2\max\{[G]_{F_{k}},[H]_{F_{k}}\}\,.
\]
\end{lem}

\begin{proof}
Assume $[G]_{F_{k}}\geq[H]_{F_{k}}$, and without loss of generality,
assume that $G=g$, and $H=h$, are cyclically reduced words. Assume
that

\[
\text{diam}(A_{G,F_{k}}\cap A_{H,F_{k}})\geq2\max\{[G]_{F_{k}},[H]_{F_{k}}\}\,.
\]
By replacing $g$ with $g^{-1}$ if necessary, we may assume that
$g,h$ translates in the same directions on their axes.

Now the segment $[1,g^{2}]$, which is a sub-segment of $A_{G,F_{k}}$,
lies entirely in $A_{H,F_{k}}$. Denote

\begin{align*}
 & x=gh\,,\\
 & y=hg\,.
\end{align*}
Since $g\in A_{H,F_{k}}$, we have that $y\in A_{H,F_{k}}$ and 
\[
d(1,y)=|g|+|h|\,.
\]
Since $|h|\leq|g|$, and $h$ and $g$ translate in the same direction,
and $[1,g]$ is a sub-segment of $A_{H,F_{k}}$, we have that $h\in[1,g]$.
In particular, $h\in A_{G,F_{k}}$, and we get that $x\in[1,g^{2}]$
and 
\[
d(1,x)=|g|+|h|\,.
\]
Since $g,h$ translate in the same direction, then $x$ and $y$ lie
on the same side of $1$ on the line $A_{H,F_{k}}$. Hence,

\[
x=y\,.
\]
\end{proof}
\begin{lem}
\label{lem:32} Let $p,q,G,H$ be reduced words in the free group
$F_{k}$, and suppose that we have the graphical equalities (with
no cancellations)

\begin{align*}
 & G=YgY^{-1}\,,\\
 & H=ZhZ^{-1}\,,
\end{align*}
for some $g,h,Y,Z\in F_{k}$ (reduced words), so that $g,h$ are cyclically
reduced. Assume that there exists an integer $n_{0}>10$ so that neither
$p$ nor $q$ contains a sub-word of the form $g^{\pm n_{0}},h^{\pm n_{0}}$.

Then, for all integers $r,s>110n_{0}$, if

\[
pG^{r}q^{-1}H^{-s}=1\,,
\]
then the elements $pGp^{-1}$ and $H$ commute. 
\end{lem}

\begin{proof}
Without loss of generality, assume that 
\[
|g|\geq|h|\,.
\]
According to the assumption, we have that 
\[
pYg^{r}Y^{-1}q^{-1}Zh^{-s}Z^{-1}=1\,.
\]
Let $T$ be a tree of cancellation for the last equation. We distinguish
$6$ special segments in the tree $T$; those $6$ segments that correspond
to the $6$ elements

\[
p,\,Yg^{r}Y^{-1},\,q,\,Zh^{s}Z^{-1}\,,g^{r},\,h^{s}\,.
\]
Since neither $p$ nor $q$ contains a sub-word of the form $g^{\pm n_{0}},h^{\pm n_{0}}$,
the tree $T$, topologically, consists of two disjoint segments corresponding
to the elements $p,q$, that are connected by a non-degenerate segment
$I$ (so that $I$ intersects each of them in exactly one point).
Denote by $A_{1},A_{2}$ and $B_{1},B_{2}$, the initial and the terminal
points of the segments corresponding to the elements $p$ and $q$
correspondingly (so that $p=[A_{1},A_{2}]$ and $q=[B_{1},B_{2}]$).
Denote also $I=[C_{1},C_{2}]$, where $C_{1}$ is the point in $I\cap[A_{1},A_{2}]$,
and $C_{2}$ is the point in $I\cap[B_{1},B_{2}]$.

Now denote by $E$ and $F$ the middle points of the segments $[A_{1},B_{1}]$
and $[A_{2},B_{2}]$ (corresponding to the elements $Zh^{s}Z^{-1}$
and $Yg^{r}Y^{-1}$) respectively. Note that $I=[A_{1},B_{1}]\cap[A_{2},B_{2}]$,
and that $E$ and $F$ must lie on $I$ (since neither $p$ nor $q$
contains a sub-word of the form $g^{\pm n_{0}},h^{\pm n_{0}}$). We
define a metric on $T$ by declaring that every edge is of length
$1$. A (compatible) directions are defined on the segments $[A_{1},B_{1}]$
and $[A_{2},B_{2}]$ by declaring that $A_{1}$ and $A_{2}$ are their
left endpoints respectively.

We further define ``coordinates'' for the points in $T$ by denoting
$D(x)$ and $U(x)$ the points of (directed) distance $x$ from the
point $E$ and that lie on the segments $[A_{1},B_{1}]$ and $[A_{2},B_{2}]$
respectively, for all (possible) real $x$. Note that $D(x)=U(x)$
when $D(x)$ or $U(x)$ lie on $I$. In such a case, we denote the
point $D(x)=U(x)$, simply by $x$.

Let $[L^{h},M^{h}]$ and $[L^{g},M^{g}]$ be the sub-segments of $[A_{1},B_{1}]$
and $[A_{2},B_{2}]$, corresponding to the elements $h^{s}$ and $g^{r}$
respectively. Since $|g|\geq|h|$, and $g,h$ are cyclically reduced
words, it suffices to show that a sub-segment of length $2|g|$ of
$[L^{g},M^{g}]$ matches a sub-segment of $[L^{h},M^{h}]$ (by \lemref{31}).
Assume by contradiction that this is not the case.

Let $R_{E}$ and $R_{F}$ be the reflections of the segments $[A_{1},B_{1}]$
and $[A_{2},B_{2}]$ around the points $E$ and $F$ respectively.
We note that

\begin{align*}
 & Lab(e_{1})=Lab(R_{E}(e_{1}))^{-1}\,,\\
 & Lab(e_{2})=Lab(R_{F}(e_{2}))^{-1}\,,
\end{align*}
for all edges $e_{1}$ and $e_{2}$ in the sets $[A_{1},L^{h}]\cup[M^{h},B_{1}]$
and $[A_{2},L^{g}]\cup[M^{g},B_{2}]$ respectively, where $Lab(e)$
denotes the label of the edge $e$ read from left to right.

Since neither $p$ nor $q$ contains a sub-word of the form $g^{\pm n_{0}}$,
there must be a sub-segment $[L_{1}^{g},M_{1}^{g}]$ of $[L^{g},M^{g}]$
that lies entirely in $I$, and whose length is at least 
\[
|[L_{1}^{g},M_{1}^{g}]|\geq108n_{0}|g|\,.
\]
Hence, there exists a sub-segment $[L_{2}^{g},M_{2}^{g}]$ of $[L_{1}^{g},M_{1}^{g}]$
that lies entirely in $I$ on either the left hand side or the right
hand side of the point $E$, and that is of length at least 
\[
|[L_{2}^{g},M_{2}^{g}]|\geq54n_{0}|g|\,.
\]
Without loss of generality, assume that $[L_{2}^{g},M_{2}^{g}]$ lies
on the left hand side of the point $E$.

Since we have assumed that there is no sub-segment of length $2|g|$
of $[L^{g},M^{g}]$ that matches a sub-segment of $[L^{h},M^{h}]$,
we deduce that the point $L^{h}$ lies in the sub-segment $[M_{2}^{g}-2|g|,E]$.

Thus, the sub-segment $[L_{2}^{g},M_{2}^{g}-2|g|]$ of $[L^{g},M^{g}]$
lies entirely in $I\cap[A_{1},L^{h}]$, and is of length at least

\[
|[L_{2}^{g},M_{2}^{g}-2|g|]|\geq53n_{0}|g|\,.
\]
Let $[L_{3}^{g},M_{3}^{g}]$ be the longest sub-segment of $[L^{g},M^{g}]$
with this property. In particular, if $E$ lies in $[A_{1},M^{g}]$,
then, $M_{3}^{g}=L^{h}$.

Then, since $[L_{3}^{g},M_{3}^{g}]\subset[A_{1},L^{h}]$, we have

\[
Lab([R_{E}(M_{3}^{g}),R_{E}(L_{3}^{g})])=Lab(R_{E}([L_{3}^{g},M_{3}^{g}]))=Lab([L_{3}^{g},M_{3}^{g}])^{-1}\,.
\]

Hence, since the label of the segment $Lab([L^{g},M^{g}])=g^{r}$
cannot contain $g^{-1}$ as a sub-word ($g\neq1$), the segment

\[
[R_{E}(M_{3}^{g}),R_{E}(L_{3}^{g})]
\]
cannot match $[L^{g},M^{g}]$ in a sub-segment of length more than
$2|g|$.

And since $q$ does not contain a sub-word of the form $g^{\pm n_{0}}$,
the segment

\[
[R_{E}(M_{3}^{g}),R_{E}(L_{3}^{g})]
\]
cannot match the segment $[B_{1},B_{2}]$ (corresponding to $q$)
in a sub-segment of length more than $(n_{0}+1)|g|$.

We deduce that the segment 
\[
[L_{3}^{g},M_{3}^{g}]
\]
contains a sub-segment $J$ (that lies in $J\subset I\cap[A_{1},L^{h}]$),
so that its reflection around $E$, $R_{E}(J)$, is a sub-segment
of $I$ that lies entirely on the right hand side of $M^{g}$, and
so that the length of $J$ equals 
\[
|J|=50n_{0}|g|\,.
\]

Since $J$ lies entirely on the right hand side of $M^{h}$, this
implies in particular that $M^{h}$ lies in $[E,C_{2}-50n_{0}|g|]\subset I$.

Now if $E$ lies in $[A_{1},M^{g}]$, then, as stated before, we have
$M_{3}^{g}=L^{h}$. In particular, the segment $[L^{h}-10|g|,L^{h}]$,
is a sub-segment of $[L^{g},M^{g}]$ and lies entirely in $I$. Hence,
applying $R_{E}$ on $[L^{h}-10|g|,L^{h}]$, gives a sub-segment of
$I$ that cannot match $[L^{g},M^{g}]$ in a sub-segment of length
more than $2|g|$. But

\[
R_{E}([L^{h}-10|g|,L^{h}])=[M^{h},M^{h}+10|g|]\,.
\]
Thus, the point $M^{g}$ must lie in $[M^{h},M^{h}+2|g|]$.

In particular, we deduce that the point $F$ lies on the left hand
side of the point $E$, and the distance between them is at least

\[
E-F\geq50n_{0}|g|\,.
\]

Since $E-F\geq50n_{0}|g|$, we get that for all $x\in I$, if $R_{E}(x)$
belongs to $I$, then $R_{F}\circ R_{E}(x)$ is a translation of $x$
by at least $100n_{0}|g|$ to the left on the segment $[A_{2},B_{2}]$.
Similarly, for all $x\in I$, if $R_{F}(x)$ belongs to $I$, then
$R_{E}\circ R_{F}(x)$ is a translation of $x$ by at least $100n_{0}|g|$
to the right on the segment $[A_{1},B_{1}]$.

Hence, there must exists an (unique) integer $n$, so that the segment
$\left(R_{F}\circ R_{E}\right)^{n}(J)$ does not lie entirely in $I$,
or the segment $R_{E}\circ\left(R_{F}\circ R_{E}\right)^{n}(J)$ does
not lie entirely in $I$. Assume without loss of generality that $\left(R_{F}\circ R_{E}\right)^{n}(J)$
does not lie entirely in $I$. Since $p$ does not contain a sub-word
of the form $g^{\pm n_{0}}$, the segment $\left(R_{F}\circ R_{E}\right)^{n}(J)$
cannot match the segment $[A_{1},A_{2}]$ (corresponding to $p$)
in a sub-segment of length more than $(n_{0}+1)|g|$, and hence, the
segment $\left(R_{F}\circ R_{E}\right)^{n}(J)$ admits a sub-segment
$J'$ that lies entirely in $I$ on the left hand side of $L^{h}$,
and whose length is

\[
|J'|=45n_{0}|g|\,.
\]
Since $q$ does not contain a sub-word of the form $g^{\pm n_{0}}$,
the segment $R_{E}(J')$ cannot match the segment $[B_{1},B_{2}]$
(corresponding to $q$) in more than $(n_{0}+1)|g|$. Hence, $R_{E}(J')$
contains a sub-segment $J''$ that lies entirely in $I$ on the right
hand side of $M^{g}$, and whose length is

\[
|J''|=40n_{0}|g|\,.
\]
But $R_{F}(J'')$ lies entirely in the segment $[A_{2},C_{1}]$ which
is a sub-segment of the segment $[A_{1},A_{2}]$ corresponding to
$p$. Thus, $p$ contains a sub-word of the form $g^{-20n_{0}}$,
a contradiction. 
\end{proof}
Finally, we state the following simple fact, which will help us in
bounding probabilities in what follows. 
\begin{lem}
\label{lem:33} For every $a>1$ and every integers $r,s>0$, the
function 
\[
\frac{x^{r\ln^{s}x}}{a^{x}}
\]
approaches $0$ as $x$ goes to $\infty$. 
\end{lem}

\begin{proof}
We have that: 
\[
\frac{x^{r\ln^{s}x}}{a^{x}}=\frac{e^{r\ln^{s+1}x}}{a^{x}}=\frac{c^{\ln^{s+1}x}}{a^{x}}\,,
\]
where $c=e^{r}$. Let $\epsilon>0$. Then, for $x>0$, we have 
\begin{align*}
 & \frac{c^{\ln^{s+1}x}}{a^{x}}<\epsilon\\
\iff & c^{\ln^{s+1}x}<\epsilon a^{x}\\
\iff & \ln^{s+1}x\cdot\ln c<\ln\epsilon+x\ln a\\
\iff & \frac{\ln^{s+1}x}{x}<\frac{\ln a}{r}+\frac{\ln\epsilon}{rx}\,.
\end{align*}
For all large enough $x$, we have $-\frac{\ln a}{2r}<\frac{\ln\epsilon}{rx}$,
and $\frac{\ln^{s+1}x}{x}<\frac{\ln a}{2r}$. 
\end{proof}

\section{Lifting Solutions of Bounded Lengths over Random Groups }\label{sec:Lifting-Solutions-of-Bounded-Lengths}
\begin{defn}
\label{def:34} Let $y=(y_{1},...,y_{q})$ be a free basis, and let
$\Sigma=\Sigma(y,a)$ be a system of equations with variables $y$
and with (or without) constants in the alphabet $a^{\pm}$. Let $\Gamma=\langle a:\mathcal{R}\rangle$
be a group presentation. 
\begin{itemize}
\item A tuple of elements $y_{0}\in\Gamma^{q}$, is called \emph{a solution
for the system $\Sigma$ over the group $\Gamma$}, if each of the
equations defining $\Sigma$ is satisfied in $\Gamma$ when we evaluate
$y=y_{0}$ in it. Briefly, in this case we write $y_{0}\in\Gamma$
and $\Sigma(y_{0},a)=1$ over $\Gamma$. We often treat $y_{0}$ as
(a tuple of) element(s) of $F_{k}$, but in those cases, we intend
that $y_{0}$ is a reduced word that represents a geodesic in $\Gamma$
(a geodesic word). 
\item We define \emph{the circumference $C_{\Sigma}(y_{0})$ of a solution
$y_{0}$ over the group $\Gamma$} of a system of equations $\Sigma(y,a)$,
to be\\
 \\
 
\[
C_{\Sigma}(y_{0})=\underset{w}{\sum}||w(y_{0})||\,,
\]
where $w$ runs on the equations in $\Sigma$, $y_{0}$ represent
geodesic words in $\Gamma$, and $||w(y_{0})||$ is the length of
the word $w(y_{0})$ in a free semi-group (without canceling). 
\item \emph{The length $L_{\Sigma}(y_{0})$ of the solution $y_{0}$ over
$\Gamma$}, is defined to be the length of the longest specialization
of the variables given by the tuple $y_{0}$ (the lengths measured
in $\Gamma$). 
\item The number of equations in the system $\Sigma$ is denoted by $N(\Sigma)$.
And \emph{the length of the system $\Sigma(y,a)$}, denoted $|\Sigma|$,
is defined to be the sum of the lengths of the words defining the
equations of $\Sigma$ (measured in a free group). 
\item Since the group is given by presentation $\Gamma=\langle a:\mathcal{R}\rangle$,
it is equipped with an epimorphism $\pi_{\Gamma}:F_{k}\rightarrow\Gamma$
(mapping $a$ to $a$ - recall that $F_{k}$ denotes a free group
with a fixed basis $a$). Let $y_{0}\in\Gamma$ be a solution of $\Sigma$
over $\Gamma$. We say that the solution $y_{0}=(y_{1}^{0},...,y_{q}^{0})$
\emph{can be lifted to a solution of $\Sigma$ in $F_{k}$}, if for
all $i=1,...,q$, there exists an element $\tilde{y}_{i}^{0}$ in
the set $\pi_{\Gamma}^{-1}(y_{i}^{0})$ of preimages of $y_{i}^{0}$,
so that $\tilde{y}_{0}=(\tilde{y}_{1}^{0},...,\tilde{y}_{q}^{0})$
is a solution for $\Sigma$ over $F_{k}$. 
\end{itemize}
\end{defn}

\begin{rem}
Note that for every solution $y_{0}\in\Gamma$ of the system $\Sigma$,
we have

\[
C_{\Sigma}(y_{0})\leq|\Sigma|\cdot L_{\Sigma}(y_{0})\,.
\]
\end{rem}

\begin{example}
Consider the equation $y^{2}=1$, and the group $\Gamma=\langle a:a_{1}^{2}\rangle$.
The solution $y_{0}=a_{1}$ of the equation $y^{2}=1$ over $\Gamma$,
cannot be lifted to a solution in $F_{k}$. 
\end{example}

\begin{thm}
\label{thm:35} Let $d<\frac{1}{2}$. Let $\Sigma$ be a system of
equations, and let $\Gamma$ be the random group of level $l$ and
density $d<\frac{1}{2}$. We denote by $p_{l}$ the probability that
$\Gamma$ admits a solution of length at most $l\ln^{2}l$, that cannot
be lifted to a solution of $\Sigma$ in $F_{k}$.

Then 
\[
p_{l}\overset{l\rightarrow\infty}{\longrightarrow}0\,.
\]
\end{thm}

\begin{proof}
Let $\mathcal{V}^{0}$ be the collection of all the reduced decorated
families of topological Van-Kampen diagrams $DFTD$ with decoration
$\Sigma$ and boundary length at most $|\Sigma|\cdot l\ln^{2}l$ (the
boundary length of a family of diagrams is the sum of the lengths
of the boundaries of the diagrams in that family), that $\Gamma$
could satisfy. For each such $DFTD$ in $\mathcal{V}^{0}$, we declare
no rigid parts in the boundary, and the parts in the boundary corresponding
to the constants appearing in $\Sigma$, are declared to be the constant
parts of the boundary.

Since the boundary length of each of the diagrams in each of the decorated
families in $\mathcal{V}^{0}$ is at most $|\Sigma|\cdot l\ln^{2}l$,
we have according to \thmref{24}, that the amount of cells $m$ in
any diagram in $\mathcal{V}^{0}$ is bounded by $m\leq\beta_{0}\ln^{2}l$,
where $\beta_{0}=\frac{|\Sigma|}{\frac{1}{2}-d}$ (see \remref{25}).
Thus, according to \thmref{28}, there exists a polynomial $Q=Q_{d}$,
so that

\[
|\mathcal{V}^{0}|\leq Q(l)^{\ln^{4}l}\,.
\]

We generally reduce each of the decorated families in $\mathcal{V}^{0}$,
and denote the collection of the obtained generally reduced decorated
families by $\mathcal{V}$.

We denote by $\mathcal{V}_{t}$ the sub-collection of decorated families
in $\mathcal{V}$ that contain no cells, i.e., those decorated families
consisting of $N(\Sigma)$ trees (where $N(\Sigma)$ is the number
of equations in $\Sigma$). And we denote by $\mathcal{V}_{1}$ the
remaining decorated families in $\mathcal{V}$.

We denote by $\mathcal{V}_{t}^{0}$ and $\mathcal{V}_{1}^{0}$, the
sub-collections consisting of those diagrams in $\mathcal{V}^{0}$
who admit an associated diagram in $\mathcal{V}_{t}$ and $\mathcal{V}_{1}$
respectively.

By the definition of the collection $\mathcal{V}_{1}$, every decorated
family in $\mathcal{V}_{1}$ admits a positive amount $n\geq1$ of
distinct numberings, but contains no isolated numberings (since there
are no rigid variables in the decoration). Hence, according to \thmref{23},
for every decorated family $DFTD$ in $\mathcal{V}_{1}$, the probability
that the random group $\Gamma$ of level $l$ and density $d$ fulfills
$DFTD$ is at most

\[
\left(\frac{2k}{(2k-1)^{(\frac{1}{2}-d)l}}\right)^{n}\leq\frac{2k}{(2k-1)^{(\frac{1}{2}-d)l}}\,.
\]

Thus, we conclude that the probability that there exists some decorated
family $DFTD$ in the collection $\mathcal{V}_{1}$ so that $\Gamma$
fulfills $DFTD$, is at most

\[
|\mathcal{V}_{1}|\cdot\frac{2k}{(2k-1)^{(\frac{1}{2}-d)l}}\leq|\mathcal{V}^{0}|\cdot\frac{2k}{(2k-1)^{(\frac{1}{2}-d)l}}\leq\frac{2k\cdot Q(l)^{\ln^{4}l}}{(2k-1)^{(\frac{1}{2}-d)l}}\,\overset{l\rightarrow\infty}{\longrightarrow}0
\]

which, according to \lemref{33}, approaches $0$ as $l$ goes to
infinity.

Recall according to \lemref{21} that general reduction can only increase
the probability of fulfilling the decorated family. Hence, we conclude
that the probability $p_{1}(l)$ that there exists some decorated
family $DF$ in the collection $\mathcal{V}_{1}^{0}$ so that the
random group $\Gamma$ of level $l$ and density $d$ fulfills $DF$,
approaches $0$ as $l$ goes to infinity 
\[
p_{1}(l)\overset{l\rightarrow\infty}{\longrightarrow}0\,.
\]

Hence, we may assume that the random group $\Gamma$ of level $l$
and density $d$, does not fulfill any of the diagrams in the collection
$\mathcal{V}_{1}^{0}$.

Finally, assume that $\Gamma$ admits a solution $y_{0}$ of length
(over $\Gamma$) at most $L_{\Sigma}(y_{0})\leq l\ln^{2}l$ for the
system $\Sigma$. We need to prove that $y_{0}$ admits a lift $\tilde{y}_{0}\in F_{k}$,
so that $\tilde{y}_{0}$ is a solution of $\Sigma$ in $F_{k}$.

Since $\Sigma(y_{0})=1$ in $\Gamma$, and $L_{\Sigma}(y_{0})\leq l\ln^{2}l$,
the group $\Gamma$ must fulfill some reduced decorated family $DFTD$
with decoration $\Sigma$ in the collection $\mathcal{V}^{0}$, so
that the obtained family $FVKD$ of Van-Kampen diagrams consists of
reduced diagrams over $\Gamma$, and their contours are given by $\Sigma(y_{0})$
(i.e., the contour of the $i$-th diagram is $w_{i}(y_{0})$, where
$w_{i}$ is the $i$-th word in $\Sigma$, for all $i=1,...,N(\Sigma)$).

However, since we have already proved that the collection of groups
of level $l$ and density $d$ that satisfy a decorated family in
$\mathcal{V}_{1}^{0}$ is negligible, we can assume that $DFTD$ belongs
to the collection $\mathcal{V}_{t}^{0}$.

Since the decorated family $DFTD$ of topological diagrams, by the
definition of $\mathcal{V}_{t}^{0}$, can be converted to a tree by
means of applying topological modifications (general reduction) on
its diagrams, the family $FVKD$ of Van-Kampen diagrams can be converted
to a tree $T$ by applying the same sequence of modifications on its
diagrams. However, these modifications do not change the values (in
$\Gamma$) of the variables along the boundary.

Hence, the procedure of applying those modifications on $FVKD$, provides
lifts $\tilde{y}_{0}\in F_{k}$ for the elements $y_{0}\in\Gamma$,
so that

\[
\Sigma(\tilde{y}_{0})=1
\]
in $F_{k}$. As claimed. 
\end{proof}
\begin{defn}
\label{def:36} Let $\Sigma(y,a)$ be a system of equations, and let
$l$ be an integer. Let $\pi:F_{k}\rightarrow\Gamma$ be a group presentation.
We say that $\Gamma$ satisfies the \emph{$\Sigma(y,a)$-$lln^{2}$
lifting property}, or briefly $\Sigma$-$lln^{2}$ l.p., if every
solution of $\Sigma$ of length at most $l\ln^{2}l$ in $\Gamma$,
can be lifted to a solution of $\Sigma$ in $F_{k}$. 
\end{defn}

\begin{thm}
\label{thm:37} Let $\mathcal{H}$ be the collection of all finite
systems of equations. Let $d<\frac{1}{2}$. Then, we can order the
systems in $\mathcal{H}$ in an ascending sub-collections $H_{1}\subset H_{2}\subset H_{3}\subset...$,
so that the following property is satisfied.

For all $l$, let $\Gamma_{l}$ be the random group of level $l$
and density $d$. Then, if $p_{l}$ is the probability that the random
group $\Gamma_{l}$ satisfies the $\Sigma$-$lln^{2}$ l.p. for all
$\Sigma\in H_{l}$, then $p_{l}$ converges to $1$ as $l$ approaches
$\infty$. 
\end{thm}

\begin{proof}
Let $q_{l}=1-p_{l}$. Of course, assuming that $H_{l}$ was defined,
$q_{l}$ equals the probability that there exists some system $\Sigma$
in $H_{l}$, so that the random group of level $l$, $\Gamma_{l}$,
does not satisfy the $\Sigma$-$lln^{2}$ l.p..

Given a system $\Sigma\in\mathcal{H}$, we denote by $q_{l}^{\Sigma}$,
the probability that $\Gamma_{l}$ does not satisfy the $\Sigma$-$lln^{2}$
l.p.. Fixing a system $\Sigma$, we already know, according to \thmref{35},
that $q_{l}^{\Sigma}\rightarrow0$ as $l\rightarrow\infty$.

Since there exists only a countable amount of finite systems, we can
write them in a sequence $\Sigma_{1},\Sigma_{2},\Sigma_{3},...$.
Now let $1\leq l_{1}\leq l_{2}\leq l_{3},...$ be an ascending sequence
of integers, so that for all $n\in\mathbb{N}$, we have that $q_{l}^{\Sigma_{1}}+q_{l}^{\Sigma_{2}}...+q_{l}^{\Sigma_{n}}<\frac{1}{n}$
for all $l\geq l_{n}$.

For every integer $l$, let $n$ be the maximal integer for which
$l_{n}\leq l$, and we define $H_{l}$ to be $H_{l}=\{\Sigma_{1},...,\Sigma_{n}\}$.
Of course, the probability $q_{l}$ that there exists some system
$\Sigma$ in $H_{l}$, so that the random group of level $l$, $\Gamma_{l}$,
does not satisfy the $\Sigma$-$lln^{2}$ l.p., satisfies $q_{l}\leq q_{l}^{\Sigma_{1}}+q_{l}^{\Sigma_{2}}...+q_{l}^{\Sigma_{n}}<\frac{1}{n}$.

Hence $q_{l}\rightarrow0$ as $l\rightarrow\infty$. As required. 
\end{proof}

\section{Axes of Elements in Random Groups }\label{sec:Axes-of-Elements-in-Random-Groups}

We recall according to \thmref{27} that for a fixed $d<\frac{1}{2}$,
the random group of level $l$ and density $d$ can be assumed to
be torsion-free $\alpha_{0}l$-hyperbolic, where 
\[
\alpha_{0}=\frac{48}{(1-2d)^{2}}\,.
\]
In what follows, we will use this fact extensively, and hence without
explicit reference. We use the abbreviation nbhd for neighborhood.

\subsection{Conjugates of Small Elements in Random Groups}
\begin{lem}
\label{lem:38} Let $d<\frac{1}{2}$. Let $\Gamma$ be the random
group of level $l$ and density $d$.

Then, with overwhelming probability, if $u,v\in\Gamma$ are two conjugate
elements, then there exists an element $z\in\Gamma$ in the ball of
radius $|u|+|v|+(8\alpha_{0}+\zeta_{0})l$, so that

\[
u=zvz^{-1}\,,
\]
where $\zeta_{0}=\frac{100\alpha_{0}}{(\frac{1}{2}-d)}$. 
\end{lem}

\begin{proof}
Suppose $U,V,Z\in F_{k}$ are geodesic words, so that $U=ZVZ^{-1}$
in $\Gamma$, and $Z$ is of minimal length with this property. Assume
by contradiction that

\[
|Z|>|U|+|V|+8\alpha_{0}l+\zeta_{0}l\,.
\]

Consider a geodesic rectangle $\Pi$ in the Cayley graph of $\Gamma$
that corresponds to the equation $U=ZVZ^{-1}$. Denote by $p_{1},p_{2}$
and $q_{1},q_{2}$ the initial and terminal points of the sides of
$\Pi$ corresponding to $U$ and $V$ respectively.

Since we have assumed that $|Z|>|U|+|V|+8\alpha_{0}l+\zeta_{0}l$,
and since the sides of $\Pi$ are geodesics, there must exist two
points $x,y$ on the side $p_{1}q_{1}$ of $\Pi$, so that $x$ lies
between $p_{1}$ and $y$ say, neither $x$ nor $y$ belongs to the
$2\alpha_{0}l$-nbhd of $p_{1}p_{2}\cup q_{1}q_{2}$, and

\[
d(x,y)=\zeta_{0}l+1\,.
\]

Let $x',y'$ be the dual points of $x,y$ respectively that lie on
the side $p_{2}q_{2}$ of $\Pi$. That is, the points $x',y'$ are
those points on $p_{2}q_{2}$ for which

\begin{align*}
 & d(p_{2},x')=d(p_{1},x)\,,\\
 & d(q_{2},y')=d(q_{1},y)\,.
\end{align*}

We want to show now that $d(y,y'),\:d(x,x')\leq10\alpha_{0}l$. By
the $\alpha_{0}l$-hyperbolicity of $\Gamma$, there exist two points
$z_{x},z_{y}$ on the side $p_{2}q_{2}$ of $\Pi$, so that

\[
d(x,z_{x}),\,d(y,z_{y})\leq2\alpha_{0}l\,.
\]

If, by contradiction, $d(x',z_{x})>2\alpha_{0}l$, then, in particular,
$x'$ lies between $p_{2}$ and $z_{x}$, or $z_{x}$ lies between
$p_{2}$ and $x'$. Assume without loss of generality the first case,
and consider the element $Z'$ of $\Gamma$ that labels the path $p_{1}xz_{x}q_{2}$.
Then, $Z'$ is of length at most

\[
|Z'|\leq d(p_{1},x)+d(x,z_{x})+d(z_{x},q_{2})\leq|Z|-(d(x',z_{x})-2\alpha_{0}l)<|Z|\,.
\]
However, we have that $Z'=ZV$ in $\Gamma$, and hence 
\[
Z'VZ'{}^{-1}=ZVZ^{-1}=U\,,
\]
in $\Gamma$, which contradicts the minimality of (the length of)
$Z$ with this property. Thus, we conclude that

\[
d(x,x'),\:d(y,y')\leq10\alpha_{0}l\,.
\]

Now consider the three geodesic words $u$, $v$, and $z$ that label
geodesic segments $xx'$, $yy'$, and $xy$ respectively in the Cayley
graph of $\Gamma$. By the construction of the dual points $x'$ and
$y'$ of the points $x$ and $y$, we know that

\[
u=zvz^{-1}
\]
in $\Gamma$. Moreover, since

\[
d(x,y)=|z|=\zeta_{0}l+1\,,
\]
if we show that the elements $u,v$ can be conjugated by an element
$z'\in\Gamma$ of length at most

\[
|z'|\leq\zeta_{0}l\,,
\]
then we conclude that the original elements $U,V$ can be conjugated
by an element of length less than $|Z|$, which is a contradiction
that implies the claim.

Indeed, we consider the decoration $\hat{u}\hat{z}\hat{v}^{-1}\hat{z}^{-1}$,
and declare the rigid variables to be $\hat{u},\hat{v}$, and no constant
variables. Let $\mathcal{V}^{0}$ be the collection of all the reduced
decorated topological Van-Kampen diagrams $DTD$ with decoration $\hat{u}\hat{z}\hat{v}^{-1}\hat{z}^{-1}$
and boundary length at most $l\ln l$, that $\Gamma$ could satisfy,
and so that the lengths of the parts of the boundary of $DTD$ corresponding
to the variables $\hat{u},\hat{v}$, sum to at most $2\nu_{0}l$,
where $\nu_{0}=10\alpha_{0}$.

Since the boundary length of each of the diagrams in $\mathcal{V}^{0}$
is at most $l\ln l$, we have according to \thmref{24}, that the
amount of cells $m$ in any diagram in $\mathcal{V}^{0}$ is bounded
by $m\leq\beta_{0}\ln l$, where $\beta_{0}=\frac{1}{\frac{1}{2}-d}$
(see \remref{25}). Thus, according to \thmref{28}, there exists
a polynomial $Q=Q_{d}$, so that

\[
|\mathcal{V}^{0}|\leq Q(l)^{\ln^{2}l}\,.
\]

We generally reduce each of the diagrams in $\mathcal{V}^{0}$, and
denote the collection of the obtained generally reduced diagrams by
$\mathcal{V}$. Recall according to \lemref{21} that general reduction
can only increase the probability of fulfilling the diagram.

For a given generally reduced decorated topological Van-Kampen diagram
$DTD$ in $\mathcal{V}$, with $m$ cells numbered by $1,...,n$,
for some $1\leq n\leq m$, we describe the following iterative procedure.
If $DTD$ contains an isolated numbering, we apply a mining modification
on $DTD$. If the obtained diagram, denoted $DTD$ again, contains
an isolated numbering, we apply another mining modification. We continue
iteratively, until we arrive to a generally reduced decorated diagram
$DTD$ with no isolated numberings, or we apply at most

\[
k_{0}=\left\lceil \frac{\nu_{0}}{\frac{1}{2}-d}\right\rceil +1\,.
\]
iterative mining modifications. Note that $k_{0}=k_{0,d}$ depends
on $d$ only.

We denote by $\mathcal{V}_{1}$ the sub-collection of $\mathcal{V}$
consisting of those diagrams $DTD$ in $\mathcal{V}$ for which the
procedure performs $k_{0}$ iterations.

We denote by $\mathcal{V}_{F}$ the sub-collection of $\mathcal{V}\backslash\mathcal{V}_{1}$
consisting of those diagrams $DTD$ in $\mathcal{V}$ for which the
procedure stops (before applying $k_{0}$ iterations) in a diagram
with no cells (a tree).

And we denote by $\mathcal{V}_{2}$ the remaining diagrams of $\mathcal{V}$;
i.e., $\mathcal{V}_{2}=\mathcal{V}\backslash\mathcal{V}_{1}\cup\mathcal{V}_{F}$
consists of those diagrams in $\mathcal{V}$ for which the procedure
stops (before applying $k_{0}$ iterations) in a diagram with a positive
number of cells. We replace the diagrams in $\mathcal{V}_{2}$ with
the diagrams obtained by applying the procedure on them.

We denote by $\mathcal{V}_{1}^{0},\mathcal{V}_{F}^{0},\mathcal{V}_{2}^{0}$
the sub-collections consisting of those diagrams in $\mathcal{V}^{0}$
who admit an associated diagram in $\mathcal{V}_{1},\mathcal{V}_{F},\mathcal{V}_{2}$
respectively.

By the definition of the procedure above, and according to \lemref{20},
the amount of distinct numberings $n$ in any diagram in the collection
$\mathcal{V}_{1}$, is at least 
\[
n\geq k_{0}=\left\lceil \frac{\nu_{0}}{\frac{1}{2}-d}\right\rceil +1\,.
\]
Hence, according to \thmref{23}, for every diagram $DTD$ in $\mathcal{V}_{1}$
with $n$ distinct numberings, the probability that the random group
$\Gamma$ of level $l$ and density $d$ fulfills $DTD$ is at most
\[
\left(\frac{2k}{(2k-1)^{(\frac{1}{2}-d-\frac{\nu_{0}}{n})l}}\right)^{n}\leq\frac{2k}{(2k-1)^{\gamma_{0}l}}\,,
\]
where $\gamma_{0}=\frac{1}{2}-d-\frac{\nu_{0}}{k_{0}}>0$ (recall
that the rigid parts in $DTD$ sum to at most $2\nu_{0}l$).

On the other hand, by the definition of the collection $\mathcal{V}_{2}$,
every diagram in $\mathcal{V}_{2}$ admits a positive amount $n\geq1$
of distinct numberings, but contains no isolated numberings. Hence,
according to \thmref{23}, for every diagram $DTD$ in $\mathcal{V}_{2}$,
the probability that $\Gamma$ fulfills $DTD$ is at most

\[
\left(\frac{2k}{(2k-1)^{(\frac{1}{2}-d)l}}\right)^{n}\leq\frac{2k}{(2k-1)^{\gamma_{0}l}}\,.
\]

Thus, we conclude that the probability that there exists some diagram
$DTD$ in the collection $\mathcal{V}_{1}\cup\mathcal{V}_{2}$ so
that the random group $\Gamma$ of level $l$ and density $d$ fulfills
$DTD$, is at most

\[
|\mathcal{V}_{1}\cup\mathcal{V}_{2}|\cdot\frac{2k}{(2k-1)^{\gamma_{0}l}}\leq|\mathcal{V}^{0}|\cdot\frac{2k}{(2k-1)^{\gamma_{0}l}}\leq\frac{2k\cdot Q(l)^{\ln^{2}l}}{(2k-1)^{\gamma_{0}l}}\,\overset{l\rightarrow\infty}{\longrightarrow}0
\]

which, according to \lemref{33}, approaches $0$ as $l$ goes to
infinity.

Recall according to \lemref{20} that mining modifications can only
increase the probability of fulfilling the diagram. Hence, we conclude
that the probability $p_{1,2}(l)$ that there exists some diagram
$D$ in the collection $\mathcal{V}_{1}^{0}\cup\mathcal{V}_{2}^{0}$
so that the random group $\Gamma$ of level $l$ and density $d$
fulfills $D$, approaches $0$ as $l$ goes to infinity 
\[
p_{1,2}(l)\overset{l\rightarrow\infty}{\longrightarrow}0\,.
\]

Hence we may assume that the random group $\Gamma$ of level $l$
and density $d$, does not fulfill any of the diagrams in the collection
$\mathcal{V}_{1}^{0}\cup\mathcal{V}_{2}^{0}$.

Now recall our elements $u$, $v$, and $z$ who satisfy 
\[
uzv^{-1}z^{-1}=1
\]
in $\Gamma$, explained in the top of the proof. Since $|u|,|v|\leq\nu_{0}l$,
and since $|z|\leq l\ln l$, we get that the group $\Gamma$ must
fulfill some reduced decorated Van-Kampen diagram $DTD$ with decoration
$\hat{u}\hat{z}\hat{v}^{-1}\hat{z}^{-1}$ in the collection $\mathcal{V}^{0}$,
so that the contour of the obtained (reduced) Van-Kampen diagram $VKD$
is $uzv^{-1}z^{-1}$.

However, since we have already proved that the collection of groups
of level $l$ and density $d$ that satisfy a diagram in $\mathcal{V}_{1}^{0}\cup\mathcal{V}_{2}^{0}$
is negligible, we can assume that $DTD$ belongs to the collection
$\mathcal{V}_{F}^{0}$. Since the topological diagram $DTD$, by the
definition of $\mathcal{V}_{F}^{0}$, can be converted to a tree by
means of applying topological modifications (mining and general reduction)
on it, the Van-Kampen diagram $VKD$ can be converted to a tree $T$
by applying the same sequence of modifications on it. However, these
modifications do not change the values (in $\Gamma$) of the variables
along the boundary.

Hence, the procedure of applying those modifications on $VKD$, provides
lifts $\tilde{u},\tilde{v},\tilde{z}\in F_{k}$ (reduced words - the
obtained words on the boundary of $T$ could be non-reduced, but we
simply reduce them) for the elements $u,v,z\in\Gamma$ respectively,
so that

\[
\tilde{u}\tilde{z}\tilde{v}^{-1}\tilde{z}^{-1}=1
\]
in the free group $F_{k}$. However, according to the definition of
$\mathcal{V}_{F}^{0}$, the tree $T$ is obtained by applying a sequence
of alternate general reduction and mining modifications on $VKD$,
applied for at most $k_{0}=k_{0,d}$ iterations. Now, by definition,
general reduction can change the rigid parts $u,v$ of the boundary,
only by removing self-intersections from them, but a single mining
modification can change a rigid part by eliminating only cells that
share some edge with that part. Thus, the lifts $\tilde{u},\tilde{v}$
that are obtained as a consequence of applying at most $k_{0}=k_{0,d}$
mining modifications, lie, as paths in the Cayley graphof $\Gamma$,
in the $k_{0}l$-nbhd of the geodesic segments $u,v$ in the Cayley
graph of $\Gamma$, respectively.

In particular, each of the lifts $\tilde{u}$ and $\tilde{v}$, as
paths in the Cayley graph of $\Gamma$, is contained in a ball of
radius

\[
\frac{|u|}{2}+k_{0}l,\,\frac{|v|}{2}+k_{0}l\leq2k_{0}l\,.
\]
.

Finally, we consider the obtained lifts $\tilde{u},\tilde{v}\in F_{k}$
for the elements $u,v\in\Gamma$. Write $\tilde{u}=A'eA'{}^{-1}$,
and $\tilde{v}=B'fB'{}^{-1}$ (graphical equality), for reduced words
$A',B'\in F_{k}$ and $e,f\in F_{k}$ cyclically reduced. Then, since
$\tilde{u}$ and $\tilde{v}$ are conjugates in the free group $F_{k}$,
the words $e$ and $f$ must be cyclic permutations to each other.
Write $f=f_{1}f_{2}$ (graphical equality), so that $e=f_{2}f_{1}$.
Then, we have

\[
\tilde{u}=A'(B'f_{1})^{-1}\tilde{v}(A'(B'f_{1})^{-1})^{-1}\,,
\]
which translates in $\Gamma$ as the equality 
\[
u=AB^{-1}v(AB^{-1})^{-1}\,,
\]
where $A$ and $B$ denote geodesic words with $A=A'$ and $B=B'f_{1}$
in $\Gamma$. However, as explained above, a path labeled $\tilde{u}=A'eA'{}^{-1}$
or $\tilde{v}=B'f_{1}f_{2}B'{}^{-1}$ in the Cayley graph of $\Gamma$,
stays inside a ball of radius $2k_{0}l$. In particular, we have that
the length of each of the elements $A,B$ is at most

\[
|A|,|B|\leq4k_{0}l\,.
\]
Hence, the length of the element $z'=AB^{-1}$ is at most

\[
|z'|\leq8k_{0}l\leq\zeta_{0}l\,.
\]
As required. 
\end{proof}

\subsection{Isometries of a Hyperbolic Space}

The readers who are not familiar with hyperbolic geometry in the context
of group theory, are recommended to see the first section in \cite{Remi=000020Coulon=000020-=000020Burnside}. 
\begin{defn}
\label{def:39} Let $\Gamma=\langle a:\mathcal{R}\rangle$ be a torsion-free
non-elementary hyperbolic group, and let $X$ be its Cayley graph
(w.r.t. the generating set $a$). Let $G$ be an element of $\Gamma$. 
\begin{enumerate}
\item The \emph{translation length} of $G$, denoted by $[G]_{X}$, or simply
$[G]$, is defined to be 
\[
[G]=\underset{x\in X}{\min}d(Gx,x)\,.
\]
\item Let $x\in X$. The \emph{asymptotic translation length} of $G$, denoted
by $[G]_{X}^{\infty}$, or simply $[G]^{\infty}$, is defined to be
\[
[G]^{\infty}=\underset{n\rightarrow\infty}{lim}\frac{d(G^{n}x,x)}{n}\,.
\]
\end{enumerate}
Note that the definition of the asymptotic translation length does
not depend on the choice of the point $x$ (see \cite{Non-Positive=000020Curvature},
Chapter II.6, exercises 6.6).

We denote by $\partial X$ the Gromov boundary of $X$ (see \cite{Non-Positive=000020Curvature},
Chapter III.H). Let $G$ be a non-trivial element of $\Gamma$. Each
one of the orbits $\{G^{n}x:n\in\mathbb{Z},\ n\geq0\}$ and $\{G^{n}x:n\in\mathbb{Z},\ n\leq0\}$
has exactly one accumulation point in $\partial X$. These two points
are distinct and we denote them by $G_{+}$ and $G_{-}$ respectively.
$G_{+}$ and $G_{-}$ do not depend on the choice of $x$. They are
the only points of $\partial X$ fixed by the action of $G$. 
\end{defn}

\begin{lem}
\label{lem:40} Let $\Gamma=\langle a:\mathcal{R}\rangle$ be a torsion-free
non-elementary $\delta$-hyperbolic group, and let $X$ be its Cayley
graph (w.r.t. the generating set $a$). Let $G$ be an element of
$\Gamma$. Let $k\in\mathbb{Z}$.

Then 
\begin{enumerate}
\item $[G]^{\infty}\leq[G]\leq[G]^{\infty}+32\delta$; 
\item $[G^{k}]\leq|k|\cdot[G]$; 
\item $[G^{k}]^{\infty}=|k|\cdot[G]^{\infty}$. 
\end{enumerate}
\end{lem}

\begin{proof}
For Point 1, see {[}\cite{Remi=000020Coulon=000020-=000020Burnside},
p.14{]}. Points 2\&3 are a consequence of triangle inequality, and
that the sequence $(\frac{d(G^{kn}x,x)}{kn})_{n}$ is a subsequence
of $(\frac{d(G^{n}x,x)}{n})$ (for $k>0$). 
\end{proof}

\subsection{Axes of Elements in Random Groups}
\begin{defn}
\label{def:41} Let $\Gamma=\langle a:\mathcal{R}\rangle$ be a hyperbolic
group, and let $G\in\Gamma$ be a hyperbolic element. Let $Z,g\in F_{k}$
be two geodesic words in $\Gamma$ so that $G=ZgZ^{-1}$, and the
length of $g$ equals the translation length of $G$ in $\Gamma$;
$|g|=[G]$. Let $X$ be the Cayley graph of $\Gamma$ (w.r.t. the
generating set $a$). 
\end{defn}

\begin{itemize}
\item The path in the Cayley graph of $\Gamma$ induced by the elements
$Zg^{n}$, $n\in\mathbb{Z}$, is called \emph{a nerve for $G$}. The
graphical word $g$, as an element of $F_{k}$, is called \emph{the
word of the nerve}. A nerve of $G$ is denoted usually by $\gamma_{G}$. 
\item \emph{An axis of $G$} is any geodesic connecting the points at infinity
$G_{-}$ and $G_{+}$. An axis of $G$ is usually denoted by $A_{G}$. 
\item Given a real number $C\geq0$, and a subset $U\subset X$, we denote
the closed $C$-neighborhood of $U$ by 
\[
U^{+C}=\{x\in X:d(x,U)\leq C\}\,.
\]
\item If $X$ is $\delta$-hyperbolic, we define the parameter 
\[
D(\Gamma,X)=\sup\{\text{diam}(A_{G}^{+60\delta}\cap A_{H}^{+60\delta}):\,[G,H]\neq1,\,[G],[H]\leq1000\delta\}\,.
\]
This parameter will play an important role when considering limit
actions over an ascending sequences of random groups. 
\end{itemize}
\begin{rem}
$ $ 
\begin{enumerate}
\item A nerve $\gamma_{G}$ of a hyperbolic element $G\in\Gamma$, is a
$[G]$-local geodesic. To see this, Let $Z,g$ be as in \defref{41}.
If, by contradiction, a cyclic shift $g'$ of $g$ is not a geodesic
word in $\Gamma$, then $[g]=[g']\leq|g'|<|g|$, but $[g]=[G]=|g|$,
a contradiction. 
\item If $A_{G}$ is an axis for $G$, then $GA_{G}$ is an axis for $G$
as well. In general, if $H,G$ are hyperbolic elements of $\Gamma$,
and $A_{G}$ and $\gamma_{G}$ are an axis and a nerve with word $w$
for $G$, then $HA_{G}$ is an axis for the element $HGH^{-1}$, and
$H\gamma_{G}$ is a nerve for $HGH^{-1}$ with word $w$. 
\end{enumerate}
\end{rem}

\begin{lem}
\label{lem:42} (\cite{Non-Positive=000020Curvature}, Theorem 1.13)
Let $\Gamma=\langle a:\mathcal{R}\rangle$ be a $\delta$-hyperbolic
group, and let $X$ be its Cayley graph.

Then, the Hausdorff distance between any two $200\delta$-local geodesics
connecting the same two endpoints in the boundary of $X$, $\partial X$,
is at most $5\delta$. In particular, if $G\in\Gamma$ is a hyperbolic
element with $[G]\geq200\delta$, then the $5\delta$-nbhd of any
nerve or axis of $G$, contains all the nerves and all the axes of
$G$. 
\end{lem}

\begin{lem}
\label{lem:43} (\cite{Xiangdong=000020Xie=000020-=000020Growth=000020of=000020relatively=000020hyperbolic})
Let $\Gamma=\langle a:\mathcal{R}\rangle$ be a $\delta$-hyperbolic
group, and let $X$ be its Cayley graph. Let $G\in\Gamma$ be a hyperbolic
element, and let $c$ be an axis of $G$. For all $y\in X$, let $P_{c}(y)$
be a projection of $y$ on $c$ (i.e., $P_{c}(y)$ is a closest point
on $c$ to $y$).

If $x\in c$ and $d(x,Gx)\geq20\delta$, then, for all integers $0<i<j$,
the point $P_{c}(G^{j}x)$ lies between $P_{c}(G^{i}x)$ and $G_{+}$. 
\end{lem}

\begin{cor}
\label{cor:44} Let $d<\frac{1}{2}$, and let $\Gamma$ be the random
group of level $l$ and density $d$.

With overwhelming probability, every two nerves of any two commuting
elements $G_{1},G_{2}\in\Gamma$ are of Hausdorff distance at most
$2\cdot(3350\alpha_{0}+\zeta_{0})l$. And hence, every nerve and every
axis of a hyperbolic element, are of Hausdorff distance at most $\rho_{0}l$,
where $\rho_{0}=2\cdot(3355\alpha_{0}+\zeta_{0})$ (and $\zeta_{0}$
is the one given in \lemref{38}). 
\end{cor}

\begin{proof}
According to \lemref{42}, it suffices to assume that $[G_{1}]=[G]<200\alpha_{0}l$,
and to show then that for every nerve $\gamma_{G}$ of $G$, there
exists an element $H\in\Gamma$ that commutes with $G$ and with translation
length $[H]\geq200\alpha_{0}l$, and there exists a nerve $\gamma_{H}$
for $H$, so that the Hausdorff distance between $\gamma_{G}$ and
$\gamma_{H}$ is at most $(3200\alpha_{0}+\zeta_{0})l$.

Indeed, Without loss of generality, assume that $G=g$ is a geodesic
word with $|g|=[G]<200\alpha_{0}l$, and let $\gamma_{G}$ be a nerve
of $G$. Let $k$ be a large enough integer so that the translation
length of the element $H=g^{k}$ satisfies $200\alpha_{0}l\leq[H]\leq400\alpha_{0}l$.
Let $Z,w$ be geodesic words so that $H=ZwZ^{-1}$, and $|w|=[H]$.
Denote by $\gamma_{H}^{1}$ the nerve of $H$ induced by the elements
$Zw^{n}$, $n\in\mathbb{Z}$. Then, since $[H]\leq400\alpha_{0}l$,
and $Z$ lies on $\gamma_{H}$, we deduce using \lemref{42} and \lemref{43}
that $d(Z,gZ)\leq450\alpha_{0}l$. In other words, if we denote $w'=Z^{-1}gZ$,
then $|w'|\leq450\alpha_{0}l$. Note that $w'{}^{k}=w$.

Hence, according to \lemref{38}, there exists some geodesic word
$z$ of length $|z|\leq(2000\alpha_{0}+\zeta_{0})l$ so that $g=zw'z^{-1}$.
Consider now the nerve $\gamma_{H}^{2}$ induced by the elements $zw^{n}$,
$n\in\mathbb{Z}$, and the path $\gamma_{G}^{2}$ induced by the elements
$zw'{}^{n}$, $n\in\mathbb{Z}$. Note that the Hausdorff distance
between $\gamma_{G}^{1}=\gamma_{G}$ and $\gamma_{G}^{2}$ is at most
$[g]+[w']+|z|\leq(2800\alpha_{0}+\zeta_{0})l$. Hence, it remains
to show that $\gamma_{G}^{2}$ and $\gamma_{H}^{2}$ are of Hausdorff
distance at most $500\alpha_{0}l$.

Indeed, since $[H]\geq200\alpha_{0}l$, we have that $\gamma_{H}^{2}$
is of Hausdorff distance at most $5\alpha_{0}l$ from an axis $c$
of $g$. Hence, for all $n$, the path $g^{n}\gamma_{H}^{2}$ is of
Hausdorff distance at most $5\alpha_{0}l$ from the axis $g^{n}c$
of $g$. Hence, the path $g^{n}\gamma_{H}^{2}$ is of Hausdorff distance
at most $10\alpha_{0}l$ from $\gamma_{H}^{2}$. In particular, $d(zw'{}^{n},\gamma_{H}^{2})=d(g^{n}z,\gamma_{H}^{2})\leq10\alpha_{0}l$.
Thus, we deduce that $\gamma_{G}^{2}$ is contained in the $(10\alpha_{0}l+[w'])$-nbhd
of $\gamma_{H}^{2}$. On the other hand, of course, $\gamma_{H}^{2}$
is contained in the $[w]$-nbhd of $\gamma_{G}^{2}$. As required. 
\end{proof}
\begin{lem}
\label{lem:45} Let $d<\frac{1}{2}$, and let $\Gamma$ be the random
group of level $l$ and density $d$. Let $X$ be the Cayley graph
of $\Gamma$ (w.r.t. the fixed generating set $a$).

Denote by $p_{l}$ the probability that $D(\Gamma,X)\ge10l\cdot\ln l$.
Then 
\[
p_{l}\overset{l\rightarrow\infty}{\longrightarrow}0\,.
\]
\end{lem}

\begin{proof}
Assume that $G,H\in\Gamma$ so that $[G],[H]\leq1000l$ and

\[
\text{diam}(A_{G}^{+60\alpha_{0}l}\cap A_{H}^{+60\alpha_{0}l})\geq10l\cdot\ln l\,,
\]
where $A_{G}$ and $A_{H}$ are two axes of $G$ and $H$ respectively.
According to \lemref{42}, we can assume that

\[
200\alpha_{0}l\leq[G],[H]\leq1000\alpha_{0}l\,,
\]
and that

\[
\text{diam}(\gamma_{G}^{+65\alpha_{0}l}\cap\gamma_{H}^{+65\alpha_{0}l})\geq10l\cdot\ln l\,,
\]
where $\gamma_{G}$ and $\gamma_{H}$ are two nerves for $G$ and
$H$ respectively. Write

\[
G=YgY^{-1},\qquad H=ZhZ^{-1}\,,
\]
so that $Y,Z,g,h\in F_{k}$ are geodesic words, and the words $g,h$
are the words of the nerves $\gamma_{G}$ and $\gamma_{H}$ respectively.

Using the $\alpha_{0}l$-hyperbolicity of $\Gamma$, we can find two
pairs of points $p_{1},p_{2}$ and $q_{1},q_{2}$, in the Cayley graph
of $\Gamma$, representing elements of the form $Zh^{s_{1}},Zh^{s_{2}}$
and $Yg^{r_{1}},Yg^{r_{2}}$ respectively, $r_{i},s_{i}\in\mathbb{Z}$,
so that

\begin{align*}
 & d(p_{1},p_{2}),\,d(q_{1},q_{2})\geq9l\ln l\,,\\
 & d(p_{1},q_{1}),\,d(p_{2},q_{2})\leq\nu_{0}l\,,
\end{align*}
where $\nu_{0}=1130\alpha_{0}$. Without loss of generality, we assume
that $r_{2}-r_{1},\,s_{2}-s_{1}\geq0$ (if not, we replace $G$ or
$H$ by $G^{-1}$ or $H^{-1}$).

We consider two geodesic words $x,y\in F_{k}$, so that 
\begin{align*}
 & x=p_{1}^{-1}q_{1}\,,\\
 & y=p_{2}^{-1}q_{2}\,.
\end{align*}
That is, $x,y$ labels geodesics $[p_{1},q_{1}],\,[p_{2},q_{2}]$
in the Cayley graph of $\Gamma$.

Under the assumption that 
\[
\text{diam}(A_{G}^{+60\alpha_{0}l}\cap A_{H}^{+60\alpha_{0}l})\geq10l\cdot\ln l\,,
\]
we need to prove that $G,H$ commute in $\Gamma$. Hence, it suffices
to prove that the elements $xgx^{-1}$ and $h$ commute, since, for
example, this implies that the axes (the whole axes) of $G$ and $H$
are of bounded Hausdorff distance from each other.

We consider the rectangle 
\[
xg^{r}y^{-1}h^{-s}
\]
in the Cayley graph of $\Gamma$, where $r=r_{2}-r_{1}$, and $s=s_{2}-s_{1}$.
Since

\[
d(p_{1},p_{2}),\,d(q_{1},q_{2})\geq9l\ln l\,,
\]
we have that 
\[
r,s\geq\frac{\ln l}{1000\alpha_{0}}\,.
\]
And since the nerve $\gamma_{G}$ is a $200\alpha_{0}l$-local geodesic,
and hence a $(\frac{204}{196},2\alpha_{0}l)$-quasi-geodesic, we have
that

\begin{align*}
 & r\cdot200\alpha_{0}l\leq r\cdot[g]\leq\frac{204}{196}\cdot|g^{r}|+2\alpha_{0}l\leq\frac{204}{196}\cdot10l\ln l+2\alpha_{0}l\\
\implies & r\leq\frac{\frac{204}{196}\cdot10\ln l+2\alpha_{0}}{200\alpha_{0}}\leq\ln l\,.
\end{align*}
Similarly, 
\[
s\le\ln l\,.
\]
Hence, we have that 
\[
\ln l\geq r,s\geq\frac{\ln l}{1000\alpha_{0}}\,.
\]

For every $\ln l\geq i,j\geq\frac{\ln l}{1000\alpha_{0}}$, Let $\mathcal{V}^{i,j}$
be the collection of all the reduced decorated topological Van-Kampen
diagrams $DTD$ with decoration $\hat{x}\hat{g}^{i}\hat{y}^{-1}\hat{h}^{-j}$,
and boundary length at most $4000\alpha_{0}l\ln l$, that $\Gamma$
could satisfy, and so that the lengths of the parts of the boundary
of $DTD$ corresponding to the variables $\hat{x},\hat{y}$, sum to
at most $2\nu_{0}l$. Let $\mathcal{V}^{0}$ be the union of all of
these collections,

\[
\mathcal{V}^{0}=\underset{i,j}{\cup}\mathcal{V}^{i,j}\,.
\]
For each diagram $DTD$ in $\mathcal{V}^{0}$, we declare $\hat{x}$
and $\hat{y}$ to be the rigid variables of the decoration, and no
constant variables.

Since the boundary length of each of the diagrams in $\mathcal{V}^{0}$
is at most $4000\alpha_{0}l\ln l$, we have according to \thmref{24},
that the amount of cells $m$ in any diagram in $\mathcal{V}^{0}$
is bounded by $m\leq\beta_{0}\ln l$, where $\beta_{0}=\frac{4000\alpha_{0}}{\frac{1}{2}-d}$
(see \remref{25}). Moreover, the decoration $\hat{x}\hat{g}^{i}\hat{y}^{-1}\hat{h}^{-j}$
of each diagram in $\mathcal{V}^{0}$ is of length at most $2\ln l+2$.
Hence, according to \thmref{28}, for all $i,j$, there exists a polynomial
$Q^{i,j}=Q_{d}^{i,j}$, so that

\[
|\mathcal{V}^{i,j}|\leq Q^{i,j}(l)^{\ln^{2}l}\,.
\]
And hence, there exists a polynomial $Q=Q_{d}$, so that

\[
|\mathcal{V}^{0}|\leq Q(l)^{\ln^{2}l}\,.
\]

We generally reduce each of the diagrams in $\mathcal{V}^{0}$, and
denote the collection of the obtained generally reduced diagrams by
$\mathcal{V}$. Recall according to \lemref{21} that general reduction
can only increase the probability of fulfilling the diagram.

For a given generally reduced decorated topological Van-Kampen diagram
$DTD$ in $\mathcal{V}$, with $m$ cells numbered by $1,...,n$,
for some $1\leq n\leq m$, we describe the following iterative procedure.
If $DTD$ contains an isolated numbering, we apply a mining modification
on $DTD$. If the obtained diagram, denoted $DTD$ again, contains
an isolated numbering, we apply another mining modification. We continue
iteratively, until we arrive to a generally reduced decorated diagram
$DTD$ with no isolated numberings, or we apply at most

\[
k_{0}=\left\lceil \frac{\nu_{0}}{\frac{1}{2}-d}\right\rceil +1
\]
iterative mining modifications. Note that $k_{0}=k_{0,d}$ depends
on $d$ only.

We denote by $\mathcal{V}_{1}$ the sub-collection of $\mathcal{V}$
consisting of those diagrams $DTD$ in $\mathcal{V}$ for which the
procedure performs $k_{0}$ iterations.

We denote by $\mathcal{V}_{ab}$ the sub-collection of $\mathcal{V}\backslash\mathcal{V}_{1}$
consisting of those diagrams $DTD$ in $\mathcal{V}$ for which the
procedure stops (before applying $k_{0}$ iterations) in a diagram
with no cells (a tree).

And we denote by $\mathcal{V}_{2}$ the remaining diagrams of $\mathcal{V}$;
i.e., $\mathcal{V}_{2}=\mathcal{V}\backslash\mathcal{V}_{1}\cup\mathcal{V}_{ab}$
consists of those diagrams in $\mathcal{V}$ for which the procedure
stops (before applying $k_{0}$ iterations) in a diagram with a positive
number of cells. We replace the diagrams in $\mathcal{V}_{2}$ by
the obtained diagrams by applying the procedure on them.

We denote by $\mathcal{V}_{1}^{0},\mathcal{V}_{ab}^{0},\mathcal{V}_{2}^{0}$
the sub-collections consisting of those diagrams in $\mathcal{V}^{0}$
who admit an associated diagram in $\mathcal{V}_{1},\mathcal{V}_{ab},\mathcal{V}_{2}$
respectively.

By the definition of the procedure above, and according to \lemref{20},
the amount of distinct numberings $n$ in any diagram in the collection
$\mathcal{V}_{1}$, is at least 
\[
n\geq k_{0}=\left\lceil \frac{\nu_{0}}{\frac{1}{2}-d}\right\rceil +1\,.
\]
Hence, according to \thmref{23}, for every diagram $DTD$ in $\mathcal{V}_{1}$
with $n$ distinct numberings, the probability that the random group
$\Gamma$ of level $l$ and density $d$ fulfills $DTD$ is at most
\[
\left(\frac{2k}{(2k-1)^{(\frac{1}{2}-d-\frac{\nu_{0}}{n})l}}\right)^{n}\leq\frac{2k}{(2k-1)^{\gamma_{0}l}}\,,
\]
where $\gamma_{0}=\frac{1}{2}-d-\frac{\nu_{0}}{k_{0}}>0$ (recall
that the rigid parts in $DTD$ sum to at most $2\nu_{0}l$).

On the other hand, by the definition of the collection $\mathcal{V}_{2}$,
every diagram in $\mathcal{V}_{2}$ admits a positive amount $n\geq1$
of distinct numberings, but contains no isolated numberings. Hence,
according to \thmref{23}, for every diagram $DTD$ in $\mathcal{V}_{2}$,
the probability that $\Gamma$ fulfills $DTD$ is at most

\[
\left(\frac{2k}{(2k-1)^{(\frac{1}{2}-d)l}}\right)^{n}\leq\frac{2k}{(2k-1)^{\gamma_{0}l}}\,.
\]

Thus, we conclude that the probability that there exists some diagram
$DTD$ in the collection $\mathcal{V}_{1}\cup\mathcal{V}_{2}$ so
that the random group $\Gamma$ of level $l$ and density $d$ fulfills
$DTD$, is at most

\[
|\mathcal{V}_{1}\cup\mathcal{V}_{2}|\cdot\frac{2k}{(2k-1)^{\gamma_{0}l}}\leq|\mathcal{V}^{0}|\cdot\frac{2k}{(2k-1)^{\gamma_{0}l}}\leq\frac{2k\cdot Q(l)^{\ln^{2}l}}{(2k-1)^{\gamma_{0}l}}\,\overset{l\rightarrow\infty}{\longrightarrow}0
\]

which, according to \lemref{33}, approaches $0$ as $l$ goes to
infinity.

Recall according to \lemref{20} that mining modifications can only
increase the probability of fulfilling the diagram. Hence, we conclude
that the probability $p_{1,2}(l)$ that there exists some diagram
$D$ in the collection $\mathcal{V}_{1}^{0}\cup\mathcal{V}_{2}^{0}$
so that the random group $\Gamma$ of level $l$ and density $d$
fulfills $D$, approaches $0$ as $l$ goes to infinity 
\[
p_{1,2}(l)\overset{l\rightarrow\infty}{\longrightarrow}0\,.
\]

Hence we may assume that the random group $\Gamma$ of level $l$
and density $d$, does not fulfill any of the diagrams in the collection
$\mathcal{V}_{1}^{0}\cup\mathcal{V}_{2}^{0}$.

Now recall our rectangle 
\[
xg^{r}y^{-1}h^{-s}
\]
in the Cayley graph of $\Gamma$, explained in the top of the proof.
Since $|x|,|y|\leq\nu_{0}l$, since $|g|,|h|\leq1000\alpha_{0}l$,
and since $\ln l\geq r,s\geq\frac{\ln l}{1000\alpha_{0}}$, the group
$\Gamma$ must fulfill some reduced decorated Van-Kampen diagram $DTD$
with decoration $\hat{x}\hat{g}^{r}\hat{y}^{-1}\hat{h}^{-s}$ in the
collection $\mathcal{V}^{0}$, so that the contour of the obtained
(reduced) Van-Kampen diagram $VKD$ is $xg^{r}y^{-1}h^{-s}$.

However, since we have already proved that the collection of groups
of level $l$ and density $d$ that satisfy a diagram in $\mathcal{V}_{1}^{0}\cup\mathcal{V}_{2}^{0}$
is negligible, we can assume that $DTD$ belongs to the collection
$\mathcal{V}_{ab}^{0}$. Since the topological diagram $DTD$, by
the definition of $\mathcal{V}_{ab}^{0}$, can be converted to a tree
by means of applying topological modifications (mining and general
reduction) on it, the Van-Kampen diagram $VKD$ can be converted to
a tree $T$ by applying the same sequence of modifications on it.
However, these modifications do not change the values (in $\Gamma$)
of the variables along the boundary.

Hence, the procedure of applying those modifications on $VKD$, provides
lifts $\tilde{x},\tilde{y},\tilde{g},\tilde{h}\in F_{k}$ (reduced
words - the obtained words on the boundary of $T$ could be non-reduced,
but we simply reduce them) for the elements $x,y,g,h\in\Gamma$ respectively,
so that

\[
\tilde{x}\tilde{g}^{r}\tilde{y}^{-1}\tilde{h}^{-s}=1
\]
in the free group $F_{k}$. However, according to the definition of
$\mathcal{V}_{ab}^{0}$, the tree $T$ is obtained by applying a sequence
of alternate general reduction and mining modifications on $VKD$,
applied for at most $k_{0}=k_{0,d}$ iterations. Now, by definition,
general reduction can change the rigid parts $x,y$ of the boundary,
only by removing self-intersections from them, but a single mining
modification can change a rigid part by eliminating only cells that
share some edge with that part. Thus, the lifts $\tilde{x},\tilde{y}$
that are obtained as a consequence of applying at most $k_{0}=k_{0,d}$
mining modifications, lie, as paths in the Cayley graphof $\Gamma$,
in the $k_{0}l$-nbhd of the geodesic segments $x,y$ in the Cayley
graph of $\Gamma$, respectively.

In particular, each of the lifts $\tilde{x}$ and $\tilde{y}$, as
paths in the Cayley graph of $\Gamma$, is contained in a ball of
radius

\[
\frac{|x|}{2}+k_{0}l,\,\frac{|y|}{2}+k_{0}l\leq2k_{0}l\,.
\]
.

Finally, we consider the obtained lifts $\tilde{g},\tilde{h}\in F_{k}$
for the elements $g,h\in\Gamma$. Write $\tilde{g}=Y'uY'{}^{-1}$,
and $\tilde{h}=Z'wZ'{}^{-1}$ (graphical equality), for reduced words
$Y',Z'\in F_{k}$ and $u,w\in F_{k}$ cyclically reduced. Since the
translation length of $g$ in $\Gamma$ satisfies $[g]\geq200\alpha_{0}l$,
we have that for all $k\in\mathbb{Z}$, the length of $u^{k}$ in
$\Gamma$ satisfies: 
\[
|u^{k}|\geq[u^{k}]=[g^{k}]\geq[g^{k}]_{\infty}=k\cdot[g]_{\infty}\geq k\cdot168\alpha_{0}l\,.
\]
Similarly, 
\[
|w^{k}|\geq k\cdot168\alpha_{0}l\,.
\]

In particular, a path induced by the powers of the word $u$ in the
Cayley graph of $\Gamma$, can count at most

\[
\left\lceil \frac{4k_{0}}{168\alpha_{0}}\right\rceil =n_{0}=n_{0,d}
\]
``vertices'' inside a ball of radius $2k_{0}l$ (by vertices we
mean the points on the path induced by the powers $w^{n}$, $n\in\mathbb{Z}$).
Similarly, at most $n_{0}$ powers of $w$ can lie inside a ball of
radius $2k_{0}l$ in the Cayley graph of $\Gamma.$

In particular, since each of the lifts $\tilde{x}$ and $\tilde{y}$,
as paths in the Cayley graph of $\Gamma$, is contained in a ball
of radius $2k_{0}l$, we deduce that neither $\tilde{x}$, nor $\tilde{y}$,
admits a sub-word of the form $u^{\pm(n_{0}+1)},w^{\pm(n_{0}+1)}$.
Hence, and since 
\[
r,s\geq\frac{\ln l}{1000\alpha_{0}}\geq110n_{0}\,,
\]
we get according to \lemref{32}, that the words $\tilde{x}\tilde{g}\tilde{x}^{-1}$
and $\tilde{h}$ commute. Hence, of course, the elements $xgx^{-1}$
and $h$ of $\Gamma$, are commutative. As claimed. 
\end{proof}

\section{Hyperbolic Geometry over Random Groups }\label{sec:Hyperbolic-Geometry-over-Random-Groups}
\begin{lem}
\label{lem:46} (See Lemma 2.10 in \cite{Remi=000020Coulon=000020-=000020Burnside})
Let $\Gamma=\langle a:\mathcal{R}\rangle$ be a $\delta$-hyperbolic
group, and let $X$ be its Cayley graph. Let $c$ be a geodesic line
in $X$. Let $x,y\in\Gamma$ be two points, and let $p,q\in c$ be
projections on $c$ for $x,y$ respectively.

If $d(p,q)\geq50\delta$, then

\[
d(x,y)+50\delta\geq d(x,p)+d(p,q)+d(q,y)\,.
\]
\end{lem}

\begin{proof}
We consider the rectangle whose (geodesic) sides are $[x,y],\,[x,p],\,[p,q],\,[y,q]$.

First note that if a point on $[x,p]$ is $2\delta$-close to a point
on $[y,q]$, then $p,q$ must be $10\delta$-close to each other,
a contradiction.

Now if the $2\delta$-nbhd of a point $r$ on $[x,p]$ contains a
point on $c$, then, since $p$ is a projection of $x$ on $c,$ we
get that $d(r,p)\leq2\delta$. Similarly, if the $2\delta$-nbhd of
a point $s$ on $[y,q]$ contains a point on $c$, then $d(s,q)\leq2\delta$.

These imply, using the $\delta$-hyperbolicity, that the three segments
$[x,p],\,[p,q],\,[y,q]$ are contained in the $4\delta$-nbhd of $[x,y]$.
In particular, there exists two points $t_{1},t_{2}$ on $[x,y]$
that are $4\delta$-close to $p,q$ respectively. Hence,

\begin{align*}
 & d(x,p)\leq d(x,t_{1})+d(t_{1},p)\leq d(x,t_{1})+4\delta\,,\\
 & d(p,q)\leq d(p,t_{1})+d(t_{1},t_{2})+d(t_{2},q)\leq d(t_{1},t_{2})+8\delta\,,\\
 & d(q,y)\leq d(q,t_{2})+d(t_{2},y)\leq d(t_{2},y)+4\delta\,.
\end{align*}
Thus, 
\[
d(x,p)+d(p,q)+d(q,y)\leq d(x,y)+16\delta\,.
\]
\end{proof}
\begin{lem}
\label{lem:47} (See Proposition 2.24 in \cite{Remi=000020Coulon=000020-=000020Burnside})
Let $d<\frac{1}{2}$, and let $\Gamma$ be the random group of level
$l$ and density $d$. Let $X$ be the Cayley graph of $\Gamma$ (w.r.t.
the fixed generating set $a$).

Then, with probability that tends to $1$, the following property
is satisfied in $\Gamma$.

Let $G\in\Gamma$ be a hyperbolic element, and let $A_{G}$ be an
axis for $G$. Let $x\in\Gamma$, and assume that $d(Gx,x)\leq A+[G]$,
for some real number $A\geq0$.

Then, 
\begin{enumerate}
\item If $[G]\geq200\alpha_{0}l$, then 
\[
d(x,A_{G})\leq\frac{A}{2}+50\alpha_{0}l\,.
\]
\item If $[G]<200\alpha_{0}l$, then 
\[
d(x,A_{G})\leq A+(7500\alpha_{0}+3\zeta_{0})l\,,
\]
where $\zeta_{0}$ is given in \lemref{38}. 
\end{enumerate}
\end{lem}

\begin{proof}
Since $d(Gx,x)=d(1,x^{-1}Gx)$, and $d(x,A_{G})=d(1,A_{x^{-1}Gx})$,
and $[x^{-1}Gx]=[G]$, we can assume that $x=1$. For simplicity we
write $\delta=\alpha_{0}l$.

Let $p,q\in A_{G}$ be projections for $x=1$ and $Gx=G$ on $A_{G}$,
and consider a geodesic segments $[x,p],\,[Gx,q]$ between $x$ and
$p$, and $Gx$ and $q$. Since $p\in A_{G}$, we have that $Gp$
belongs to the $5\delta$-nbhd of $A_{G}$. Hence, 
\[
[G]\leq d(p,Gp)\leq[G]+15\delta\,.
\]
Moreover, since $p$ is a projection for $x$ on $A_{G}$, we have
that $Gp$ is a projection for $Gx$ on $GA_{G}$. But $A_{G}$ and
$GA_{G}$ are both axes for $G$, and in particular the Hausdorff
distance between them is at most $5\delta$. Hence, the projection
$q$ of $Gx$ on $A_{G}$ and the projection $Gp$ of $Gx$ on $GA_{G}$,
are $30\delta$-close to each other 
\[
d(q,Gp)\leq30\delta\,.
\]
Hence, 
\begin{align*}
 & d(p,q)\leq d(p,Gp)+d(Gp,q)\leq[G]+45\delta\,,\\
 & d(p,q)\geq d(p,Gp)-d(Gp,q)\geq[G]-15\delta\,.
\end{align*}
In particular, if $[G]\geq200\delta$, then $d(p,q)\geq50\delta$,
and thus by \lemref{46}, we get that

\[
d(x,p)+d(Gx,q)\leq d(x,Gx)-d(p,q)+50\delta\leq A+65\delta\,.
\]
But $d(Gx,q)\geq d(Gx,Gp)-30\delta=d(x,p)-30\delta$, so we conclude
that 
\[
d(x,p)\leq\frac{A}{2}+50\delta\,,
\]
in case $[G]\geq200\delta$.

Now assume that $[G]<200\delta$.

Write $G=zgz^{-1}$, for geodesic words $z,g\in F_{k}$, with $|g|=[G]$.
Since

\[
|G|=d(x,Gx)\leq[G]+A\leq200\delta+A\,,
\]
and $|g|=[G]\leq200\delta$, we can assume by \lemref{38} that

\[
|z|\leq A+(408\alpha_{0}+\zeta_{0})l\,,
\]
where $\zeta_{0}=\frac{100\alpha_{0}}{(\frac{1}{2}-d)}$. In particular,
$x=1$ is of distance at most $A+(408\alpha_{0}+\zeta_{0})l$ from
the nerve $\gamma_{G}$ of $G$ induced by the elements $zg^{n}$,
$n\in\mathbb{Z}$.

Hence, according to \corref{44}, we deduce that $x=1$ is of distance
at most $A+(7500\alpha_{0}+3\zeta_{0})l$ from $A_{G}$. 
\end{proof}
\begin{thm}
\label{thm:48} (See Proposition 2.41 in \cite{Remi=000020Coulon=000020-=000020Burnside})
Let $d<\frac{1}{2}$, and let $\Gamma$ be the random group of level
$l$ and density $d$. Let $X$ be the Cayley graph of $\Gamma$ (w.r.t.
the fixed generating set $a$).

Then, with probability that tends to $1$, the following property
is satisfied in $\Gamma$.

Let $G,H\in\Gamma$ be two non-commuting elements, and let $A_{G},A_{H}$
be two axes for them respectively. Then

\[
\text{diam}(A_{G}^{+60\alpha_{0}l}\cap A_{H}^{+60\alpha_{0}l})\leq4\cdot([G]+[H])+D(\Gamma,X)+10\alpha_{0}\cdot l\ln l\,.
\]
\end{thm}

\begin{proof}
For simplicity we write $\delta=\alpha_{0}l$.

According to the definition of $D(\Gamma,X)$, we can assume that
$[G]>1000\delta$. By raising $H$ to a suitable power if necessary,
we can assume that $[H]\geq200\delta$, and prove then that 
\[
\text{diam}(A_{G}^{+60\delta}\cap A_{H}^{+60\delta})\leq4\cdot([G]+[H])+D(\Gamma,X)+9\delta\ln l\,.
\]

Indeed, assume by contradiction that 
\[
\text{diam}(A_{G}^{+60\delta}\cap A_{H}^{+60\delta})>4\cdot([G]+[H])+D(\Gamma,X)+9\delta\ln l\,.
\]

By the $\alpha_{0}l$-hyperbolicity, we get that there exists two
sub-segments $I_{1}=[x_{1},y_{1}]\subset A_{G}$, and $I_{2}=[x_{2},y_{2}]\subset A_{H}$,
whose Hausdorff distance from each other is at most $5\delta$, and
their lengths satisfy $|I_{1}|,|I_{2}|\geq6\cdot([G]+[H])+D(\Gamma,X)+8\delta\ln l$.
And again by $\delta$-hyperbolicity, these two sub-segments admit
two sub-segments $I_{1}'=[x_{1}',y_{1}']$, and $I_{2}'=[x_{2}',y_{2}']$,
whose Hausdorff distance is at most $5\delta$ from each other, their
lengths satisfy $|I_{1}'|,|I_{2}'|\geq2\cdot([G]+[H])+D(\Gamma,X)+7\delta\ln l$,
and whose endpoints are of distances at least $[G]+[H]+1000\delta$
from the endpoints of $I_{1}$ and $I_{2}$, respectively.

Let $p_{1},p_{2}$ be two points on $I_{1}',I_{2}'$ respectively,
so that

\[
d(p_{1},p_{2})\leq5\delta\,.
\]
Then, Since $p_{1}\in A_{G}$, we have 
\[
d(Gp_{1},p_{1})\leq10\delta+[G]\,.
\]
Let $q_{1}$ be a projection of $Gp_{1}$ on $A_{G}$. Since $Gp_{1}$
belongs to the $5\delta$-nbhd of $A_{G}$, we deduce that

\[
[G]-15\delta\leq d(p_{1},q_{1})\leq[G]+15\delta\,.
\]
In particular, $q_{1}$ belongs to $I_{1}$, and hence to the $5\delta$-nbhd
of $I_{2}$. Let $q_{2}$ be a projection of $q_{1}$ on $I_{2}$.
Note that

\[
d(q_{1},q_{2})\leq5\delta,\quad d(Gp_{1},q_{2})\leq10\delta\,.
\]

Hence 
\begin{align*}
 & d(p_{2},q_{2})\leq d(p_{2},p_{1})+d(p_{1},q_{1})+d(q_{1},q_{2})\leq[G]+25\delta\,,\\
 & d(p_{2},q_{2})\geq d(p_{1},q_{1})-d(p_{2},p_{1})-d(q_{1},q_{2})\geq[G]-25\delta\,.
\end{align*}

Now since $Gp_{1}$ belongs to the $10\delta$-nbhd of $A_{H}$, we
get that

\[
[H]-30\delta\leq d(HGp_{1},Gp_{1})\leq[H]+30\delta\,,
\]
and that $HGp_{1}$ belongs to the $15\delta$-nbhd of $A_{H}$. Let
$r_{2}$ be a projection of $HGp_{1}$ on $A_{H}$. Of course

\[
d(HGp_{1},r_{2})\leq15\delta\,,
\]
and we deduce that

\begin{align*}
 & d(q_{2},r_{2})\leq d(q_{2},Gp_{1})+d(Gp_{1},HGp_{1})+d(HGp_{1},r_{2})\leq[H]+55\delta\,,\\
 & d(q_{2},r_{2})\geq d(Gp_{1},HGp_{1})-d(q_{2},Gp_{1})-d(HGp_{1},r_{2})\geq[H]-55\delta\,.
\end{align*}

By replacing $H$ by $H^{-1}$ if necessary, we may assume that $q_{2}$
lies between $p_{2}$ and $r_{2}$ (see \lemref{43}). Hence,

\[
[G]+[H]-80\delta\leq d(p_{2},r_{2})=d(p_{2},q_{2})+d(q_{2},r_{2})\leq[G]+[H]+80\delta\,.
\]

In particular $r_{2}$ belongs to $I_{2}$, and hence to the $5\delta$-nbhd
of $I_{1}$. Since $d(HGp_{1},r_{2})\leq15\delta$, we have that $HGp_{1}$
belongs to the $20\delta$-nbhd of $I_{1}$. Let $r_{1}$ be a projection
of $HGp_{1}$ on $I_{1}$. Then,

\[
d(HGp_{1},r_{1})\leq20\delta\,,\quad d(r_{1},r_{2})\leq35\delta\,.
\]

Then, 
\begin{align*}
 & d(p_{1},r_{1})\leq d(p_{1},p_{2})+d(p_{2},r_{2})+d(r_{2},r_{1})\leq[G]+[H]+120\delta\,,\\
 & d(p_{1},r_{1})\geq d(p_{2},r_{2})-d(p_{1},p_{2})-d(r_{2},r_{1})\geq[G]+[H]-120\delta\,.
\end{align*}

Now by applying the same argument again, we deduce that $GHp_{2}$
lies in the $15\delta$-nbhd of $A_{G}$, and if $r_{1}'$ is a projection
of $GHp_{2}$ on $A_{G}$, then

\[
[G]+[H]-80\delta\leq d(p_{1},r_{1}')\leq[G]+[H]+80\delta\,.
\]
Hence, 
\[
d(GHp_{2},HGp_{1})\leq d(GHp_{2},r_{1}')+d(r_{1}',r_{1})+d(r_{1},HGp_{1})\leq190\delta\,.
\]
Hence, 
\[
d(GHp_{1},HGp_{1})\leq d(GHp_{1},HGp_{2})+d(HGp_{2},HGp_{1})\leq195\delta\,.
\]
Thus, the commutator $u=G^{-1}H^{-1}GH$ satisfies that 
\[
d(up,p)\leq195\delta
\]
for every $p\in I_{1}'$, which implies also that 
\[
[u]\leq195\delta\,.
\]

According to the assumption, $u\neq1$. Now since $d(up,p)\leq195\delta$
for all $p\in I_{1}'$, we get by \lemref{47} that $I_{1}'$ lies
in the $(10405\alpha_{0}+3\zeta_{0})l=\lambda_{0}\delta$-nbhd of
an axis $A_{u}$ of $u$, where $\lambda_{0}=\frac{(10405\alpha_{0}+3\zeta_{0})}{\alpha_{0}}$.

Since 
\[
|I_{1}'|\geq2[G]+D(\Gamma,X)+7\delta\ln l\,,
\]
then, by the $\delta$-hyperbolicity of $\Gamma$, there exist sub-segments
$J_{1}$ and $J_{2}$ of $I_{1}'\subset A_{G}$ and $A_{u}$ respectively,
so that the Hausdorff distance between $J_{1}$ and $J_{2}$ is at
most $5\delta$, and $|J_{1}|,|J_{2}|\geq2[G]+D(\Gamma,X)+6\delta\ln l$.
Say $G$ translates ``right'', and denote by $e$ the ``left''
endpoint of $J_{1}$ (see \lemref{43}). Let $f$ be a point on $J_{1}$
so that

\[
D(\Gamma,X)+200\delta\leq d(e,f)\leq[G]+D(\Gamma,X)+200\delta\,.
\]
Let $p_{e},p_{f}$ be projections of $e,f$ on $J_{2}$ respectively.
Then, 
\[
d(p_{e},p_{f})\geq d(e,f)-10\alpha_{0}l\geq D(\Gamma,X)+190\delta\,.
\]
Now, since $|J_{1}|\geq2[G]+D(\Gamma,X)+6\delta\ln l$ and $G$ translates
``right'', each of $Ge$ and $Gf$ is of $10\delta$-distance from
$J_{1}$. Hence each of $Gp_{e}$ and $Gp_{f}$ is of $15\delta$-distance
from $J_{1}$, and hence, each of $Gp_{e}$ and $Gp_{f}$ is of $20\delta$-distance
from $J_{2}$. In particular,

\[
Gp_{e},Gp_{f}\in A_{u}^{+60\delta}\,.
\]

But $G[p_{e},p_{f}]$, where $[p_{e},p_{f}]$ is the corresponding
sub-segment of $A_{u}$, is a sub-segment of an axis $A_{GuG^{-1}}=GA_{u}$
for the element $GuG^{-1}$, whose translation length satisfies $[GuG^{-1}]=[u]\leq1000\delta$.
Thus, by the definition of $D(\Gamma,X)$, either $|G[p_{e},p_{f}]|\leq D(\Gamma,X)$,
or otherwise $GuG^{-1}$ commutes with $u$. We already know that
$|G[p_{e},p_{f}]|=d(p_{e},p_{f})\geq D(\Gamma,X)+190\delta$, so we
conclude that the elements $GuG^{-1}$ and $u$ commute. Hence, Since
$G,u$ are hyperbolic elements in the hyperbolic group $\Gamma$,
we conclude that $G$ commutes with $u$. But $u=G^{-1}H^{-1}GH$,
and thus, the elements $G$ and $H^{-1}GH$ commute. And again, this
implies that $G$ and $H$ commute. 
\end{proof}
\begin{defn}
Let $d<\frac{1}{2}$. An \emph{ascending sequence of groups in the
model of density $d$}, is a sequence of groups $\{\Gamma_{l_{n}}\}_{n}$,
so that: 
\begin{enumerate}
\item The sequence $\{l_{n}\}_{n}$ is an unbounded increasing sequence
of integers. 
\item For all $n$, the group $\Gamma_{l_{n}}$ belongs to the level $l_{n}$
of the model (of density $d$). 
\end{enumerate}
\end{defn}

\begin{thm}
\label{thm:49} Let $f,t$ be two real functions with $f(x),t(x)=o(x\ln^{2}x)$,
and let $d<\frac{1}{2}$. Let $\{\Gamma_{l_{n}}\}_{n}$ be an ascending
sequence of groups in the model.

For all $n$, let $g_{n},h_{n}$ be two elements in the group $\Gamma_{l_{n}}$,
and let $A_{g_{n}},A_{h_{n}}$ be two axes for them respectively.
Denote

\[
I_{n}=\text{diam}(A_{g_{n}}^{+f(l_{n})}\cap A_{h_{n}}^{+f(l_{n})})\,.
\]
Assume that the sequences $[g_{n}],[h_{n}]$ of translation lengths
of $g_{n},h_{n}$ resp. in $\Gamma_{l_{n}}$, is bounded by

\[
[g_{n}],[h_{n}]\leq t(l_{n})\,.
\]
If the sequence 
\[
\frac{I_{n}}{l_{n}\ln^{2}l_{n}},
\]
is bounded from below by a positive number, then $g_{n},h_{n}$ commute
eventually. 
\end{thm}

\begin{proof}
By raising $g_{n}$ and $h_{n}$ to large enough powers, and modifying
the function $t$ if necessary, we can assume that the translation
lengths of $g_{n}$ and $h_{n}$ in $\Gamma_{l_{n}}$ satisfy $[g_{n}],[h_{n}]\geq200\alpha_{0}l_{n}$.

Since $f(x)=o(x\ln^{2}x)$, and the sequence 
\[
\frac{I_{n}}{l_{n}\ln^{2}l_{n}},
\]
is bounded from below by a positive number, we deduce by the $\alpha_{0}l_{n}$-hyperbolicity
of $\Gamma_{l_{n}}$ that the sequence

\[
\frac{\text{diam}(A_{g_{n}}^{+60\alpha_{0}l_{n}}\cap A_{h_{n}}^{+60\alpha_{0}l_{n}})}{l_{n}\ln^{2}l_{n}}
\]
is bounded from below by a positive number.

Now if, by contradiction, the elements $g_{n},h_{n}$ do not commute
eventually, then, according to \thmref{48}, (for a subsequence of
$l_{n}$) we get that

\[
\frac{\text{diam}(A_{g_{n}}^{+60\alpha_{0}l_{n}}\cap A_{h_{n}}^{+60\alpha_{0}l_{n}})}{l_{n}\ln^{2}l_{n}}\leq\frac{4\cdot([g_{n}]+[h_{n}])+D(\Gamma_{l_{n}},X)+10\alpha_{0}l_{n}\ln l_{n}}{l_{n}\ln^{2}l_{n}}\,.
\]
And hence, according to \lemref{45},

\[
\frac{\text{diam}(A_{g_{n}}^{+60\alpha_{0}l_{n}}\cap A_{h_{n}}^{+60\alpha_{0}l_{n}})}{l_{n}\ln^{2}l_{n}}\leq\frac{20l_{n}\cdot\ln l_{n}+8t(l_{n})}{l_{n}\ln^{2}l_{n}}\,,
\]
which approaches $0$ as $n$ goes to $\infty$. A contradiction. 
\end{proof}

\section{Limit Groups over Ascending Sequences of Random Groups }\label{sec:Limit-Groups-over-Ascending-Sequences-of-Random-Groups}

We fix a density $d<\frac{1}{2}$, and consider the density model
with respect to this fixed value of $d$. We start by defining $\mathcal{N}_{l}$
to be the collection of all groups in level $l$ that does not satisfy
the $\Sigma$-$lln^{2}$ l.p. for some system in the collection $H_{l}$
(see \defref{36}), where $H_{l}$ is defined in \thmref{37}. Then,
we set $\mathcal{N}=\underset{l}{\cup}\mathcal{N}_{l}$. According
to \thmref{37}, the collection $\mathcal{N}$ is negligible. Along
this section, we drop all the groups in $\mathcal{N}$ from the model,
and we continue working with the obtained model (which admits the
same asymptotic functionality).

\subsection{Limit Groups over Ascending Sequences of Random Groups}
\begin{thm}
\label{thm:50} Let $F_{q}$ be a free group, with a basis $x=(x_{1},...,x_{q})$.
Let $\{h_{l_{n}}:F_{q}\rightarrow\Gamma_{l_{n}}\}$ be a sequence
of homomorphisms over an ascending sequence of groups in the model.

Assume that the sequence of stretching factors

\[
\mu_{n}=\underset{f_{n}\in\Gamma_{l_{n}}}{\min}\underset{1\le u\leq q}{\max}d_{\Gamma_{l_{n}}}(1,\tau_{f_{n}}\circ h_{l_{n}}(x_{u}))\,,
\]
satisfies that 
\[
\mu_{n}\geq l_{n}\ln^{2}l_{n}
\]
for all $n$, where $\tau_{f_{n}}$ is the inner automorphism $g\mapsto f_{n}gf_{n}^{-1}$
of $\Gamma_{l_{n}}$.

For all $n$, we pick an element $f_{n}\in\Gamma_{l_{n}}$, so that

\[
\mu_{n}=\underset{1\le u\leq q}{\max}d_{\Gamma_{l_{n}}}(1,\tau_{f_{n}}\circ h_{l_{n}}(x_{u}))\,.
\]

Then, the sequence $\tau_{f_{n}}\circ h_{l_{n}}$, subconverges (in
the Gromov topology) to an isometric non-trivial action of $F_{q}$
on a pointed real tree $(Y,y_{0})$. 
\end{thm}

\begin{proof}
For all $n$, let $\delta_{n}$ be the hyperbolicity constant of the
group $\Gamma_{l_{n}}$.

According to \thmref{27}, we have that $\delta_{n}=\alpha_{0}l_{n}$.
Then,

\[
\frac{\delta_{n}}{\mu_{n}}\overset{n\rightarrow\infty}{\longrightarrow}0\,.
\]
Hence, according to Proposition 1.1 in \cite{DGI}, the limit space
of the sequence of the rescaled Cayley graphs $\{(X_{l_{n}}/\mu_{n},1)\}_{n}$
of the groups $\{\Gamma_{l_{n}}\}_{n}$, is $0$-hyperbolic, i.e.,
a real tree. 
\end{proof}
\begin{thm}
\label{thm:51} Keep the notation of \ref{thm:50}.

We denote 
\[
K_{h_{l_{n}},\infty}=\{g\in F_{q}:\forall y\in Y,\quad gy=y\}\,.
\]
Let $L_{h_{l_{n}},\infty}$ be the quotient group $F_{q}/K_{h_{l_{n}},\infty}$.
Then, 
\begin{enumerate}
\item $L_{h_{l_{n}},\infty}$ is a f.g. group. 
\item $L_{h_{l_{n}},\infty}$ is torsion-free. 
\item If $Y$ is a real line, then $L_{h_{l_{n}},\infty}$ is free abelian. 
\item If an element $g\in F_{q}$ stabilizes (pointwise) a non-trivial tripod
in $Y$, then $h_{l_{n}}(g)=1$ eventually. 
\item Let $g\in F_{q}$ be an element that does not belong to $K_{h_{l_{n}},\infty}$.
Then, $h_{l_{n}}(g)\neq1$ eventually. 
\item Let $[y_{1},y_{2}]\subset[y_{3},y_{4}]$ be a pair of non-degenerate
segments in $Y$, and assume that $stab([y_{3},y_{4}])$, the (pointwise)
stabilizer of $[y_{3},y_{4}]$ is non-trivial. Then, $stab([y_{3},y_{4}])$
is abelian, and 
\[
stab([y_{1},y_{2}])=stab([y_{3},y_{4}])\,.
\]
\end{enumerate}
\end{thm}

\begin{proof}
Point 1 is clear.

For the other statements, let $g_{1},g_{2}\in F_{q}$.

If $g_{1}$ either stabilizes a non-trivial tripod in $Y$, or $Y$
is a real line and $g_{1}$ does not translate it, then, by the definition
of convergence, according to \lemref{47}, and since random groups
are torsion-free, for large enough $n$, we must have that $h_{l_{n}}(g_{1})=1$
in $\Gamma_{l_{n}}$.

Assume now that $g_{1},g_{2}$ stabilizes a non-degenerate segment
in $Y$. Then, according to \thmref{49}, the elements $h_{l_{n}}(g_{1})$
and $h_{l_{n}}(g_{2})$ of $\Gamma_{l_{n}}$, must commute eventually.

Finally, if $g_{1}\in stab([y_{1},y_{2}])$ and $g_{2}\in stab([y_{3},y_{4}])$,
then $g_{2}$ stabilizes the tripod $[y_{3},g_{1}y_{3},y_{2}]$. 
\end{proof}
\begin{defn}
\label{def:52} Let $F_{q}$ be a free group, with a basis $x=(x_{1},...,x_{q})$.
Let $\{h_{l_{n}}:F_{q}\rightarrow\Gamma_{l_{n}}\}$ be a sequence
of homomorphisms over an ascending sequence of groups in the model.

Assume that there exists a positive number $\epsilon>0$, so that
the sequence of stretching factors

\[
\mu_{n}=\underset{f_{n}\in\Gamma_{l_{n}}}{\min}\underset{1\le u\leq q}{\max}d_{\Gamma_{l_{n}}}(1,\tau_{f_{n}}\circ h_{l_{n}}(x_{u}))\,,
\]
satisfies that 
\[
\mu_{n}\geq\epsilon l_{n}\ln^{2}l_{n}
\]
for all $n$. For all $n$, let $f_{n}\in\Gamma_{l_{n}}$ be an element
so that

\[
\mu_{n}=\underset{1\le u\leq q}{\max}d_{\Gamma_{l_{n}}}(1,\tau_{f_{n}}\circ h_{l_{n}}(x_{u}))\,.
\]

The sequence $g_{n}=\tau_{f_{n}}\circ h_{l_{n}}$, is called an \emph{$lln^{2}$-sequence
over the ascending sequence $\Gamma_{l_{n}}$ of groups in the model},
or briefly, an $lln^{2}$-sequence.

In what follows in this section, we will usually write $h_{l_{n}}$
in place of $\tau_{f_{n}}\circ h_{l_{n}}$.

In case the sequence $\tau_{f_{n}}\circ h_{l_{n}}$ converges (in
the Gromov topology) to an isometric non-trivial action of $F_{q}$
on a real tree, then, in accordance with the notation of \ref{thm:51},
the quotient group

\[
L_{h_{l_{n}},\infty}=F_{q}/K_{h_{l_{n}},\infty}
\]
is called an \emph{$lln^{2}$-limit group over the ascending sequence
$\Gamma_{l_{n}}$ of groups in the model}, or briefly, an $lln^{2}$-limit
group.

We say also that \emph{the $lln^{2}$-sequence $h_{l_{n}}$ converges
into the $lln^{2}$-limit group $L_{h_{l_{n}},\infty}$}. 
\end{defn}

\begin{defn}
Let $F_{q}$ be a free group, with a fixed basis $x=(x_{1},...,x_{q})$,
and let $\{h_{l_{n}}:F_{q}\rightarrow\Gamma_{l_{n}}\}$ be a sequence
of homomorphisms over an ascending sequence of groups in the model.
Let $\pi:F_{k}\rightarrow\Gamma$ be a given presentation of a group
$\Gamma$, and let $h:F_{q}\rightarrow\Gamma$ be a homomorphism. 
\begin{enumerate}
\item The length of the homomorphism $h:F_{q}\rightarrow\Gamma$ is defined
to be 
\[
||h||=\underset{1\le u\leq1}{\max}d_{\Gamma}(1,h(x_{u}))\,.
\]
\item Let $\Sigma$ be a collection of words in the free group $F_{q}$.
Let $\pi:F_{k}\rightarrow\Gamma$ be a given presentation of a group
$\Gamma$. Let $h:F_{q}\rightarrow\Gamma$ be a homomorphism with
$h(\Sigma)=1$. We say that the homomorphism $h$ \emph{can be $\Sigma$-lifted},
if there exists some homomorphism $\tilde{h}:F_{q}\rightarrow F_{k}$
with $\tilde{h}(\Sigma)=1$ and $h=\pi\circ\tilde{h}$. 
\item We say that the sequence $\{h_{l_{n}}:F_{q}\rightarrow\Gamma_{l_{n}}\}$
is \emph{convergent} if for all $w\in F_{q}$, either $h_{l_{n}}(w)=1$
eventually, or $h_{l_{n}}(w)\neq1$ eventually. The collection 
\[
K_{\infty}(h_{l_{n}})=\{w\in F_{q}:h_{l_{n}}(w)=1\text{ eventually}\}
\]
is called \emph{the kernal of the sequence $\{h_{l_{n}}:F_{q}\rightarrow\Gamma_{l_{n}}\}$}.
In this case, the group 
\[
F_{q}/K_{\infty}(h_{l_{n}})
\]
is called \emph{the algebraic limit of the sequence $\{h_{l_{n}}:F_{q}\rightarrow\Gamma_{l_{n}}\}$}. 
\item We say that the sequence $\{h_{l_{n}}:F_{q}\rightarrow\Gamma_{l_{n}}\}$
is a \emph{non-lift sequence}, if there exists a finite set of words
$\Sigma\in F_{q}$, so that $h_{l_{n}}(\Sigma)=1$ eventually, but
for all $n$, the homomorphism $h_{l_{n}}$ cannot be $\Sigma$-lifted. 
\item If the sequence $\{h_{l_{n}}:F_{q}\rightarrow\Gamma_{l_{n}}\}$ is
a non-lift sequence that converges to an $lln^{2}$-limit group $L_{\infty}$,
then the group $L_{\infty}$ is called a \emph{non-lift $lln^{2}$-limit
group}. 
\end{enumerate}
\end{defn}

\begin{rem}
$ $ 
\begin{enumerate}
\item Since we have dropped the collection $\mathcal{N}$, we can assume
that any non-lift sequence is in necessity an $lln^{2}$-sequence. 
\item Our interest in the remainder of this chapter will be in non-lift
$lln^{2}$-limit groups. Paradoxically, our aim in this chapter is
showing that there exist no non-lift $lln^{2}$-limit groups. 
\end{enumerate}
\end{rem}

\begin{defn}
\label{def:53} Let $F_{q}$ be a free group, with a fixed basis $x=(x_{1},...,x_{q})$.
For $j=1,2$, let $\{h_{l_{n}^{j}}^{j}:F_{q}\rightarrow\Gamma_{l_{n}^{j}}^{j}\}_{n}$
be a non-lift sequence of homomorphisms over an ascending sequence
of groups in the model, that converges to the non-lift $lln^{2}$-limit
group $\eta_{j}:F_{q}\rightarrow L_{j}$.

If the function $\eta_{1}(x_{i})\mapsto\eta_{2}(x_{i})$, $i=1,...,q$,
defines a homomorphism $L_{1}\rightarrow L_{2}$, then, we say that
\emph{$L_{2}$ is an $lln^{2}$-limit quotient of $L_{1}$}, and write

\[
L_{1}\geq L_{2}\,.
\]

If further the map $\eta_{1}(x_{i})\mapsto\eta_{2}(x_{i})$, $i=1,...,q$,
is not injective, we say that \emph{$L_{2}$ is a proper $lln^{2}$-limit
quotient of $L_{1}$}, and write

\[
L_{1}>L_{2}\,.
\]
\end{defn}

\begin{lem}
\label{lem:54} Let $F_{q}$ be a free group, with a fixed basis $x=(x_{1},...,x_{q})$.
Let $\{h_{l_{n}}:F_{q}\rightarrow\Gamma_{l_{n}}\}$ be a non-lift
sequence that converges to a non-lift $lln^{2}$-limit group $\eta:F_{q}\rightarrow L_{\infty}$.
Then, 
\end{lem}

\begin{enumerate}
\item $L_{\infty}$ is not free, nor free abelian. 
\item Let $L_{\infty}=L_{1}\ast...\ast L_{g}\ast F$ be a Grushko decomposition
for $L_{\infty}$, so that $L_{j}$ is non-cyclic, $j=1,...,g$, and
$F$ is a free group. Then, there exists $j_{0}=1,...,g$, so that
the subgroup $L_{j_{0}}$ is a non-lift $lln^{2}$-limit group.

Moreover, let $b=(b_{1},...,b_{p})$ be a generating set for $L_{j_{0}}$.
Let $F_{p}$ be a free group with a fixed basis $y=(y_{1},...,y_{p})$.
Let $\beta:F_{p}\rightarrow F_{q}$ be defined by mapping $y_{i}$
to some $\eta$-preimage of $b_{i}$, $i=1,...,p$. Then, the sequence

\[
\{h_{l_{n}}\circ\beta:F_{p}\rightarrow\Gamma_{l_{n}}\}
\]
contains a non-lift subsequence that converges to $L_{j_{0}}$.

\end{enumerate}
\begin{proof}
For Point 1, if $L_{\infty}$ was free or free abelian, then, according
to \thmref{51}, and by moving to a subsequence if necessary, the
algebraic limit $A_{\infty}$ of the sequence $\{h_{l_{n}}:F_{q}\rightarrow\Gamma_{l_{n}}\}$
must be a finitely generated free group, or a finitely generated free
abelian group. Hence, the homomorphisms in the sequence $\{h_{l_{n}}:F_{q}\rightarrow\Gamma_{l_{n}}\}$
factor through $A_{\infty}$ eventually. But $A_{\infty}$ is free
or free abelian, and the groups $\Gamma_{l_{n}}$ are torsion-free
hyperbolic. This contradicts the assumption that the sequence $\{h_{l_{n}}:F_{q}\rightarrow\Gamma_{l_{n}}\}$
is a non-lift sequence.

For point 2, let $b^{1},...,b^{g},f$ be generating sets of sizes
$p_{1},...,p_{g},r$ for the components $L_{1},...,L_{g},F$ respectively.
Let

\[
\beta_{i}:F_{p_{i}}\rightarrow F_{q},\quad\beta_{f}:F_{r}\rightarrow F_{q}
\]
be maps that map the $j$-th element in a fixed basis for the respective
free group $F_{p_{i}}$ or $F_{r}$, to an $\eta$-preimage of the
$j$-th element in the respective set $b^{i}$ or $f$, $i=1,...,g$.

We consider the sequences

\[
\{h_{l_{n}}^{i}:F_{p_{i}}\rightarrow\Gamma_{l_{n}}\},\,i=1,...,g,\,\text{ and }\{h_{l_{n}}^{f}:F_{r}\rightarrow\Gamma_{l_{n}}\}\,,
\]
given by 
\[
h_{l_{n}}^{i}=h_{l_{n}}\circ\beta_{i},\quad h_{l_{n}}^{f}=h_{l_{n}}\circ\beta_{f}\,.
\]

We also denote $\eta_{i}:F_{p_{i}}\rightarrow L_{i}$, and $\eta_{f}:F_{r}\rightarrow F$
the maps $\eta_{i}=\eta\circ\beta_{i}$, and $\eta_{f}=\eta\circ\beta_{f}$
respectively.

According to the proof of Point 1, the sequence $\{h_{l_{n}}^{f}:F_{r}\rightarrow\Gamma_{l_{n}}\}$
contains no non-lift subsequence. Moreover, for large enough $n$,
the homomorphism $h_{l_{n}}^{f}:F_{r}\rightarrow\Gamma_{l_{n}}$ can
be lifted by a homomorphism $\tilde{h}_{l_{n}}^{f}:F_{r}\rightarrow F_{k}$
that factors through $\eta_{f}:F_{r}\rightarrow F$, that is, $h_{l_{n}}^{f}=\pi_{\Gamma_{l_{n}}}\circ\tilde{h}_{l_{n}}^{f}$
and there exists some homomorphism $\lambda_{n}^{f}:L_{i}\rightarrow F_{k}$
so that $\tilde{h}_{l_{n}}^{f}=\lambda_{n}^{f}\circ\eta_{f}$.

Now assume by contradiction that for all $i=1,...,g$, the sequence
$\{h_{l_{n}}^{i}:F_{p_{i}}\rightarrow\Gamma_{l_{n}}\}$ contains no
non-lift subsequence. Then, for all $i=1,...,g$, according to the
definition of non-lift sequences, for every finite collection $\Sigma$
in $F_{p_{i}}$, if $h_{l_{n}}^{i}(\Sigma)=1$ eventually, then, the
homomorphisms $h_{l_{n}}^{i}$ can be $\Sigma$-lifted eventually.

Let $i=1,...,g$. By Guba's theorem, every system over the free group
$F_{k}$, is equivalent to a finite subsystem. In particular, the
system describing the defining relations of the subgroup $L_{i}$
(w.r.t. the set of generators $b^{i}$), is equivalent to a finite
subsystem over the free group $F_{k}$. Recall that for every $w$
in $F_{p_{i}}$, we have that $\eta_{i}(w)=1$ if and only if $w$
is mapped by $h_{l_{n}}^{i}$ to the trivial element for all large
enough $n$.

Hence, we conclude that for all large enough $n$, the homomorphism
$h_{l_{n}}^{i}:F_{p_{i}}\rightarrow\Gamma_{l_{n}}$ can be lifted
by a homomorphism $\tilde{h}_{l_{n}}^{i}:F_{p_{i}}\rightarrow F_{k}$
that factors through $\eta_{i}:F_{p_{i}}\rightarrow L_{i}$, that
is, $h_{l_{n}}^{i}=\pi_{\Gamma_{l_{n}}}\circ\tilde{h}_{l_{n}}^{i}$
and there exists some homomorphism $\lambda_{n}^{i}:L_{i}\rightarrow F_{k}$
so that $\tilde{h}_{l_{n}}^{i}=\lambda_{n}^{i}\circ\eta_{i}$.

Summing together, we obtain that for all large enough $n$, the homomorphisms
$\lambda_{n}^{1},...,\lambda_{n}^{g},\lambda_{n}^{f}$ implies a homomorphism

\[
\lambda_{n}:L_{\infty}\rightarrow F_{k}\,,
\]
so that 
\[
\pi_{\Gamma_{l_{n}}}\circ\lambda_{n}\circ\eta=h_{l_{n}}\,.
\]

In particular, the homomorphism $\lambda_{n}\circ\eta:F_{q}\rightarrow F_{k}$
lifts the homomorphism $h_{l_{n}}:F_{q}\rightarrow\Gamma_{l_{n}}$.

In order to get a contradiction, recall that the sequence $\{h_{l_{n}}:F_{q}\rightarrow\Gamma_{l_{n}}\}$
is non-lift. I.e., there exists a finite collection $\Sigma_{0}$
of elements in $F_{q}$, so that $h_{l_{n}}(\Sigma_{0})=1$ eventually,
but for all $n$, the homomorphism $h_{l_{n}}$ cannot be $\Sigma_{0}$-lifted.
However, as the sequence $\{h_{l_{n}}:F_{q}\rightarrow\Gamma_{l_{n}}\}$
converges to $L_{\infty}$, the canonical map $\eta:F_{q}\rightarrow L_{\infty}$
must map the elements of $\Sigma_{0}$ to $1$. Hence, $\lambda_{n}\circ\eta$
is a $\Sigma_{0}$-lift for $h_{l_{n}}$ for all large enough $n$,
a contradiction. 
\end{proof}

\subsection{Shortening Quotients of Limit Groups over Ascending Sequences of
Random Groups}
\begin{lem}
\label{lem:55} (See Lemma 1.4 in \cite{DGI}) Let $F_{q}$ be a free
group, with a fixed basis $x=(x_{1},...,x_{q})$. Let $\{h_{l_{n}}:F_{q}\rightarrow\Gamma_{l_{n}}\}$
be a non-lift sequence, that converges to a non-lift $lln^{2}$-limit
group

\[
\eta:F_{q}\rightarrow L_{\infty}\,.
\]

Then, the following properties are satisfied in $L_{\infty}$. 
\begin{enumerate}
\item Let $u_{1},u_{2},u_{3}$ be non-trivial elements of $L_{\infty}$,
and suppose that $[u_{1},u_{2}]=1$ and $[u_{1},u_{3}]=1$. Then,
$[u_{2},u_{3}]=1$. It follows that every abelian subgroup in $L_{\infty}$
is contained in a unique maximal abelian subgroup. 
\item Every maximal abelian subgroup of $L_{\infty}$ is malnormal. 
\item Every solvable subgroup of $L_{\infty}$ is abelian. 
\end{enumerate}
\end{lem}

\begin{proof}
In light of \thmref{51}, exactly the same proof of Lemma 1.4 in \cite{DGI}
proves the statement in the lemma, with the only exception that instead
of using malnormality of maximal cyclic subgroups of $F_{k}$, we
use malnormality of maximal cyclic subgroups of torsion-free hyperbolic
groups (in proving Point 2). 
\end{proof}
Let $L_{\infty}$ be a non-lift $lln^{2}$-limit group, and assume
that $L_{\infty}$ is freely indecomposable. In light of \lemref{55},
\thmref{50}, and \thmref{51}, Chapter 2 of \cite{DGI} can be applied
(as it is) in our context, in order to construct \emph{the abelian
JSJ decomposition of the freely indecomposable non-lift $lln^{2}$-limit
group $L_{\infty}$}. For the exact definition of the JSJ decomposition,
see Chapter 2 of \cite{DGI}. For the exact statement on how the (abelian)
JSJ decomposition encodes all the abelian splittings of $L_{\infty}$,
see Theorem 2.7 in \cite{DGI}. 
\begin{defn}
\label{def:56}Let $F_{q}$ be a free group, with a fixed basis $x=(x_{1},...,x_{q})$.
Let $\{h_{l_{n}}:F_{q}\rightarrow\Gamma_{l_{n}}\}$ be a non-lift
sequence, that converges to a non-lift $lln^{2}$-limit group

\[
\eta:F_{q}\rightarrow L_{\infty}\,.
\]

According to \lemref{54}, the group $L_{\infty}$ is non-free.

We are going to define a ``shortening quotient'' of $L_{\infty}$,
inspired by the manner explained in {[}\cite{Daniel=000020Groves=000020-=000020Limit=000020groups=000020relatively=000020hyperbolic},
sections 3 and 4{]}, and that in \cite{Zlil=000020and=000020Jaligot=000020-=000020free=000020products}. 
\begin{itemize}
\item Assume first that $L_{\infty}$ is freely-indecomposable. 
\end{itemize}
Let $\Lambda_{L_{\infty}}$ be the abelian JSJ decomposition of $L_{\infty}$.
Let $V^{1},...,V^{m}$, $E^{1},...,E^{s}$, and $t^{1},...,t^{b}$
be the corresponding vertex groups, edge groups, and Bass-Serre generators
(w.r.t some maximal subtree of $\Lambda_{L_{\infty}}$) respectively,
of $\Lambda_{L_{\infty}}$. Since $L_{\infty}$ is finitely generated,
for all $j=1,...,m$, the vertex group $V^{j}$ is generated by a
finite subset $v_{1}^{j},...,v_{r_{j}}^{j}$, together with the adjacent
edge groups in $\Lambda_{L_{\infty}}$. For all $j=1,...,s$, we choose
an ordered basis for the edge group $E^{j}$, and write $E_{p}^{j}$
for the first $p$ elements in that ordered basis.

For all $n$, let $U_{n}$ be the group given by the presentation
\[
U_{n}=\langle v_{u}^{j},t^{j},E_{n}^{j}:\mathcal{R}_{n}\rangle\,,
\]
where $\mathcal{R}_{n}$ denotes the collection of all the relations
of length $n$ between the generators $v_{u}^{j},t^{j},E_{n}^{j}$
as an elements of $L_{\infty}$. Let

\[
\varphi_{n}:U_{n}\rightarrow L_{\infty}
\]
be the obvious mapping.

Since $L_{\infty}$ is generated by 
\[
V^{1},...,V^{m},E^{1},...,E^{s},t^{1},...,t^{b}
\]

we get that, for some integer $p$, there exist words $w_{i}=w_{i}(v_{u}^{j},t^{j},E_{p}^{j})$,
$i=1,...,q$, in the formal elements $v_{u}^{j},t^{j},E_{n}^{j}$,
so that $w_{i}$ represents the element $g_{i}$ in $L_{\infty}$,
where 
\[
g_{1}=\eta(x_{1}),...,g_{q}=\eta(x_{q})
\]
($x=(x_{1},...,x_{q})$ is the fixed basis of $F_{q}$). According
to that, for all $n$, we define the map

\[
\psi_{n}:F_{q}\rightarrow U_{n}\,,
\]
by $\psi_{n}(x_{i})=w_{i}$. Note that $\varphi_{n}\circ\psi_{n}=\eta$
for all $n$ (recall $\eta:F_{q}\rightarrow L_{\infty}$ is the canonical
map).

For every integer $n$, the relations of the group $U_{n}$ are taken
from those in $L_{\infty}$. Hence, in accordance with \thmref{51},
there exists an ascending sequence $m(n)$ of integers, so that the
homomorphism $h_{l_{m(n)}}$ factors through $\psi_{n}:F_{q}\rightarrow U_{n}$,
i.e., for all $n$, there exists a homomorphism

\[
\lambda_{n}:U_{n}\rightarrow\Gamma_{l_{m(n)}}\,,
\]
so that $\lambda_{n}\circ\psi_{n}=h_{l_{m(n)}}$.

The sequence of groups $U_{n}$ is thought of as an ``approximation
sequence'' for the decomposition $\Lambda_{L_{\infty}}$. Indeed,
since $L_{\infty}$ is freely indecomposable non-cyclic, and since
the sequence $\{h_{l_{n}}:F_{q}\rightarrow\Gamma_{l_{n}}\}$ converges
to a non-trivial action of $L_{\infty}$ on a pointed real tree so
that the properties listed in \thmref{51} are satisfied, we get according
to the shortening argument presented in {[}\cite{Daniel=000020Groves=000020-=000020Limit=000020groups=000020relatively=000020hyperbolic},
sections 3 and 4{]}, that for all large enough $n$, the homomorphism
$h_{l_{m(n)}}$ can be shortened (strictly) by ``intermediate''
composing with an automorphism of $Mod(U_{n})$, or post-composing
with an inner automorphism of $\Gamma_{l_{m(n)}}$.

That is, for all large enough $n$, there exists a modular automorphism
$\phi_{n}\in Mod(U_{n})$, together with some element $f_{n}\in\Gamma_{l_{m(n)}}$,
so that

\[
||h_{l_{m(n)}}||>||\tau_{f_{n}}\circ\lambda_{n}\circ\phi_{n}\circ\psi_{n}||\,.
\]

For all large enough $n$, we choose $\phi_{n}\in Mod(U_{n})$ and
$f_{n}\in\Gamma_{l_{m(n)}}$ so that the right hand side of the above
inequality is minimal possible. Then, we consider the sequence of
shortened homomorphisms

\[
v_{l_{m(n)}}=\tau_{f_{n}}\circ\lambda_{n}\circ\phi_{n}\circ\psi_{n}:F_{q}\rightarrow\Gamma_{l_{m(n)}}\,.
\]

Such a sequence $\{v_{l_{m(n)}}:F_{q}\rightarrow\Gamma_{l_{m(n)}}\}$
is called a \emph{shortened sequence for the sequence $\{h_{l_{n}}:F_{q}\rightarrow\Gamma_{l_{n}}\}$}. 
\begin{lem}
\label{lem:57} Let $\{v_{l_{m(n)}}:F_{q}\rightarrow\Gamma_{l_{m(n)}}\}$
be a shortened sequence for the non-lift sequence $\{h_{l_{n}}:F_{q}\rightarrow\Gamma_{l_{n}}\}$.
Then, the sequence $\{v_{l_{m(n)}}:F_{q}\rightarrow\Gamma_{l_{m(n)}}\}$
contains a non-lift subsequence. 
\end{lem}

\begin{proof}
Assume by contradiction the claim is false. By moving to a subsequence
if necessary, we may assume that the sequence $\{v_{l_{m(n)}}:F_{q}\rightarrow\Gamma_{l_{m(n)}}\}$
is convergent. Then, according to the definition of non-lift sequences,
for every finite collection $\Sigma$ in $F_{q}$, if $v_{l_{m(n)}}(\Sigma)=1$
eventually, then, the homomorphisms $v_{l_{m(n)}}$ can be $\Sigma$-lifted
for all large enough $n$.

By Guba's theorem, every system over the free group $F_{k}$, is equivalent
to a finite subsystem. In particular, the system describing the defining
relations of the group $L_{\infty}$ (w.r.t. the set of generators
$\eta(x)$, where $x$ is the fixed basis of $F_{q}$), is equivalent
to a finite subsystem over the free group $F_{k}$. Recall that for
every $w$ in $F_{q}$, we have that $\eta(w)=1$ if and only if $w$
is mapped by $\psi_{n}:F_{q}\rightarrow U_{n}$ to the trivial element
for all large enough $n$.

Hence, we conclude that for all large enough $n$, the homomorphism
$v_{l_{m(n)}}:F_{q}\rightarrow\Gamma_{l_{m(n)}}$ can be lifted by
a homomorphism $\tilde{v}_{l_{m(n)}}:F_{q}\rightarrow F_{k}$ that
factors through $\eta:F_{q}\rightarrow L_{\infty}$, that is, $v_{l_{m(n)}}=\pi_{\Gamma_{l_{m(n)}}}\circ\tilde{v}_{l_{m(n)}}$
and there exists some homomorphism

\[
\tilde{u}_{n}:L_{\infty}\rightarrow F_{k}
\]
so that $\tilde{v}_{l_{m(n)}}=\tilde{u}_{n}\circ\eta$.

We denote

\[
u_{n}=\pi_{\Gamma_{l_{m(n)}}}\circ\tilde{u}_{n}:L_{\infty}\rightarrow\Gamma_{l_{m(n)}}\,,
\]
and we note that, tautologically, the homomorphism $\tilde{u}_{n}$
lifts the homomorphism $u_{n}$. Hence, the homomorphism $\tilde{u}_{n}\circ\varphi_{n}:U_{n}\rightarrow F_{k}$
lifts the homomorphism

\[
\tau_{f_{n}}\circ\lambda_{n}\circ\phi_{n}:U_{n}\rightarrow\Gamma_{l_{m(n)}}\,.
\]

Let $\tilde{f}_{n}\in F_{k}$ be a word representing the element $f_{n}$
in $\Gamma_{l_{m(n)}}$. Then, since

\begin{align*}
\pi_{\Gamma_{l_{m(n)}}}\circ\tau_{\tilde{f}_{n}^{-1}}\circ\tilde{u}_{n}\circ\varphi_{n}\circ\phi_{n}^{-1} & =\\
 & =\tau_{f_{n}^{-1}}\circ\pi_{\Gamma_{l_{m(n)}}}\circ\tilde{u}_{n}\circ\varphi_{n}\circ\phi_{n}^{-1}\\
 & =\tau_{f_{n}^{-1}}\circ\tau_{f_{n}}\circ\lambda_{n}\circ\phi_{n}\circ\phi_{n}^{-1}\\
 & =\lambda_{n}\,,
\end{align*}

we have that the homomorphism

\[
\tilde{\lambda}_{n}=\tau_{\tilde{f}_{n}^{-1}}\circ\tilde{u}_{n}\circ\varphi_{n}\circ\phi_{n}^{-1}:U_{n}\rightarrow F_{k}
\]
(denoted $\tilde{\lambda}_{n}$) lifts the homomorphism 
\[
\lambda_{n}:U_{n}\rightarrow\Gamma_{l_{m(n)}}\,.
\]

We recall the homomorphism 
\[
h_{l_{m(n)}}=\lambda_{n}\circ\psi_{n}:F_{q}\rightarrow\Gamma_{l_{m(n)}}\,,
\]
and we note that the homomorphism

\[
\tilde{h}_{l_{m(n)}}=\tilde{\lambda}_{n}\circ\psi_{n}:F_{q}\rightarrow F_{k}
\]
lifts it.

As we have mentioned above, by Guba's theorem, the system describing
the defining relations of the group $L_{\infty}$, is equivalent to
a finite subsystem over the free group $F_{k}$. Thus, since every
element $w$ of $F_{q}$ with $\eta(w)=1$ is mapped by $\psi_{n}$
to the trivial element in $U_{n}$ for all large enough $n$, we conclude
that for all large enough $n$, the homomorphism $\tilde{h}_{l_{m(n)}}$
factors through $\eta:F_{q}\rightarrow L_{\infty}$. Hence, the original
sequence $\{h_{l_{n}}:F_{q}\rightarrow\Gamma_{l_{n}}\}$ is not non-lift,
a contradiction. 
\end{proof}
We continue the definition with our shortened sequence $\{v_{l_{m(n)}}:F_{q}\rightarrow\Gamma_{l_{m(n)}}\}$
(for the original non-lift sequence $\{h_{l_{n}}:F_{q}\rightarrow\Gamma_{l_{n}}\}$).

According to \lemref{57}, by moving to a subsequence, we may assume
the shortened sequence $\{v_{l_{m(n)}}:F_{q}\rightarrow\Gamma_{l_{m(n)}}\}$
is non-lift again. Hence, according to \thmref{50}, we may assume
further that the shortened sequence $\{v_{l_{m(n)}}:F_{q}\rightarrow\Gamma_{l_{m(n)}}\}$
converges to a (non-lift) $lln^{2}$-limit group $\eta_{Q}:F_{q}\rightarrow Q_{\infty}$.

Such an $lln^{2}$-limit group $Q_{\infty}$ is called an \emph{$lln^{2}$-shortening
quotient of the non-lift $lln^{2}$-limit group $L_{\infty}$}.

Note that, since every element $w$ of $F_{q}$ with $\eta(w)=1$
is mapped by $\psi_{n}$ to the trivial element in $U_{n}$ for all
large enough $n$, we have that $Q_{\infty}$ is indeed a quotient
of $L_{\infty}$; the map $\eta(x)\mapsto\eta_{Q}(x)$ is a quotient
map: 
\[
L_{\infty}\geq Q_{\infty}\,.
\]

Moreover, as we have explained above, it cannot happen that $L_{\infty}\cong Q_{\infty}$,
for otherwise, one can shorten the homomorphism $v_{l_{m(n)}}$ by
composing with an automorphism in $Mod(U_{n})$ (the same $U_{n}$
defined above) and inner automorphism of $\Gamma_{l_{m(n)}}$, for
infinitely many integers $n$, which contradicts the construction
of $v_{l_{m(n)}}$. Hence, actually $Q_{\infty}$ is a proper quotient
of $L_{\infty}$: 
\[
L_{\infty}>Q_{\infty}\,.
\]
\begin{itemize}
\item Now assume that $L_{\infty}$ admits a non-trivial Grushko decomposition
$L_{\infty}=L_{1}\ast...\ast L_{g}\ast F$, so that $L_{j}$ is non-cyclic,
$j=1,...,g$, and $F$ is a free group. 
\end{itemize}
According to \lemref{54}, we may assume that $L_{1}$ is a non-lift
$lln^{2}$-limit group. Moreover, choosing a generating set $b^{1}=(b_{1}^{1},...,b_{p_{1}}^{1})$
for $L_{1}$, choosing a free group $F_{p_{1}}$ with basis $y^{1}=(y_{1}^{1},...,y_{p_{1}}^{1})$,
and passing to a subsequence, we may assume that the sequence

\[
\{h_{l_{n}}^{1}=h_{l_{n}}\circ\beta_{1}:F_{p_{1}}\rightarrow\Gamma_{l_{n}}\}
\]
is a non-lift sequence, where $\beta_{1}:F_{p_{1}}\rightarrow F_{q}$
is the map defined by mapping $y_{i}^{1}$ to some $\eta$-preimage
of $b_{i}^{1}$, $i=1,...,p_{1}$. We denote also

\[
\eta_{1}:F_{p_{1}}\rightarrow L_{1}
\]
the map $\eta_{1}=\eta\circ\beta_{1}$.

We further, abuse the notation and denote the free component $L_{2}\ast...\ast L_{g}\ast F$
by $L_{2}$.

We fix a generating set $b^{2}=(b_{1}^{2},...,b_{p_{2}}^{2})$ for
$L_{2}$, a free group $F_{p_{2}}$ with basis $y^{2}=(y_{1}^{2},...,y_{p_{2}}^{2})$,
and a map $\beta_{2}:F_{p_{2}}\rightarrow F_{q}$ defined by mapping
$y_{i}^{2}$ to some $\eta$-preimage of $b_{i}^{2}$, $i=1,...,p_{2}$.
We consider the sequence

\[
\{h_{l_{n}}^{2}=h_{l_{n}}\circ\beta_{2}:F_{p_{2}}\rightarrow\Gamma_{l_{n}}\}\,,
\]
and we denote by

\[
\eta_{2}:F_{p_{2}}\rightarrow L_{2}
\]
the map $\eta_{2}=\eta\circ\beta_{2}$.

Now the sequence 
\[
\{h_{l_{n}}^{1}:F_{p_{1}}\rightarrow\Gamma_{l_{n}}\}
\]
is a non-lift sequence that converges to the freely-indecomposable
$lln^{2}$-limit group $L_{1}$. Hence, we may bring the notation
of the first part of the definition (when $L_{\infty}$ was freely-indecomposable)
for the sequence $\{h_{l_{n}}^{1}:F_{p_{1}}\rightarrow\Gamma_{l_{n}}\}$
and the non-lift $lln^{2}$-limit group $L_{1}$, with the difference
that we add an superscript $1$, e.g., in place of $U_{n}$, we write
$U_{n}^{1}$.

We further define new groups $U_{n}^{2}$ to be ``approximations''
for the group $L_{2}$, i.e., for all $n$ we define

\[
U_{n}^{2}=\langle b^{2}:\mathcal{R}_{n}^{2}\rangle\,,
\]
where $\mathcal{R}_{n}^{2}$ denotes the collection of all the relations
of length $n$ between the generators $b^{2}$ as an elements of $L_{2}$.

We consider the groups 
\[
U_{n}=U_{n}^{1}\ast U_{n}^{2}\,.
\]

Since $L_{\infty}$ is generated by the corresponding elements in
the JSJ of $L_{1}$

\[
V^{1},...,V^{m},E^{1},...,E^{s},t^{1},...,t^{b}
\]
together with $b^{2}$, we get that, for some integer $r$, there
exist words $w_{i}=w_{i}(v_{u}^{j},t^{j},E_{r}^{j},b^{2})$, $i=1,...,q$,
in the formal elements $v_{u}^{j},t^{j},E_{n}^{j},b^{2}$, so that
$w_{i}$ represents the element $g_{i}$ in $L_{\infty}$, where 
\[
g_{1}=\eta(x_{1}),...,g_{q}=\eta(x_{q})
\]
($x=(x_{1},...,x_{q})$ is the fixed basis of $F_{q}$). According
to that, for all $n$, we define the map

\[
\psi_{n}:F_{q}\rightarrow U_{n}\,,
\]
by $\psi_{n}(x_{i})=w_{i}$. Note that $\varphi_{n}\circ\psi_{n}=\eta$
for all $n$ (recall $\eta:F_{q}\rightarrow L_{\infty}$ is the canonical
map).

We consider the maps 
\[
\psi_{n}^{j}:F_{p_{j}}\rightarrow U_{n}
\]
given by $\psi_{n}^{j}=\psi_{n}\circ\beta_{j}$, and we note that
for all large enough $n$, the image of $\psi_{n}^{j}$ is contained
entirely in the component $U_{n}^{j}$, $j=1,2$.

As in the previous part of the definition, there exists an ascending
sequence $m(n)$ of integers, so that the homomorphism $h_{l_{m(n)}}$
factors through $\psi_{n}:F_{q}\rightarrow U_{n}$, i.e., for all
large enough $n$, there exists a homomorphism

\[
\lambda_{n}:U_{n}\rightarrow\Gamma_{l_{m(n)}}\,,
\]
so that $\lambda_{n}\circ\psi_{n}=h_{l_{m(n)}}$. For $j=1,2$, we
denote by 
\[
\lambda_{n}^{j}:U_{n}^{j}\rightarrow\Gamma_{l_{m(n)}}\,,
\]
the restriction of $\lambda_{n}$ on $U_{n}^{j}$.

Now, for all large enough $n$, there exists a modular automorphism
$\phi_{n}^{1}\in Mod(U_{n}^{1})$, together with some element $f_{n}\in\Gamma_{l_{m(n)}}$,
so that

\[
||h_{l_{m(n)}}^{1}||>||\tau_{f_{n}}\circ\lambda_{n}\circ\phi_{n}^{1}\circ\psi_{n}^{1}||\,.
\]

For all large enough $n$, we choose $\phi_{n}^{1}\in Mod(U_{n})$
and $f_{n}\in\Gamma_{l_{m(n)}}$ so that the right hand side of the
above inequality is minimal possible. Then, we consider the sequence
of homomorphisms

\[
v_{l_{m(n)}}=\tau_{f_{n}}\circ\lambda_{n}\circ\phi_{n}^{1}\circ\psi_{n}:F_{q}\rightarrow\Gamma_{l_{m(n)}}\,.
\]

For $j=1,2$, we let $v_{l_{m(n)}}^{j}:F_{p_{j}}\rightarrow\Gamma_{l_{m(n)}}$
be the homomorphism

\[
v_{l_{m(n)}}^{j}=v_{l_{m(n)}}\circ\beta_{j}\,.
\]
We note that for all large enough $n$, we have that

\[
v_{l_{m(n)}}^{2}=\tau_{f_{n}}\circ\lambda_{n}\circ\psi_{n}^{2}=\tau_{f_{n}}\circ h_{l_{m(n)}}^{2}\,.
\]
And since the algebraic limit of the sequence $\{h_{l_{n}}^{2}:F_{p_{2}}\rightarrow\Gamma_{l_{n}}\}$
is $L_{2}$, we have that the algebraic limit of (every subsequence
of) the sequence $\{v_{l_{m(n)}}^{2}:F_{p_{2}}\rightarrow\Gamma_{l_{n}}\}$
is $L_{2}$ too.

On the other hand, the sequence $\{v_{l_{m(n)}}^{1}:F_{p_{1}}\rightarrow\Gamma_{l_{n}}\}$
is a shortened sequence for the sequence $\{h_{l_{n}}^{1}:F_{p_{1}}\rightarrow\Gamma_{l_{n}}\}$,
and hence, according to \lemref{57}, we may assume it is a non-lift
sequence. Moreover, we may assume that the sequence $\{v_{l_{m(n)}}^{1}:F_{p_{1}}\rightarrow\Gamma_{l_{n}}\}$
converges to an $lln^{2}$-shortening quotient $Q_{1}$ of $L_{1}$.

In accordance with that, we call the sequence $\{v_{l_{m(n)}}:F_{q}\rightarrow\Gamma_{l_{m(n)}}\}$
a \emph{shortened sequence for the sequence $\{h_{l_{n}}:F_{q}\rightarrow\Gamma_{l_{n}}\}$}. 
\begin{lem}
\label{lem:58} (In the case when $L_{\infty}$ is freely-decomposable
too). Let $\{v_{l_{m(n)}}:F_{q}\rightarrow\Gamma_{l_{m(n)}}\}$ be
a shortened sequence for the non-lift sequence $\{h_{l_{n}}:F_{q}\rightarrow\Gamma_{l_{n}}\}$.
Then, the sequence $\{v_{l_{m(n)}}:F_{q}\rightarrow\Gamma_{l_{m(n)}}\}$
contains a non-lift subsequence. 
\end{lem}

\begin{proof}
Assume by contradiction the claim is false. We already know that the
sequence $\{v_{l_{m(n)}}^{1}:F_{p_{1}}\rightarrow\Gamma_{l_{n}}\}$
is non-lift. Let $\Sigma$ be a finite subset of $F_{p_{1}}$ so that
$v_{l_{m(n)}}^{1}(\Sigma)=1$ eventually, and for all $n$, the homomorphism
$v_{l_{m(n)}}^{1}$ cannot be $\Sigma$-lifted.

Let $\Sigma'$ be the finite subset $\beta_{1}(\Sigma)$ of $F_{q}$.
Since $v_{l_{m(n)}}^{1}=v_{l_{m(n)}}\circ\beta_{1}$, we conclude
that $v_{l_{m(n)}}(\Sigma')=1$ eventually. Since we are assuming
the claim is false, for all large enough $n$, the homomorphism $v_{l_{m(n)}}$
can be $\Sigma'$-lifted. But this implies that for all large enough
$n$, the homomorphism $v_{l_{m(n)}}^{1}$ can be $\Sigma$-lifted,
a contradiction. 
\end{proof}
Hence, in total, by moving to a subsequence, the shortened sequence
$\{v_{l_{m(n)}}:F_{q}\rightarrow\Gamma_{l_{m(n)}}\}$ is a non-lift
sequence, that converges to the non-lift $lln^{2}$-limit group $\eta_{Q}:F_{q}\rightarrow Q_{\infty}=Q_{1}\ast L_{2}$,
and since every element $w$ of $F_{q}$ with $\eta(w)=1$ is mapped
by $\psi_{n}$ to the trivial element in $U_{n}$ for all large enough
$n$, we have that the map $\eta(x)\mapsto\eta_{Q}(x)$ is a quotient
map

\[
L_{\infty}\geq Q_{\infty}\,.
\]
Moreover, since $Q_{1}$ is a proper quotient of $L_{1}$, we have
that $Q_{\infty}$ is a proper quotient of $L_{\infty}$

\[
L_{\infty}>Q_{\infty}\,.
\]

Such an $lln^{2}$-limit group $Q_{\infty}$ is called an \emph{$lln^{2}$-shortening
quotient of the non-lift $lln^{2}$-limit group $L_{\infty}$}. 
\end{defn}

\begin{cor}
\label{cor:59} Let $L_{\infty}$ be a non-lift $lln^{2}$-limit group.
Then, $L_{\infty}$ admits an infinite (strictly) decreasing sequence
of non-lift $lln^{2}$-limit quotients

\[
L_{\infty}>Q_{\infty}^{1}>Q_{\infty}^{2}>...\,.
\]
\end{cor}

\begin{proof}
As we have explained in \defref{56}, every non-lift $lln^{2}$-limit
group $L_{\infty}$ admits an $lln^{2}$-shortening quotient $Q_{\infty}$,
so that $Q_{\infty}$ itself is a non-lift $lln^{2}$-limit group
which is a proper quotient of $L_{\infty}$. Hence, by induction,
one can create an infinite decreasing sequence as the required one. 
\end{proof}
\begin{thm}
\label{thm:60} Every strictly decreasing sequence

\[
L_{1}>L_{2}>L_{3}>...
\]
of non-lift $lln^{2}$-limit groups, terminates after finitely many
steps. 
\end{thm}

\begin{proof}
We follow the proof of theorem 13 in \cite{Zlil=000020and=000020Jaligot=000020-=000020free=000020products}.
Assume by contradiction the claim is false.

We construct then the following strictly decreasing sequence of non-lift
$lln^{2}$-limit groups.

For brevity, given a non-lift $lln^{2}$-limit group $R$, we denote
by $\mathcal{A}_{R}$ the collection of all the non-lift $lln^{2}$-limit
groups $R_{1}'$, for which $R>R_{1}'$ and there exists a strictly
decreasing infinite sequence of non-lift $lln^{2}$-limit groups that
begins with $R_{1}'$

\[
R_{1}'>R_{2}'>R_{3}'>...\,.
\]

For convenience, given a free group $F_{q}$ with a free basis $x=(x_{1},...,x_{q})$,
we write $\mathcal{A}_{F_{q}}$ to denote the collection of all the
non-lift $lln^{2}$-limit groups $R$ whose canonical generating set
consists of $q$ elements, and for which $\mathcal{A}_{R}$ is non-empty.
For all $R$ in $\mathcal{A}_{F_{q}}$, we denote $\eta_{R}:F_{q}\rightarrow R$
the quotient map that maps $x_{i}$ to the $i$-th element in the
canonical generating set of $R$, for all $i=1,...,q$.

From the collection of free groups (bases) $F_{q}$ for which $\mathcal{A}_{F_{q}}$
is non-empty, we choose an $F_{q}$ with $q$ minimal.

We define the strictly decreasing sequence 
\[
R_{1}>R_{2}>R_{3}>...
\]
by induction. First, we denote $R_{0}=F_{q}$. Then, assuming that
$R_{n-1}$ is defined, and that $\mathcal{A}_{R_{n-1}}$ is non-empty,
we let $R_{n}$ to be an element of the collection $\mathcal{A}_{R_{n-1}}$,
so that the map $\eta_{R_{n}}=\eta_{n}:F_{q}\rightarrow R_{n}$, maps
to the identity the maximal possible number of elements in the ball
of radius $n$ in the Cayley graph of $F_{q}$ (w.r.t. the fixed basis
$x$).

For brevity, given a sequence of homomorphisms $\{v_{m}:F_{q}\rightarrow\Gamma_{m}\}$,
an integer $m_{0}$, and a collection of elements $A\subset F_{q}$,
we say that the sequence $\{v_{m}:F_{q}\rightarrow\Gamma_{m}\}$ \emph{stabilizes
$A$ $m_{0}$-eventually}, if for all $w\in A$, we have that {[}$v_{m}(w)=1$
for all $m\ge m_{0}$, or $v_{m}(w)\neq1$ for all $m\ge m_{0}${]}.

Now for all $n\geq1$, the group $R_{n}$ is in particular a non-lift
$lln^{2}$-limit group. For all $n\geq1$, let

\[
\{v_{m}^{n}:F_{q}\rightarrow\Gamma_{l_{m}^{n}}^{n}\}_{m}
\]
be a non-lift sequence that converges to $R_{n}$. For all $n\geq1$,
using the definition of non-lift sequences, let $\Sigma_{n}$ be a
finite subset of $F_{q}$, so that $v_{m}^{n}(\Sigma_{n})=1$ for
all large enough $m$, and assume that the homomorphism $v_{m}^{n}$
cannot be $\Sigma_{n}$-lifted for all $m$. Given two integers $n\geq i\geq1$,
we recall that $R_{n}\geq R_{i}$. Hence, $v_{m}^{n}(\Sigma_{i})=1$
for all large enough $m$. Thus, we may assume that the sequence $\{\Sigma_{n}\}_{n}$
is an ascending sequence

\[
\Sigma_{1}\subset\Sigma_{2}\subset\Sigma_{3}\subset...\,.
\]

In accordance with \lemref{54} and \thmref{51}, for all $n\geq1$,
let $m(n)$ be a large enough integer, so that 
\begin{enumerate}
\item $v_{m(n)}^{n}(\Sigma_{n})=1$, 
\item the sequence $\{v_{m}^{n}:F_{q}\rightarrow\Gamma_{l_{m}^{n}}^{n}\}_{m}$
stabilizes the ball of radius $n$ of the Cayley graph of $F_{q}$
$m(n)$-eventually, 
\item and so that after all, the sequence $\{\Gamma_{l_{m(n)}^{n}}^{n}\}_{n}$
is an ascending sequence of groups in the model. 
\end{enumerate}
Of course, the sequence of integers $\{m(n)\}_{n}$ can further be
chosen to be strictly increasing, and thus, to simplify the notation
and write

\[
\Gamma_{l_{m(n)}}=\Gamma_{l_{m(n)}^{n}}^{n}
\]
without ambiguity.

We define the sequence of homomorphisms $\{h_{n}:F_{q}\rightarrow\Gamma_{l_{m(n)}}\}_{n}$
by setting $h_{n}=v_{m(n)}^{n}$ for all $n$.

Now we prove that the constructed sequence $\{h_{n}:F_{q}\rightarrow\Gamma_{l_{m(n)}}\}_{n}$
over the ascending sequence $\{\Gamma_{l_{m(n)}}\}_{n}$ of groups
in the model, contains a non-lift subsequence. Indeed, assume the
converse.

We denote by $\Sigma_{\infty}$ the union of all the subsets $\Sigma_{n}$
\[
\Sigma_{\infty}=\underset{n}{\cup}\Sigma_{n}\,.
\]

By Guba's theorem, every system over the free group $F_{k}$, is equivalent
to a finite subsystem. In particular, the system $\Sigma_{\infty}$
(in the variables $x$, where $x$ is the fixed basis of $F_{q}$),
is equivalent to a finite subsystem $\Sigma_{r_{0}}$ over $F_{k}$.
Note that, by the construction of $\{\Sigma_{i}\}_{i}$ and $\{h_{n}=v_{m(n)}^{n}\}_{n}$,
for every $i$, we have that $h_{n}(\Sigma_{i})=1$ eventually. In
particular, we have that $h_{n}(\Sigma_{r_{0}})=1$ for all large
enough $n$.

However, we are assuming that the sequence $\{h_{n}:F_{q}\rightarrow\Gamma_{l_{m(n)}}\}_{n}$
contains no non-lift subsequence, and hence, for all large enough
$n$, the homomorphism $h_{n}=v_{m(n)}^{n}$ can be $\Sigma_{r_{0}}$-lifted.
Since the system $\Sigma_{r_{0}}$ is equivalent to $\Sigma_{\infty}$
over $F_{k}$, we deduce that for all large enough $n$, the homomorphism
$h_{n}=v_{m(n)}^{n}$ can be $\Sigma_{\infty}$-lifted, and in particular
$\Sigma_{n}$-lifted. But this contradicts the defining property of
the homomorphism $v_{m(n)}^{n}$ as one that cannot be $\Sigma_{n}$-lifted.

We deduce that the sequence $\{h_{n}:F_{q}\rightarrow\Gamma_{l_{m(n)}}\}_{n}$
contains a non-lift subsequence, and in accordance with \thmref{50},
by passing to a subsequence, we can continue the proof under the assumption
that the sequence $\{h_{n}:F_{q}\rightarrow\Gamma_{l_{m(n)}}\}_{n}$
is a non-lift sequence that converges into a non-lift $lln^{2}$-limit
group that we denote by

\[
\eta_{\infty}:F_{q}\rightarrow R_{\infty}\,.
\]

By construction, the group $R_{\infty}$ is a proper quotient of $R_{n}$,
for all $n$.

According to \corref{59}, the non-lift $lln^{2}$-limit group $R_{\infty}$,
admits an infinite decreasing sequence of non-lift $lln^{2}$-limit
quotients

\[
R_{\infty}>L_{1}>L_{2}>...\,.
\]
In particular, the group $L_{1}$ belongs to the collection $\mathcal{A}_{R_{n}}$
(defined in the beginning of the proof), for all $n$. Let $w\in F_{q}$
be a word representing a non-trivial element in $R_{\infty}$, but
$w$ represents the trivial element in $L_{1}$. Denote by $n_{0}$
the length of $w$ in $F_{q}$. Then, the canonical map $\eta_{L_{1}}:F_{q}\rightarrow L_{1}$
maps to the identity strictly more elements of the ball of radius
$n_{0}$ in the Cayley graph of $F_{q}$, than the map $\eta_{R_{n_{0}}}:F_{q}\rightarrow R_{n_{0}}$.
This contradicts the defining property of $R_{n_{0}}$. 
\end{proof}
\begin{cor}
\label{cor:61} There exist no non-lift $lln^{2}$-limit groups. 
\end{cor}

\begin{proof}
According to \corref{59} every non-lift $lln^{2}$-limit group admits
an infinite strictly decreasing sequence of non-lift $lln^{2}$-limit
quotients. And according to \thmref{60}, every such sequence must
terminate after finitely many steps. Hence there cannot exist a non-lift
$lln^{2}$-limit group. 
\end{proof}

\subsection{Restricted Limit Groups over Ascending Sequences of Random Groups}

Since we want to be able to treat sentences that contain constants,
we should be able to lift solutions for systems of equations with
constants. Recall that the constants are the words of the fixed free
group $F_{k}$ with the fixed basis $a=(a_{1},...,a_{k})$. 
\begin{defn}
Let $F_{q}$ be a free group, with a basis $x=(x_{1},...,x_{q})$.
Let $\{h_{l_{n}}:F_{q}\ast F_{k}\rightarrow\Gamma_{l_{n}}\}$ be a
sequence of homomorphisms over an ascending sequence of groups in
the model. Let $\pi:F_{k}\rightarrow\Gamma$ be a given presentation
of a group $\Gamma$, and let $h:F_{q}\ast F_{k}\rightarrow\Gamma$
be a homomorphism. 
\begin{enumerate}
\item We say that the homomorphism $h:F_{q}\ast F_{k}\rightarrow\Gamma$
is a \emph{restricted homomorphism}, if $h(a_{i})=a_{i}$ (the right
hand side means precisely $\pi(a_{i})$) for all $i=1,...,k$. 
\item We say that the sequence $\{h_{l_{n}}:F_{q}\ast F_{k}\rightarrow\Gamma_{l_{n}}\}$
is a \emph{restricted sequence}, If for all $n$, the homomorphism
$h_{l_{n}}$ is restricted. 
\end{enumerate}
\end{defn}

\begin{thm}
\label{thm:62} Let $F_{q}$ be a free group, with a basis $x=(x_{1},...,x_{q})$.
Let $\{h_{l_{n}}:F_{q}\ast F_{k}\rightarrow\Gamma_{l_{n}}\}$ be a
restricted sequence of homomorphisms over an ascending sequence of
groups in the model.

Assume that the sequence of stretching factors

\[
\mu_{n}=\underset{1\le u\leq q}{\max}d_{\Gamma_{l_{n}}}(1,h_{l_{n}}(x_{u}))\,,
\]
(note that here, due to the existence of constants, we do not post-compose
with inner automorphisms) satisfies that

\[
\mu_{n}\geq l_{n}\ln^{2}l_{n}
\]
for all $n$.

Then, the sequence $h_{l_{n}}$, subconverges (in the Gromov topology)
to an isometric non-trivial action of $F_{q}\ast F_{k}$ on a pointed
real tree $(Y,y_{0})$, in which the subgroup $F_{k}$ is elliptic.
Moreover, the real tree is not a real line. 
\end{thm}

\begin{proof}
The same proof of \thmref{50}. If $Y$ was a real line, then the
generators of $F_{k}$ represent commuting elements in the groups
$\Gamma_{l_{n}}$ for infinitely many integers $n$, which cannot
happen according to \thmref{26}. 
\end{proof}
\begin{defn}
Let $F_{q}$ be a free group, with a basis $x=(x_{1},...,x_{q})$.
Let $\{h_{l_{n}}:F_{q}\ast F_{k}\rightarrow\Gamma_{l_{n}}\}$ be a
restricted sequence of homomorphisms over an ascending sequence of
groups in the model.

Assume that there exists a positive number $\epsilon>0$, so that
the sequence of stretching factors

\[
\mu_{n}=\underset{1\le u\leq q}{\max}d_{\Gamma_{l_{n}}}(1,h_{l_{n}}(x_{u}))\,,
\]
satisfies that 
\[
\mu_{n}\geq\epsilon l_{n}\ln^{2}l_{n}
\]
for all $n$.

The sequence $h_{l_{n}}$, is called a \emph{restricted $lln^{2}$-sequence
over the ascending sequence $\Gamma_{l_{n}}$ of groups in the model},
or briefly, a restricted $lln^{2}$-sequence.

In case the sequence $h_{l_{n}}$ is a restricted $lln^{2}$-sequence
that converges (in the Gromov topology) to an isometric non-trivial
action of $F_{q}\ast F_{k}$ on a real tree, then, in accordance with
the notation of \thmref{51}, the quotient group

\[
L_{h_{l_{n}},\infty}=F_{q}\ast F_{k}/K_{h_{l_{n}},\infty}
\]
is called a \emph{restricted $lln^{2}$-limit group over the ascending
sequence $\Gamma_{l_{n}}$ of groups in the model}, or briefly, a
restricted $lln^{2}$-limit group.

We say also that \emph{the restricted $lln^{2}$-sequence $h_{l_{n}}$
converges into the restricted $lln^{2}$-limit group $L_{h_{l_{n}},\infty}$}. 
\end{defn}

\begin{rem}
Let $\eta:F_{q}\ast F_{k}\rightarrow L_{\infty}$ be a restricted
$lln^{2}$-limit group (with the canonical generating set associated
to it). Then, the canonical map $\eta$ maps $F_{k}$ isomorphically
into $L_{\infty}$. This follows by \thmref{26} and the appropriate
version of \thmref{51} for the restricted case (see the discussion
below). 
\end{rem}

\begin{defn}
Let $F_{q}$ be a free group, with a fixed basis $x=(x_{1},...,x_{q})$,
and let $\{h_{l_{n}}:F_{q}\ast F_{k}\rightarrow\Gamma_{l_{n}}\}$
be a restricted sequence of homomorphisms over an ascending sequence
of groups in the model. Let $\pi:F_{k}\rightarrow\Gamma$ be a given
presentation of a group $\Gamma$, and let $h:F_{q}\ast F_{k}\rightarrow\Gamma$
be a restricted homomorphism. 
\begin{enumerate}
\item Let $\Sigma$ be a collection of words in the free group $F_{q}\ast F_{k}$.
We say that the homomorphism $h$ \emph{can be $\Sigma$-lifted restrictly}
, if there exists some restricted homomorphism $\tilde{h}:F_{q}\ast F_{k}\rightarrow F_{k}$
with $\tilde{h}(\Sigma)=1$ and $h=\pi\circ\tilde{h}$. 
\item We say that the restricted sequence $\{h_{l_{n}}:F_{q}\ast F_{k}\rightarrow\Gamma_{l_{n}}\}$
is a \emph{restricted non-lift sequence}, if there exists a finite
set of words $\Sigma\subset F_{q}\ast F_{k}$, so that $h_{l_{n}}(\Sigma)=1$
eventually, but for all $n$, the homomorphism $h_{l_{n}}$ cannot
be $\Sigma$-lifted restrictly. 
\item If the sequence $\{h_{l_{n}}:F_{q}\ast F_{k}\rightarrow\Gamma_{l_{n}}\}$
is a restricted non-lift sequence that converges to a restricted $lln^{2}$-limit
group $L_{\infty}$, then the group $L_{\infty}$ is called a \emph{restricted
non-lift $lln^{2}$-limit group}. 
\end{enumerate}
\end{defn}

\begin{rem}
Since we have dropped the collection $\mathcal{N}$, we can assume
that any restricted non-lift sequence is in necessity a restricted
$lln^{2}$-sequence. 
\end{rem}

Now the formulation of \thmref{51} for the restricted case (w.r.t.
to the notation of \thmref{62} instead of that of \thmref{50}) can
be proved by the same argument in \thmref{51}. \Defref{53} stays
unchanged for the restricted case. Point 1 of \lemref{54} is replaced
by the fact that the limit tree is not a real line in the restricted
case. The abelian decompositions that we are interested in, for a
restricted non-lift $lln^{2}$-limit group $L_{\infty}$, are only
those for which the subgroup $F_{k}\leq L_{\infty}$ is elliptic in
them (i.e., can be conjugated into a vertex group). Hence, Point 2
of \lemref{54} naturally obtains a formulation that is consistent
with this. \Lemref{55} stays unchanged for the restricted case.

Let $L_{\infty}$ be a restricted non-lift $lln^{2}$-limit group,
and assume that $L_{\infty}$ is freely indecomposable w.r.t. the
subgroup $F_{k}$.

For defining the \emph{restricted abelian JSJ decomposition} associated
to the restricted non-lift $lln^{2}$-limit group $L_{\infty}$, we
recall again that the abelian decompositions of $L_{\infty}$ that
we are interested in, are only those for which the subgroup $F_{k}\leq L_{\infty}$
is elliptic in them.

In light of the versions of \lemref{55}, \thmref{50}, and \thmref{51}
for the restricted case, Chapter 2 of \cite{DGI} naturally generalizes
(in the same way as it generalizes for restricted limit groups in
the standard context of \cite{DGI}), in order to construct \emph{the
restricted abelian JSJ decomposition of the restricted freely indecomposable
(w.r.t. $F_{k}$) non-lift $lln^{2}$-limit group $L_{\infty}$}.

In particular, the subgroup $F_{k}\leq L_{\infty}$ is elliptic in
the restricted JSJ of $L_{\infty}$. For the exact statement on how
the restricted (abelian) JSJ decomposition encodes all the abelian
splittings of $L_{\infty}$ in which $F_{k}$ is elliptic, see Theorem
8.1 (the restricted case) of \cite{DGI}.

\Defref{56} together with \lemref{57} and \lemref{58}, obtain
a natural form for the restricted case, without changing the arguments.
Thus, \corref{59} obtains also a natural form with the same argument
for the restricted case.

\Thmref{60} stays unchanged for the restricted case.

And hence \corref{61} obtains the following form: 
\begin{thm}
\label{thm:63} There exist no restricted non-lift $lln^{2}$-limit
groups. 
\end{thm}

\section{Consequences }\label{sec:Consequences}

\subsection{Lifting All the Solutions of a Given System over Random Groups }

Let $k\geq2$ be an integer, let $a=(a_{1},...,a_{k})$ be a given
tuple of formal letters, and let $F_{k}$ be the free group with basis
$a$. Let $d<\frac{1}{2}$ be a real number, and consider the Gromov
model of density $d$. 
\begin{thm}
\label{thm:64} Let $\Sigma_{0}=\Sigma_{0}(y,a)$ be a system of equations
in the variables $y$ and constants (maybe empty) $a$. Then, the
random group $\Gamma$ of density $d$ satisfies the following property
with overwhelming probability.

Every solution of the system $\Sigma_{0}$ in $\Gamma$, can be lifted
to a solution of $\Sigma_{0}$ in $F_{k}$. 
\end{thm}

\begin{proof}
Recall the collection $\mathcal{N}_{l}$ of all groups in level $l$
that do not satisfy the $\Sigma$-$lln^{2}$ l.p. (see \defref{36})
for some system in the collection $H_{l}$, where $H_{l}$ is defined
in \thmref{37}. We set $\mathcal{N}=\underset{l}{\cup}\mathcal{N}_{l}$.

We denote the collection of all the groups in level $l$ of the model
of density $d$ by $\mathcal{M}_{l}$.

According to \thmref{63}, there exists an integer $l_{0}$, so that
for all integers $l\geq l_{0}$, every group $\Gamma_{l}$ in $\mathcal{M}_{l}\backslash\mathcal{N}_{l}$
satisfies the $\Sigma_{0}$-l.p., i.e., every solution of $\Sigma_{0}$
in $\Gamma_{l}$ can be lifted to a solution of $\Sigma_{0}$ in $F_{k}$.

Finally, recall by \thmref{37} that the collection $\mathcal{N}=\underset{l}{\cup}\mathcal{N}_{l}$
is negligible, i.e.,

\[
\frac{|\mathcal{N}_{l}|}{|\mathcal{M}_{l}|}\overset{l\rightarrow\infty}{\longrightarrow}0\,,
\]
and the claim follows. 
\end{proof}
\begin{thm}
\label{thm:65} Let $\Sigma_{0}=\Sigma_{0}(y,a)$ be a system of equations
in the variables $y$ and constants (maybe empty) $a$. Let $MR$
be the Makanin-Razborov diagram of the system $\Sigma_{0}$ constructed
over the free group $F_{k}$. Then, the random group $\Gamma$ of
density $d$ satisfies the following property with overwhelming probability.

Every solution of the system $\Sigma_{0}$ in $\Gamma$, factors through
$MR$. 
\end{thm}

\begin{proof}
Every solution of $\Sigma_{0}$ in $\Gamma$ can be lifted to a solution
of $\Sigma_{0}$ in $F_{k}$. By the construction of $MR$ (see {[}\cite{DGI},
sections 5 and 8{]}), every solution of $\Sigma_{0}$ in $F_{k}$
factors through $MR$. Hence, every solution of $\Sigma_{0}$ in $\Gamma$
factors through $MR$. 
\end{proof}

\subsection{Universal Sentence over Random Groups }

Let $k\geq2$ be an integer, let $a=(a_{1},...,a_{k})$ be a given
tuple of formal letters, and let $F_{k}$ be the free group with basis
$a$. Let $d<\frac{1}{2}$ be a real number, and consider the Gromov
model of density $d$. 
\begin{thm}
\label{thm:66} Let $\psi=\psi(x,a)$ be a sentence in the Boolean
algebra of one-quantifier sentences (which may or may not contain
constants). Let $\Gamma$ be the random group of density $d$.

Then, the sentence $\psi$ is a truth sentence over the free group
$F_{k}$, if and only if, with overwhelming probability $\psi$ is
a truth sentence over $\Gamma$. 
\end{thm}

\begin{proof}
Assume that $\psi$ is an existential sentence. Then $\psi$ is equivalent
(tautologically - over any group) to a sentence of the form

\[
\stackrel[i=1]{r}{\vee}\left(\exists x\quad\Sigma_{i}(x,a)=1\wedge v_{1}^{i}(x,a)\neq1\wedge...\wedge v_{n_{i}}^{i}(x,a)\neq1\right)\,,
\]
where $r,n_{i}$ are integers, $\Sigma_{i}(x,a)=1$ is a system of
equations, and $v_{j}^{i}(x,a)$ is a word in the free group $F(x,a)$,
$i=1,...,r$, $j=1,...,n_{i}$.

We denote by $p_{l}$ the probability that the random group of level
$l$ and density $d$ satisfies the sentence $\psi$.

Assume that $\psi$ is true over $F_{k}$. Then, according to \thmref{26},
it follows directly that with overwhelming probability, the random
group $\Gamma$ of density $d$ satisfies $\psi$ 
\[
p_{l}\overset{l\rightarrow\infty}{\longrightarrow}1\,.
\]

Now assume that the sequence $\{p_{l}\}_{l}$ does not converge to
$0$ as $l$ goes to $\infty$. We will show that this assumption
implies that the free group $F_{k}$ satisfies $\psi$. Indeed, according
to \thmref{64}, we can assume that every solution for each of the
systems $\Sigma_{1},...,\Sigma_{r}$ over any of the groups in the
model can be lifted to a solution for that system over $F_{k}$. Since
we are assuming that $p_{l}\not\longrightarrow0$, we deduce that
for some large enough $l$, there exists a group $\Gamma_{l}$ in
the level $l$ of the model, so that $\psi$ is a truth sentence over
$\Gamma_{l}$, and every solution for each of the systems $\Sigma_{1},...,\Sigma_{r}$
over $\Gamma_{l}$ can be lifted to a solution for that system over
$F_{k}$.

Since $\psi$ is true over $\Gamma_{l}$, there exists $i_{0}=1,...,r$,
and some solution $x_{0}$ for the system $\Sigma_{i_{0}}(x,a)$ over
$\Gamma_{l}$, so that

\[
v_{1}^{i_{0}}(x_{0},a)\neq1\wedge...\wedge v_{n_{i_{0}}}^{i_{0}}(x_{0},a)\neq1\,
\]
in $\Gamma_{l}$.

Let $\tilde{x}_{0}$ be a lift for the solution $x_{0}$ to a solution
for the system $\Sigma_{i_{0}}(x,a)$ over the free group $F_{k}$.
Then

\[
\Sigma_{i_{0}}(\tilde{x}_{0},a)=1\wedge v_{1}^{i_{0}}(\tilde{x}_{0},a)\neq1\wedge...\wedge v_{n_{i_{0}}}^{i_{0}}(\tilde{x}_{0},a)\neq1\,
\]
in $F_{k}$, which implies that $\psi$ is a truth sentence over $F_{k}$.

This argument, obviously, implies the statement in the theorem. 
\end{proof}


\begin{thebibliography}{Br-Ha}
\bibitem[Br-Ha]{Non-Positive=000020Curvature} M. R. Bridson and A.
Haefliger, \emph{``}Metric spaces of non-positive curvature'', volume
319 of \emph{Grundlehren der Mathematischen Wissenschaften {[}Fundamental
Principles of Mathematical Sciences{]}}. Springer-Verlag, Berlin,
1999.

\bibitem[Co]{Remi=000020Coulon=000020-=000020Burnside} Rémi Coulon,
``Small cancellation theory and Burnside problem'', Internat. J.
Algebra Comput. 24 (2014), no. 3, 251-345, arXiv:1302.6933

\bibitem[Gr1]{Gromov=000020-=000020Asymptotic=000020invariants} M.
Gromov, ``Asymptotic invariants of infinite groups'', Geometric
group theory, Vol. 2 (Sussex, 1991), London Math. Soc. Lecture Note
Ser., vol. 182, Cambridge Univ. Press, Cambridge, 1993, pp. 1--295.
MR 1253544

\bibitem[Gr2]{Gromov=000020Hyperbolic=000020groups} M. Gromov. ``Hyperbolic
groups''. In Essays in group theory, volume 8 of Math. Sci. Res.
Inst. Publ., pages 75--263. Springer, New York, 1987.

\bibitem[Gr3]{Daniel=000020Groves=000020-=000020Limit=000020groups=000020relatively=000020hyperbolic}
Daniel Groves, ``Limit groups for relatively hyperbolic groups'',
II: Makanin-Razborov diagrams, arXiv:math/0503045

\bibitem[Kh-My]{Kharlampovich=000020Myasnikov=000020-=000020Elementary=000020theory=000020of=000020free=000020non-abelian=000020groups}
Kharlampovich, Olga; Myasnikov, Alexei (2006). ``Elementary theory
of free non-abelian groups''. Journal of Algebra. 302 (2): 451--552.
https://doi.org/10.1016/j.jalgebra.2006.03.033

\bibitem[Kh-Sk]{Kharlampovich=000020Sklinos=000020-=000020FIRST-ORDER=000020SENTENCES}
O. Kharlampovich, R. Sklinos, ``FIRST-ORDER SENTENCES IN RANDOM GROUPS
I: UNIVERSAL SENTENCES'', arXiv:2106.05461

\bibitem[Ol1]{Ollivier=000020-=000020Sharp=000020phase} Y. Ollivier,
``Sharp phase transition theorems for hyperbolicity of random groups'',
arXiv:math/0301187

\bibitem[Ol2]{Ollivier=000020-=000020Some=000020small=000020cancellation}
Y. Ollivier, ``Some small cancellation properties of random groups'',
arXiv:math/0409226

\bibitem[Ri-Se]{Structure=000020and=000020rigidity=000020in=000020hyperbolic=000020groups=000020I}
Rips, E., Sela, Z. ``Structure and rigidity in hyperbolic groups
I''. Geometric and Functional Analysis 4, 337--371 (1994). https://doi.org/10.1007/BF01896245

\bibitem[Se1]{DGI} Z. Sela, ``Diophantine geometry over groups I:
Makanin-Razborov diagrams'', Publications Mathematiques de l'IHES
93(2001), 31-105.

\bibitem[Se2]{DGII} Z. Sela, ``Diophantine geometry over groups
II: Completions, closures and formal solutions'', Israel Jour. of
Math. 134(2003), 173-254.

\bibitem[Se3]{DGIII} Z. Sela, ``Diophantine geometry over groups
III: Rigid and solid solutions'', Israel Jour. of Math. 147(2005),
1-73.

\bibitem[Se4]{DGIV} Z. Sela, ``Diophantine geometry over groups
IV: An iterative procedure for validation of a sentence'', Israel
Jour. of Math. 143(2004), 1-130.

\bibitem[Se5]{DGV} Z. Sela, ``Diophantine geometry over groups V:
Quantifier elimination'', divided into two parts - Israel Jour. of
Math. 150(2005), 1-197, and GAFA 16(2006), 537-706.

\bibitem[Se6]{DGV2} Z. Sela, ``Diophantine geometry over groups
V2: quantifier elimination II.'' GAFA, Geom. funct. anal. 16, 537--706
(2006).

\bibitem[Se7]{DGVI} Z. Sela, ``Diophantine geometry over groups
VI: the elementary theory of a free group.'' GAFA, Geom. funct. anal.
16, 707--730 (2006).

\bibitem[Se-Ja]{Zlil=000020and=000020Jaligot=000020-=000020free=000020products}
Eric Jaligot. Zlil Sela. ``Makanin--Razborov diagrams over .''
Illinois J. Math. 54 (1) 19 - 68, Spring 2010. https://doi.org/10.1215/ijm/1299679737

\bibitem[Xi]{Xiangdong=000020Xie=000020-=000020Growth=000020of=000020relatively=000020hyperbolic}
Xiangdong Xie, ``Growth of relatively hyperbolic groups'', arXiv:math/0504191 

\end{thebibliography}
\end{document}